\documentclass[11pt, fleqn]{article}

\usepackage{amssymb, amsthm, amsmath}

\numberwithin{equation}{section}

\theoremstyle{plain}
\newtheorem{theorem}{Theorem}[section]
\newtheorem{lemma}{Lemma}[section]
\newtheorem{corollary}{Corollary}[section]

\theoremstyle{definition}
\newtheorem{definition}{Definition}[section]

\theoremstyle{remark}
\newtheorem{remark}{Remark}[section]

\title{The initial value problem for motion of micropolar fluids with heat conduction in Banach spaces}

\author{Ry\^{o}hei Kakizawa\thanks{Graduate School of Mathematical Sciences, The University of Tokyo, 3-8-1 Komaba Meguro-ku Tokyo 153-8914, Japan (\textit{E-mail address:} kakizawa@ms.u-tokyo.ac.jp)}}

\date{}

\begin{document}

\maketitle

\begin{abstract}
We consider the abstract initial value problem for the system of evolution equations which describe motion of micropolar fluids with heat conduction in a bounded domain.
This problem has uniquely a mild solution locally in time for general initial data, and globally in time for small initial data.
Moreover, a mild solution of this problem can be a strong or classical solution under appropriate assumptions for initial data.
We prove the above properties by the theory of analytic semigroups on Banach spaces.
\end{abstract}


\section{Introduction}
Let $\Omega$ be a bounded domain in $\mathbb{R}^{3}$ with its $C^{2,1}$-boundary $\partial\Omega$.
Motion of micropolar fluids with heat conduction in $\Omega$ is described by the system of eight equations as follows:
\begin{equation}
\begin{split}
&\mathrm{div}u=0 & \mathrm{in} \ \Omega\times(0,T), \\
&\rho\{\partial_{t}+(u\cdot\nabla)\}u=2\mu_{r}\mathrm{rot}\omega+\rho f(\theta)-\nabla p+(\mu+\mu_{r})\Delta u & \mathrm{in} \ \Omega\times(0,T), \\
&\rho\{\partial_{t}+(u\cdot\nabla)\}\omega=-4\mu_{r}\omega+2\mu_{r}\mathrm{rot}u+\rho g(\theta)+(c_{a}+c_{d})\Delta\omega+(c_{0}+c_{d}-c_{a})\nabla\mathrm{div}\omega & \mathrm{in} \ \Omega\times(0,T), \\
&\rho c_{v}\{\partial_{t}+(u\cdot\nabla)\}\theta=\Phi(u;\omega)+\kappa\Delta\theta & \mathrm{in} \ \Omega\times(0,T),
\end{split}
\end{equation}
where $u=(u_{1},u_{2},u_{3})$ is the fluid velocity, $\omega=(\omega_{1},\omega_{2},\omega_{3})$ is the angular velocity, $p$ is the pressure, $\theta$ is the absolute temperature, $\rho$ is the density, $\mu$ is the coefficient of viscosity, $\mu_{r}$ is the coefficient of microrotation viscosity, $c_{0}$, $c_{a}$ and $c_{d}$ are coefficients of angular viscosity, $\kappa$ is the coefficient of heat conductivity, $c_{v}$ is the specific heat at constant volume, $f=(f_{1},f_{2},f_{3})$ and $g=(g_{1},g_{2},g_{3})$ are external force fields affected by $\theta$, $\Phi(u;\omega)$ is the viscous dissipation function defined as
\begin{equation*}
\Phi(u;\omega)=\Phi(u,u;\omega,\omega), \ \Phi(u,v;\omega,\psi)=\Phi_{1}(u,v)+\Phi_{2}(u,v;\omega,\psi)+\Phi_{3}(\omega,\psi)+\Phi_{4}(\omega,\psi)+\Phi_{5}(\omega,\psi),
\end{equation*}
\begin{equation*}
\Phi_{1}(u,v)=2\mu D(u):D(v), \ D(u)=\frac{1}{2}(\nabla u+(\nabla u)^{T}), \ \Phi_{2}(u,v;\omega,\psi)=4\mu_{r}\left(\frac{1}{2}\mathrm{rot}u-\omega\right)\left(\frac{1}{2}\mathrm{rot}v-\psi\right),
\end{equation*}
\begin{equation*}
\Phi_{3}(\omega,\psi)=c_{0}(\mathrm{div}\omega)(\mathrm{div}\psi), \ \Phi_{4}(\omega,\psi)=(c_{a}+c_{d})\nabla\omega:\nabla\psi, \ \Phi_{5}(\omega,\psi)=(c_{d}-c_{a})\nabla\omega:(\nabla\psi)^{T},
\end{equation*}
$(\nabla u)^{T}$ and $(\nabla\psi)^{T}$ are transposed matrices of $\nabla u$ and $\nabla\psi$ respectively.
These equations correspond to the law of conservation of mass, momentum, angular momentum and energy respectively.
Moreover, it is required that $\rho$, $\mu$, $\mu_{r}$, $c_{0}$, $c_{a}$, $c_{d}$, $\kappa$ and $c_{v}$ are positive constants, $c_{0}+c_{d}>c_{a}$.
See, for example, \cite{Lamb, Lukaszewicz 2, Serrin} on conservation laws of fluid motion and the derivation of the above equations.

Micropolar fluids belong to a kind of viscous fluids with the asymmetric stress tensor.
The law of conservation of angular momentum must be taken into account due to asymmetry of the stress tensor.
It is quite natural from the continuum mechanical point of view that (1.1) can be regarded as a generalization of the Navier-Stokes equations of heat-conductive fluids.
If $\omega$, $g$, $\mu_{r}$, $c_{0}$, $c_{a}$ are $c_{d}$ are formally taken as zeros, then the Navier-Stokes equations of heat-conductive fluids are deduced from (1.1).
We can utilize micropolar fluid mechanics to consider physical phenomena with the micro structure.
More precisely, micropolar fluid mechanics is valid for motion of viscous fluids consisting of rigid and randomly oriented (or spherical) particles in the case where deformation of particles can be neglected.
Some problems related to (1.1) have been studied in recent years.
{\L}ukaszewicz \cite{Lukaszewicz 1} treated the initial-boundary value problem for the Navier-Stokes equations of micropolar fluids in $L^{2}_{\sigma}(\Omega)\times (L^{2}(\Omega))^{3}$, where $L^{p}_{\sigma}(\Omega)$ $(1<p<\infty)$ is the closed subspace of $(L^{p}(\Omega))^{3}$ defined as in section 2.
Moreover, Yamaguchi \cite{Yamaguchi} considered the existence and uniqueness of global solutions of the initial-boundary value problem for the Navier-Stokes equations of micropolar fluids in $L^{3}_{\sigma}(\Omega)\times (L^{3}(\Omega))^{3}$.
Kagei and Skowron \cite{Kagei} discussed the existence and uniqueness of solutions of the initial-boundary value problem for (1.1) in $L^{2}_{\sigma}(\Omega)\times (L^{2}(\Omega))^{3}\times L^{2}(\Omega)$.

The main purpose of this paper is to discuss the existence, uniqueness and regularity of solutions of the initial-boundary value problem for (1.1) with the following initial-boundary data:
\begin{equation}
\begin{split}
&u|_{t=0}=u_{0} & \mathrm{in} \ \Omega, \\
&u|_{\partial\Omega}=0 & \mathrm{on} \ \partial\Omega\times(0,T), \\
&\omega|_{t=0}=\omega_{0} & \mathrm{in} \ \Omega, \\
&\omega|_{\partial\Omega}=0 & \mathrm{on} \ \partial\Omega\times(0,T), \\
&\theta|_{t=0}=\theta_{0} & \mathrm{in} \ \Omega, \\
&\theta|_{\partial\Omega}=\theta_{s} & \mathrm{on} \ \partial\Omega\times(0,T),
\end{split}
\end{equation}
where $\theta_{s}$ is the surface temperature on $\partial\Omega$ assumed to be a nonnegative constant.
As is mentioned in \cite{Fujita, Giga 3, Hishida, Kakizawa}, the abstract initial value problem for (1.1), (1.2) in Banach spaces is a strong method of analyzing (1.1), (1.2) with initial data $(u_{0},\omega_{0},\theta_{0})$ in $L^{p}_{\sigma}(\Omega)\times (L^{q}(\Omega))^{3}\times L^{r}(\Omega)$ $(1<p<\infty, \ 1<q<\infty, \ 1<r<\infty)$.
It is explained in section 2 that we can transform (1.1), (1.2) into the following abstract initial value problem for the system of evolution equations:
\begin{equation}\tag{I}
\begin{split}
&d_{t}u+A_{p}u=F(u,\omega,\theta) & \mathrm{in} \ (0,T), \\
&d_{t}\omega+\Gamma_{q}\omega=G(u,\omega,\theta) & \mathrm{in} \ (0,T), \\
&d_{t}\theta+B_{r}\theta=H(u,\omega,\theta) & \mathrm{in} \ (0,T), \\
&u(0)=u_{0}, \\
&\omega(0)=\omega_{0}, \\
&\theta(0)=\theta_{0},
\end{split}
\end{equation}
where $A_{p}$, $\Gamma_{q}$ and $B_{r}$ are sectorial operators in $L^{p}_{\sigma}(\Omega)$, $(L^{q}(\Omega))^{3}$ and $L^{r}(\Omega)$ respectively, $F(u,\omega,\theta)$, $G(u,\omega,\theta)$ and $H(u,\omega,\theta)$ are nonlinear terms corresponding to $(1.1)_{2}$, $(1.1)_{3}$ and $(1.1)_{4}$ respectively.
It is well known in \cite[Chapter 3]{Henry}, \cite[Chapter 6]{Pazy} that we can consider not only strong solutions but also mild solutions of (1.1), (1.2).

It is proved in this paper that (1.1), (1.2) has uniquely a mild solution locally in time for general initial data, and globally in time for small initial data.
Moreover, a mild solution of this problem can be a strong or classical solution under appropriate assumptions for initial data.
By following the argument based on \cite{Fujita, Giga 3, Hishida, Kakizawa}, first of all, the existence of local mild solutions follows from the successive approximation method.
Second, not only the existence of global mild solutions but also the asymptotic behavior of global mild solutions are obtained by global a priori estimates for mild solutions of (1.1), (1.2).

This paper is organized as follows: In section 2, we define basic notation used in this paper and a strong and mild solution of (1.1), (1.2), and state our main results and some lemmas for them.
We prove the existence and uniqueness of mild solutions of (1.1), (1.2) in section 3.
The regularity of mild solutions of (1.1), (1.2) is discussed in sections 4 and 5.

\section{Preliminaries and main results}
\subsection{Function spaces}
Function spaces and basic notation which we use throughout this paper are introduced as follows: The norm in $L^{s}(\Omega)$ $(1\leq s\leq\infty)$ and the norm in $W^{k,s}(\Omega)$ (the Sobolev space, $k \in \mathbb{Z}$, $k\geq 0$) are denoted by $\|\cdot\|_{s}$ and $\|\cdot\|_{k,s}$ respectively, $W^{0,s}(\Omega)=L^{s}(\Omega)$, $\|\cdot\|_{0,s}=\|\cdot\|_{s}$.
$C^{\infty}_{0}(\Omega)$ is the set of all functions which are infinitely differentiable and have compact support in $\Omega$.
$W^{k,s}_{0}(\Omega)$ is the completion of $C^{\infty}_{0}(\Omega)$ in $W^{k,s}(\Omega)$.
Let us introduce solenoidal function spaces.
$C^{\infty}_{0,\sigma}(\Omega):=\{u \in (C^{\infty}_{0}(\Omega))^{3} \ ; \ \mathrm{div}u=0\}$.
$L^{p}_{\sigma}(\Omega)$ $(1<p<\infty)$ is the completion of $C^{\infty}_{0,\sigma}(\Omega)$ in $(L^{p}(\Omega))^{3}$.
It follows from \cite[Theorem 2]{Fujiwara} that $(L^{p}(\Omega))^{3}$ is decomposed into $(L^{p}(\Omega))^{3}=L^{p}_{\sigma}(\Omega)\oplus L^{p}_{\pi}(\Omega)$, where $L^{p}_{\pi}(\Omega)=\{\nabla p \ ; \ p \in W^{1,p}(\Omega)\}$.
Let $P_{p}$ be the projection of $(L^{p}(\Omega))^{3}$ onto $L^{p}_{\sigma}(\Omega)$.
$C^{k,\delta}(\overline{\Omega})$ $(0<\delta\leq 1)$ is the H\"{o}lder space defined as in \cite[1.26--1.29]{Adams}, $C^{k,0}(\overline{\Omega})=C^{k}(\overline{\Omega})$, $C^{0}(\overline{\Omega})=C(\overline{\Omega})$.

Let $I$ be an interval in $\mathbb{R}$, $X$ be a Banach space.
$C(I;X)$ is the set of all $X$-valued functions which are continuous in $I$.
$C^{k}(I;X)$ $(k \in \mathbb{Z}, \ k\geq 0)$ is the set of all $X$-valued function which are continuously differentiable up to the order $k$ in $I$, $C^{0}(I;X)=C(I;X)$.
In the case where $I$ is a bounded closed interval in $\mathbb{R}$, $C^{0,\delta}(I;X)$ $(0<\delta\leq 1)$ is the set of all $X$-valued function which are uniformly H\"{o}lder continuous with the exponent $\delta$ on $I$.
If $I$ is not bounded or closed, $u \in C^{0,\delta}(I;X)$ means that $u \in C^{0,\delta}(I_{1};X)$ for any bounded closed interval $I_{1}$ contained in $I$.
$C^{k,\delta}(I;X)$ is the set of all $X$-valued functions $u$ which $u \in C^{k}(I;X)$ and $d^{k}_{t}u \in C^{0,\delta}(I;X)$, $C^{k,0}(I;X)=C^{k}(I;X)$.

$C_{b}(\mathbb{R};\mathbb{R}^{3})$ is the set of all $\mathbb{R}^{3}$-valued functions which are bounded continuous in $\mathbb{R}$.
$C^{k}(\mathbb{R};\mathbb{R}^{3})$ $(k \in \mathbb{Z}, \ k\geq 0)$ is the set of all $\mathbb{R}^{3}$-valued functions which are continuously differentiable up to the order $k$ in $\mathbb{R}$, $C^{0}(\mathbb{R};\mathbb{R}^{3})=C(\mathbb{R};\mathbb{R}^{3})$.
$C^{0,1}(\mathbb{R};\mathbb{R}^{3})$ is the set of all $\mathbb{R}^{3}$-valued functions which are uniformly Lipschitz continuous in $\mathbb{R}$.

\subsection{Strongly elliptic operators}
For the sake of simplicity, we assume that $\rho=1$, $\mu+\mu_{r}=1$, $c_{0}=1/2$, $c_{a}=1/4$, $c_{d}=3/4$, $\kappa=1$, $c_{v}=1$ and $\theta_{s}=0$ throughout this paper.
Let us introduce three linear operators $A_{p}$ $(1<p<\infty)$, $\Gamma_{q}$ $(1<q<\infty)$ and $B_{r}$ $(1<r<\infty)$ which appeared in (I).
$\Gamma_{q}$ is the strongly elliptic operator in $(L^{q}(\Omega))^{3}$ with the zero Dirichlet boundary condition defined as $\Gamma_{q}=-\Delta-\nabla\mathrm{div}$, $D(\Gamma_{q})=(W^{2,q}(\Omega))^{3}\cap(W^{1,q}_{0}(\Omega))^{3}$, where $D(\Gamma_{q})$ is the domain of $\Gamma_{q}$.
$B_{r}$ is the Laplace operator in $L^{r}(\Omega)$ with the same boundary condition as above defined as $B_{r}=-\Delta$, $D(B_{r})=W^{2,r}(\Omega)\cap W^{1,r}_{0}(\Omega)$.
We introduce the Stokes operator $A_{p}$ in $L^{p}_{\sigma}(\Omega)$ by $A_{p}=-P_{p}\Delta$, $D(A_{p})=(D(B_{p}))^{3}\cap L^{p}_{\sigma}(\Omega)$.
It is well known in \cite[Theorems 2.5.2 and 7.3.6]{Pazy}, \cite[Theorem 1]{Giga 1} that $\Gamma_{q}$, $B_{r}$ and $A_{p}$ are sectorial operators in $(L^{q}(\Omega))^{3}$, $L^{r}(\Omega)$ and $L^{p}_{\sigma}(\Omega)$ respectively.
Therefore, $-\Gamma_{q}$ generates an uniformly bounded analytic semigroup $\{e^{-t\Gamma_{q}}\}_{t\geq 0}$ on $(L^{q}(\Omega))^{3}$, fractional powers $\Gamma^{\beta}_{q}$ of $\Gamma_{q}$ can be defined for any $\beta\geq 0$, $\Gamma^{0}_{q}=I_{q}$, where $I_{q}$ is the identity operator in $(L^{q}(\Omega))^{3}$.
Similarly to $\Gamma_{q}$, an uniformly bounded analytic semigroup $\{e^{-tB_{r}}\}_{t\geq 0}$ $(\{e^{-tA_{p}}\}_{t\geq 0})$ on $L^{r}(\Omega)$ $(L^{p}_{\sigma}(\Omega))$ is generated, fractional powers $B^{\gamma}_{r}$ $(A^{\alpha}_{p})$ of $B_{r}$ $(A_{p})$ are defined for any $\gamma\geq 0$ $(\alpha\geq 0)$.
Moreover, it follows from \cite[Theorem 3]{Giga 2} that $D(A^{\alpha}_{p})$ is characterized as $D(A^{\alpha}_{p})=(D(B^{\alpha}_{p}))^{3}\cap L^{p}_{\sigma}(\Omega)$ for any $0\leq \alpha\leq 1$.
Let us introduce Banach spaces derived from $A^{\alpha}_{p}$, $\Gamma^{\beta}_{q}$ and $B^{\gamma}_{r}$.
$X^{\alpha}_{p}$, $Y^{\beta}_{q}$ and $Z^{\gamma}_{r}$ are defined as $D(A^{\alpha}_{p})$, $D(\Gamma^{\beta}_{q})$ and $D(B^{\gamma}_{r})$ with the norm $\|\cdot\|_{X^{\alpha}_{p}}=\|A^{\alpha}_{p}\cdot\|_{p}$, $\|\cdot\|_{Y^{\beta}_{q}}=\|\Gamma^{\beta}_{q}\cdot\|_{q}$ and $\|\cdot\|_{Z^{\gamma}_{r}}=\|B^{\gamma}_{r}\cdot\|_{r}$ respectively.
$\Lambda_{1}$ is the first eigenvalue of the Laplace operator with the zero Dirichlet boundary condition.

We state some lemmas concerning sectorial operators in Banach spaces.
See, for example, \cite[Chapter 1]{Henry}, \cite[Chapter 2]{Pazy} on the theory of analytic semigroups on Banach spaces and fractional powers of sectorial operators.
\begin{lemma}
Let $1<p<\infty$, $1<q<\infty$, $1<r<\infty$, $\alpha\geq 0$, $\beta\geq 0$, $\gamma\geq 0$, $0<\lambda<\Lambda_{1}$.
Then
\begin{equation}
\|A^{\alpha}_{p}e^{-tA_{p}}u\|_{p}\leq C_{A_{p},\alpha,\lambda}t^{-\alpha}e^{-\lambda t}\|u\|_{p},
\end{equation}
\begin{equation}
\|\Gamma^{\beta}_{q}e^{-t\Gamma_{q}}\omega\|_{q}\leq C_{\Gamma_{q},\beta,\lambda}t^{-\beta}e^{-\lambda t}\|\omega\|_{q},
\end{equation}
\begin{equation}
\|B^{\gamma}_{r}e^{-tB_{r}}\theta\|_{r}\leq C_{B_{r},\gamma,\lambda}t^{-\gamma}e^{-\lambda t}\|\theta\|_{r}
\end{equation}
for any $u \in L^{p}_{\sigma}(\Omega)$, $\omega \in (L^{q}(\Omega))^{3}$, $\theta \in L^{r}(\Omega)$, where $C_{A_{p}, \alpha,\lambda}$, $C_{\Gamma_{q},\beta,\lambda}$ and $C_{B_{r},\gamma,\lambda}$ are positive constants.
\end{lemma}
\begin{proof}
It is \cite[Theorem 1.4.3]{Henry}.
\end{proof}
\begin{lemma}
Let $1<p<\infty$, $1<q<\infty$, $1<r<\infty$, $0<\alpha\leq 1$, $0<\beta\leq 1$, $0<\gamma\leq 1$.
Then
\begin{equation}
\|(e^{-tA_{p}}-I_{p})u\|_{p}\leq C_{A_{p},\alpha}t^{\alpha}\|u\|_{X^{\alpha}_{p}},
\end{equation}
\begin{equation}
\|(e^{-t\Gamma_{q}}-I_{q})\omega\|_{q}\leq C_{\Gamma_{q},\beta}t^{\beta}\|\omega\|_{Y^{\beta}_{q}},
\end{equation}
\begin{equation}
\|(e^{-tB_{r}}-I_{r})\theta\|_{r}\leq C_{B_{r},\gamma}t^{\gamma}\|\theta\|_{Z^{\gamma}_{r}}
\end{equation}
for any $u \in X^{\alpha}_{p}$, $\omega \in Y^{\beta}_{q}$, $\theta \in Z^{\gamma}_{r}$, where $C_{A_{p},\alpha}$, $C_{\Gamma_{q},\beta}$ and $C_{B_{r},\gamma}$ are positive constants.
\end{lemma}
\begin{proof}
It is \cite[Theorem 1.4.3]{Henry}.
\end{proof}
\begin{lemma}
Let $1<p<\infty$, $1<q<\infty$, $1<r<\infty$, $0<\alpha\leq 1$, $0<\beta\leq 1$, $0<\gamma\leq 1$.
Then
\begin{equation}
\|e^{-tA_{p}}u\|_{X^{\alpha}_{p}}=o(t^{-\alpha}) \ \mathrm{as} \ t\rightarrow +0,
\end{equation}
\begin{equation}
\|e^{-t\Gamma_{q}}\omega\|_{Y^{\beta}_{q}}=o(t^{-\beta}) \ \mathrm{as} \ t\rightarrow +0,
\end{equation}
\begin{equation}
\|e^{-tB_{r}}\theta\|_{Z^{\gamma}_{r}}=o(t^{-\gamma}) \ \mathrm{as} \ t\rightarrow +0
\end{equation}
for any $u \in L^{p}_{\sigma}(\Omega)$, $\omega \in (L^{q}(\Omega))^{3}$, $\theta \in L^{r}(\Omega)$.
\end{lemma}
\begin{proof}
It is \cite[Exercise 1.4.10]{Henry}.
\end{proof}
\begin{lemma}
Let $1<p<\infty$, $1<q<\infty$, $1<r<\infty$, $0\leq\alpha\leq 1$, $0\leq\beta\leq 1$, $0\leq\gamma\leq 1$.
Then
\begin{equation}
X^{\alpha}_{p}\hookrightarrow (W^{k,s}(\Omega))^{3} \ \mathrm{if} \ \frac{1}{p}-\frac{2\alpha-k}{3}\leq\frac{1}{s}\leq\frac{1}{p},
\end{equation}
\begin{equation}
Y^{\beta}_{q}\hookrightarrow (W^{k,s}(\Omega))^{3} \ \mathrm{if} \ \frac{1}{q}-\frac{2\beta-k}{3}\leq\frac{1}{s}\leq\frac{1}{q},
\end{equation}
\begin{equation}
Z^{\gamma}_{r}\hookrightarrow W^{k,s}(\Omega) \ \mathrm{if} \ \frac{1}{r}-\frac{2\gamma-k}{3}\leq\frac{1}{s}\leq\frac{1}{r},
\end{equation}
where $\hookrightarrow$ is the continuous inclusion.
\end{lemma}
\begin{proof}
It is \cite[Theorem 1.6.1]{Henry}.
\end{proof}

\subsection{Abstract initial value problem for \rm{(1.1)}, \rm{(1.2)}}
Let $1<p<\infty$, $1<q<\infty$, $1<r<\infty$, $0\leq \alpha_{0}<1$, $0\leq \beta_{0}<1$, $0\leq\gamma_{0}<1$, $u_{0} \in X^{\alpha_{0}}_{p}$, $\omega_{0} \in Y^{\beta_{0}}_{q}$, $\theta_{0} \in Z^{\gamma_{0}}_{r}$.
Then we apply $P_{p}$ to $(1.1)_{2}$, and get the following abstract initial value problem:
\begin{equation}\tag{I}
\begin{split}
&d_{t}u+A_{p}u=F(u,\omega,\theta) & \mathrm{in} \ (0,T], \\
&d_{t}\omega+\Gamma_{q}\omega=G(u,\omega,\theta) & \mathrm{in} \ (0,T], \\
&d_{t}\theta+B_{r}\theta=H(u,\omega,\theta) & \mathrm{in} \ (0,T], \\
&u(0)=u_{0}, \\
&\omega(0)=\omega_{0}, \\
&\theta(0)=\theta_{0},
\end{split}
\end{equation}
where
\begin{equation*}
F(u,\omega,\theta):=-P_{p}(u\cdot\nabla)u+2\mu_{r}P_{p}\mathrm{rot}\omega+P_{p}f(\theta),
\end{equation*}
\begin{equation*}
G(u,\omega,\theta):=-(u\cdot\nabla)\omega-4\mu_{r}\omega+2\mu_{r}\mathrm{rot}u+g(\theta),
\end{equation*}
\begin{equation*}
H(u,\omega,\theta):=-(u\cdot\nabla)\theta+\Phi(u;\omega).
\end{equation*}
In order to solve (I), first of all, we shall find a solution satisfying the following abstract integral equations related to (I):
\begin{equation}\tag{II}
\begin{split}
&u(t)=e^{-tA_{p}}u_{0}+\displaystyle\int^{t}_{0}e^{-(t-s)A_{p}}F(u,\omega,\theta)(s)ds, \\
&\omega(t)=e^{-t\Gamma_{q}}\omega_{0}+\displaystyle\int^{t}_{0}e^{-(t-s)\Gamma_{q}}G(u,\omega,\theta)(s)ds, \\
&\theta(t)=e^{-tB_{r}}\theta_{0}+\displaystyle\int^{t}_{0}e^{-(t-s)B_{r}}H(u,\omega,\theta)(s)ds
\end{split}
\end{equation}
for any $0\leq t\leq T$.
Let us introduce a strong and mild solution of (1.1), (1.2) defined on $[0,T]$.
A strong and mild solution of (1.1), (1.2) defined on $[0,\infty)$ is similarly defined.
\begin{definition}
$(u,\omega,\theta)$ is called a strong solution of (1.1), (1.2) if it satisfies
\begin{equation*}
u \in C([0,T];X^{\alpha_{0}}_{p})\cap C((0,T];X^{1}_{p}), \ d_{t}u \in C((0,T];L^{p}_{\sigma}(\Omega)),
\end{equation*}
\begin{equation*}
\omega \in C([0,T];X^{\beta_{0}}_{q})\cap C((0,T];Y^{1}_{q}), \ d_{t}\omega \in C((0,T];(L^{q}(\Omega))^{3}),
\end{equation*}
\begin{equation*}
\theta \in C([0,T];Z^{\gamma_{0}}_{r})\cap C((0,T];Z^{1}_{r}), \ d_{t}\theta \in C((0,T];L^{r}(\Omega))
\end{equation*}
and (I).
\end{definition}
\begin{definition}
$(u,\omega,\theta)$ is called a mild solution of (1.1), (1.2) if it satisfies
\begin{equation*}
u \in C([0,T];X^{\alpha_{0}}_{p}),
\end{equation*}
\begin{equation*}
\omega \in C([0,T];Y^{\beta_{0}}_{q}),
\end{equation*}
\begin{equation*}
\theta \in C([0,T];Z^{\gamma_{0}}_{r})
\end{equation*}
and (II).
\end{definition}

\subsection{Main results}
We will state our main results in this subsection.
It is essential for our main results to be assumed that $f \in C^{0,1}(\mathbb{R};\mathbb{R}^{3})$ with the Lipschitz constant $L_{f}$, $f(0)=0$, $g \in C^{0,1}(\mathbb{R};\mathbb{R}^{3})$ with the Lipschitz constant $L_{g}$, $g(0)=0$, $p$, $q$, $r$, $\alpha_{0}$, $\beta_{0}$ and $\gamma_{0}$ satisfy the following inequalities:
\begin{equation}
\begin{split}
&1<p<\infty, \ 1<q<\infty, \ 1<r<\infty, \ \left|\frac{1}{p}-\frac{1}{q}\right|<\frac{1}{3}, \\
&\frac{1}{p}-\frac{1}{2r}<\frac{1}{3}, \ \frac{1}{r}-\frac{1}{p}<\frac{2}{3}, \ \frac{1}{q}-\frac{1}{2r}<\frac{1}{3}, \ \frac{1}{r}-\frac{1}{q}<\frac{2}{3},
\end{split}
\end{equation}
\begin{equation}
\begin{split}
&\max\left\{0, \frac{3}{2p}-\frac{1}{2}\right\}\leq\alpha_{0}<1, \ 0\leq\beta_{0}<1, \ 0\leq\gamma_{0}<1, \\
&\alpha_{0}-\beta_{0}-\frac{3}{2}\left(\frac{1}{p}-\frac{1}{q}\right)\leq\frac{1}{2}, \\
&\alpha_{0}-\frac{\gamma_{0}}{2}-\frac{3}{2}\left(\frac{1}{p}-\frac{1}{2r}\right)\geq 0, \ -1<\alpha_{0}-\gamma_{0}-\frac{3}{2}\left(\frac{1}{p}-\frac{1}{r}\right)\leq 1, \\
&\beta_{0}-\frac{\gamma_{0}}{2}-\frac{3}{2}\left(\frac{1}{q}-\frac{1}{2r}\right)\geq 0, \ -1<\beta_{0}-\gamma_{0}-\frac{3}{2}\left(\frac{1}{q}-\frac{1}{r}\right)\leq 1.
\end{split}
\end{equation}
The first purpose of this paper is to discuss the existence and uniqueness of mild solutions of (1.1), (1.2).
We shall prove the following theorems:
\begin{theorem}
Let $f \in C^{0,1}(\mathbb{R};\mathbb{R}^{3})$ with the Lipschitz constant $L_{f}$, $f(0)=0$, $g \in C^{0,1}(\mathbb{R};\mathbb{R}^{3})$ with the Lipschitz constant $L_{g}$, $g(0)=0$, $p$, $q$, $r$, $\alpha_{0}$, $\beta_{0}$ and $\gamma_{0}$ satisfy $(2.13)$, $(2.14)$, $u_{0} \in X^{\alpha_{0}}_{p}$, $\omega_{0} \in Y^{\beta_{0}}_{q}$, $\theta_{0} \in Z^{\gamma_{0}}_{r}$.
Then there exists a positive constant $T_{*}\leq T$ depending only on $\Omega$, $p$, $q$, $r$, $\alpha_{0}$, $\beta_{0}$, $\gamma_{0}$, $u_{0}$, $\omega_{0}$, $\theta_{0}$, $\mu_{r}$, $L_{f}$, $L_{g}$ and $T$ such that $(1.1)$, $(1.2)$ has uniquely a mild solution $(u,\omega,\theta)$ on $[0,T_{*}]$ satisfying
\begin{equation*}
\mathrm{(i)} \ t^{\alpha-\alpha_{0}}u \in C([0,T_{*}];X^{\alpha}_{p}),
\end{equation*}
\begin{equation*}
t^{\beta-\beta_{0}}\omega \in C([0,T_{*}];Y^{\beta}_{q}),
\end{equation*}
\begin{equation*}
t^{\gamma-\gamma_{0}}\theta \in C([0,T_{*}];Z^{\gamma}_{r}),
\end{equation*}
\begin{equation}
\|u(t)\|_{X^{\alpha}_{p}}\leq Ct^{\alpha_{0}-\alpha}(\|u_{0}\|_{X^{\alpha_{0}}_{p}}+\|\omega_{0}\|_{Y^{\beta_{0}}_{q}}+\|\theta_{0}\|_{Z^{\gamma_{0}}_{r}}),
\end{equation}
\begin{equation}
\|\omega(t)\|_{Y^{\beta}_{q}}\leq Ct^{\beta_{0}-\beta}(\|u_{0}\|_{X^{\alpha_{0}}_{p}}+\|\omega_{0}\|_{Y^{\beta_{0}}_{q}}+\|\theta_{0}\|_{Z^{\gamma_{0}}_{r}}),
\end{equation}
\begin{equation}
\|\theta(t)\|_{Z^{\gamma}_{r}}\leq Ct^{\gamma_{0}-\gamma}(\|u_{0}\|_{X^{\alpha_{0}}_{p}}+\|\omega_{0}\|_{Y^{\beta_{0}}_{q}}+\|\theta_{0}\|_{Z^{\gamma_{0}}_{r}})
\end{equation}
for any $\alpha_{0}\leq\alpha<1$, $\beta_{0}\leq\beta<1$, $\gamma_{0}\leq\gamma<1$, $0<t\leq T_{*}$, where $C$ is a positive constant independent of $u$, $\omega$, $\theta$ and $t$.
\begin{equation}
\mathrm{(ii)} \ \|u(t)\|_{X^{\alpha}_{p}}=o(t^{\alpha_{0}-\alpha}) \ \mathrm{as} \ t\rightarrow +0,
\end{equation}
\begin{equation}
\|\omega(t)\|_{Y^{\beta}_{q}}=o(t^{\beta_{0}-\beta}) \ \mathrm{as} \ t\rightarrow +0,
\end{equation}
\begin{equation}
\|\theta(t)\|_{Z^{\gamma}_{r}}=o(t^{\gamma_{0}-\gamma}) \ \mathrm{as} \ t\rightarrow +0
\end{equation}
for any $\alpha_{0}<\alpha<1$, $\beta_{0}<\beta<1$, $\gamma_{0}<\gamma<1$.
\end{theorem}
\begin{theorem}
Let $f \in C^{0,1}(\mathbb{R};\mathbb{R}^{3})$ with the Lipschitz constant $L_{f}$, $f(0)=0$, $g \in C^{0,1}(\mathbb{R};\mathbb{R}^{3})$ with the Lipschitz constant $L_{g}$, $g(0)=0$, $p$, $q$, $r$, $\alpha_{0}$, $\beta_{0}$ and $\gamma_{0}$ satisfy $(2.13)$, $(2.14)$, $u_{0} \in X^{\alpha_{0}}_{p}$, $\omega_{0} \in Y^{\beta_{0}}_{q}$, $\theta_{0} \in Z^{\gamma_{0}}_{r}$, $0<\lambda<\Lambda_{1}$.
Then there exist positive constants $\varepsilon_{1}$ and $\varepsilon_{2}$ depending only on $\Omega$, $p$, $q$, $r$, $\alpha_{0}$, $\beta_{0}$, $\gamma_{0}$, $L_{f}$, $L_{g}$ and $\lambda$ such that $(1.1)$, $(1.2)$ has uniquely a mild solution $(u,\omega,\theta)$ on $[0,\infty)$ satisfying
\begin{equation*}
\mathrm{(i)} \ \min\{t,1\}^{\alpha-\alpha_{0}}e^{\lambda t}u \in C_{b}([0,\infty);X^{\alpha}_{p}),
\end{equation*}
\begin{equation*}
\min\{t,1\}^{\beta-\beta_{0}}e^{\lambda t}\omega \in C_{b}([0,\infty);Y^{\beta}_{q}),
\end{equation*}
\begin{equation*}
\min\{t,1\}^{\gamma-\gamma_{0}}e^{\lambda t}\theta \in C_{b}([0,\infty);Z^{\gamma}_{r}),
\end{equation*}
\begin{equation}
\|u(t)\|_{X^{\alpha}_{p}}\leq C\min\{t,1\}^{\alpha_{0}-\alpha}e^{-\lambda t}(\|u_{0}\|_{X^{\alpha_{0}}_{p}}+\|\omega_{0}\|_{Y^{\beta_{0}}_{q}}+\|\theta_{0}\|_{Z^{\gamma_{0}}_{r}}),
\end{equation}
\begin{equation}
\|\omega(t)\|_{Y^{\beta}_{q}}\leq C\min\{t,1\}^{\beta_{0}-\beta}e^{-\lambda t}(\|u_{0}\|_{X^{\alpha_{0}}_{p}}+\|\omega_{0}\|_{Y^{\beta_{0}}_{q}}+\|\theta_{0}\|_{Z^{\gamma_{0}}_{r}}),
\end{equation}
\begin{equation}
\|\theta(t)\|_{Z^{\gamma}_{r}}\leq C\min\{t,1\}^{\gamma_{0}-\gamma}e^{-\lambda t}(\|u_{0}\|_{X^{\alpha_{0}}_{p}}+\|\omega_{0}\|_{Y^{\beta_{0}}_{q}}+\|\theta_{0}\|_{Z^{\gamma_{0}}_{r}})
\end{equation}
for any $\alpha_{0}\leq\alpha<1$, $\beta_{0}\leq\beta<1$, $\gamma_{0}\leq\gamma<1$, $t>0$, where $C$ is a positive constant independent of $u$, $\omega$, $\theta$ and $t$ provided that
\begin{equation*}
\mu_{r}\leq\varepsilon_{1},
\end{equation*}
\begin{equation*}
\|u_{0}\|_{X^{\alpha_{0}}_{p}}+\|\omega_{0}\|_{Y^{\beta_{0}}_{q}}+\|\theta_{0}\|_{Z^{\gamma_{0}}_{r}}\leq\varepsilon_{2}.
\end{equation*}
\end{theorem}
The second purpose of this paper is to study the regularity of mild solutions of (1.1), (1.2).
As for the regularity of $(d_{t}u,d_{t}\omega,d_{t}\theta)$, it will be required that $p$, $q$, $r$, $\alpha_{0}$, $\beta_{0}$ and $\gamma_{0}$ satisfy the following inequalities:
\begin{equation}
\alpha_{0}\geq 3\left(\frac{1}{p}-\frac{1}{2r}\right), \ \beta_{0}\geq \max\left\{\frac{3}{2p}-\frac{1}{2}, 3\left(\frac{1}{q}-\frac{1}{2r}\right)\right\}, \ \gamma_{0}\geq \frac{3}{2p}-\frac{1}{2},
\end{equation}
\begin{equation}
\begin{split}
&\alpha_{0}\geq \frac{3}{2}\left(\frac{1}{p}+\frac{1}{q}-\frac{1}{r}\right), \ \beta_{0}\geq \frac{3}{2}\left(\frac{1}{p}+\frac{1}{q}-\frac{1}{r}\right), \\
&\alpha_{0}-\gamma_{0}\geq \frac{3}{2}\left(\frac{1}{p}-\frac{1}{q}\right)-\frac{1}{2}, \ \beta_{0}-\gamma_{0}\geq \frac{3}{2}\left(\frac{1}{q}-\frac{1}{p}\right)-\frac{1}{2}, \\
&\beta_{0}-\alpha_{0}>-\frac{1}{2}, \ \beta_{0}-\gamma_{0}>-\frac{1}{2}, \\
&2\alpha_{0}-\beta_{0}\geq \max\left\{\frac{3}{2p}-\frac{1}{2}, 3\left(\frac{1}{p}-\frac{1}{2r}\right)\right\}, \ 2\alpha_{0}-\gamma_{0}\geq \frac{3}{2p}-\frac{1}{2}, \\
&2\beta_{0}-\alpha_{0}\geq 3\left(\frac{1}{q}-\frac{1}{2r}\right), \\
&\alpha_{0}+\beta_{0}-\gamma_{0}\geq \frac{3}{2p}-\frac{1}{2}, \ \alpha_{0}-\beta_{0}+\gamma_{0}\geq \frac{3}{2p}-\frac{1}{2}, \\
&-1<\alpha_{0}-\gamma_{0}-\frac{3}{2}\left(\frac{1}{q}-\frac{1}{r}\right)\leq 1, \ -1<\beta_{0}-\gamma_{0}-\frac{3}{2}\left(\frac{1}{p}-\frac{1}{r}\right)\leq 1.
\end{split}
\end{equation}
We shall prove the following theorems:
\begin{theorem}
If a mild solution $(u,\omega,\theta)$ of $(1.1)$, $(1.2)$ in Theorem $2.1$ is defined on $[0,T]$, then $(u,\omega,\theta)$ is a strong solution of $(1.1)$, $(1.2)$ on $[0,T]$ satisfying
\begin{equation*}
\mathrm{(i)} \ u \in C^{0,\hat{\alpha}}((0,T];X^{1}_{p}), \ d_{t}u \in C^{0,\hat{\alpha}}((0,T];L^{p}_{\sigma}(\Omega)),
\end{equation*}
\begin{equation*}
\omega \in C^{0,\hat{\beta}}((0,T];Y^{1}_{q}), \ d_{t}\omega \in C^{0,\hat{\beta}}((0,T];(L^{q}(\Omega))^{3}),
\end{equation*}
\begin{equation*}
\theta \in C^{0,\hat{\gamma}}((0,T];Z^{1}_{r}), \ d_{t}\theta \in C^{0,\hat{\gamma}}((0,T];L^{r}(\Omega))
\end{equation*}
for some $0<\hat{\alpha}<1$, $0<\hat{\beta}<1$, $0<\hat{\gamma}<1$,
\begin{equation}
\|u(t)\|_{X^{1}_{p}}\leq Ct^{\alpha_{0}-1}(\|u_{0}\|_{X^{\alpha_{0}}_{p}}+\|\omega_{0}\|_{Y^{\beta_{0}}_{q}}+\|\theta_{0}\|_{Z^{\gamma_{0}}_{r}}),
\end{equation}
\begin{equation}
\|\omega(t)\|_{Y^{1}_{q}}\leq Ct^{\beta_{0}-1}(\|u_{0}\|_{X^{\alpha_{0}}_{p}}+\|\omega_{0}\|_{Y^{\beta_{0}}_{q}}+\|\theta_{0}\|_{Z^{\gamma_{0}}_{r}}),
\end{equation}
\begin{equation}
\|\theta(t)\|_{Z^{1}_{r}}\leq Ct^{\gamma_{0}-1}(\|u_{0}\|_{X^{\alpha_{0}}_{p}}+\|\omega_{0}\|_{Y^{\beta_{0}}_{q}}+\|\theta_{0}\|_{Z^{\gamma_{0}}_{r}})
\end{equation}
for any $0<t\leq T$, where $C$ is a positive constant independent of $u$, $\omega$, $\theta$ and $t$.
\begin{equation*}
\mathrm{(ii)} \ u \in C^{0,\hat{\alpha}}((0,T];X^{1}_{p}), \ d_{t}u \in C^{0,\tilde{\alpha}}((0,T];X^{\alpha}_{p}),
\end{equation*}
\begin{equation*}
\omega \in C^{0,\hat{\beta}}((0,T];Y^{1}_{q}), \ d_{t}\omega \in C^{0,\tilde{\beta}}((0,T];Y^{\beta}_{q}),
\end{equation*}
\begin{equation*}
\theta \in C^{0,\hat{\gamma}}((0,T];Z^{1}_{r}), \ d_{t}\theta \in C^{0,\tilde{\gamma}}((0,T];Z^{\gamma}_{r})
\end{equation*}
for any $0<\hat{\alpha}<1$, $0<\hat{\beta}<1$, $0<\hat{\gamma}<1$, $0\leq\alpha<1$, $0\leq\beta<1$, $0\leq\gamma<1$, $0<\tilde{\alpha}<1-\alpha$, $0<\tilde{\beta}<1-\beta$, $0<\tilde{\gamma}<1-\gamma$ provided that $p$, $q$, $r$, $\alpha_{0}$, $\beta_{0}$ and $\gamma_{0}$ satisfy $(2.24)$.
\begin{equation}
\mathrm{(iii)} \ \|d_{t}u(t)\|_{X^{\alpha}_{p}}\leq Ct^{\alpha_{0}-\alpha-1}(\|u_{0}\|_{X^{\alpha_{0}}_{p}}+\|\omega_{0}\|_{Y^{\beta_{0}}_{q}}+\|\theta_{0}\|_{Z^{\gamma_{0}}_{r}}),
\end{equation}
\begin{equation}
\|d_{t}\omega(t)\|_{Y^{\beta}_{q}}\leq Ct^{\beta_{0}-\beta-1}(\|u_{0}\|_{X^{\alpha_{0}}_{p}}+\|\omega_{0}\|_{Y^{\beta_{0}}_{q}}+\|\theta_{0}\|_{Z^{\gamma_{0}}_{r}}),
\end{equation}
\begin{equation}
\|d_{t}\theta(t)\|_{Z^{\gamma}_{r}}\leq Ct^{\gamma_{0}-\gamma-1}(\|u_{0}\|_{X^{\alpha_{0}}_{p}}+\|\omega_{0}\|_{Y^{\beta_{0}}_{q}}+\|\theta_{0}\|_{Z^{\gamma_{0}}_{r}})
\end{equation}
for any $0\leq\alpha<1$, $0\leq\beta<1$, $0\leq\gamma<1$, $0<t\leq T$, where $C$ is a positive constant independent of $u$, $\omega$, $\theta$ and $t$ provided that $p$, $q$, $r$, $\alpha_{0}$, $\beta_{0}$ and $\gamma_{0}$ satisfy $(2.24)$, $(2.25)$.
\end{theorem}
\begin{theorem}
Let $(u,\omega,\theta)$ be a mild solution of $(1.1)$, $(1.2)$ on $[0,\infty)$ satisfying continuity properties and estimates $(2.21)$--$(2.23)$ in Theorem $2.2$.
Then $(u,\omega,\theta)$ is a strong solution of $(1.1)$, $(1.2)$ on $[0,\infty)$ satisfying
\begin{equation*}
\mathrm{(i)} \ u \in C^{0,\hat{\alpha}}((0,\infty);X^{1}_{p}), \ d_{t}u \in C^{0,\hat{\alpha}}((0,\infty);L^{p}_{\sigma}(\Omega)),
\end{equation*}
\begin{equation*}
\omega \in C^{0,\hat{\beta}}((0,\infty);Y^{1}_{q}), \ d_{t}\omega \in C^{0,\hat{\beta}}((0,\infty);(L^{q}(\Omega))^{3}),
\end{equation*}
\begin{equation*}
\theta \in C^{0,\hat{\gamma}}((0,\infty);Z^{1}_{r}), \ d_{t}\theta \in C^{0,\hat{\gamma}}((0,\infty);L^{r}(\Omega))
\end{equation*}
for some $0<\hat{\alpha}<1$, $0<\hat{\beta}<1$, $0<\hat{\gamma}<1$,
\begin{equation}
\|u(t)\|_{X^{1}_{p}}\leq C\min\{t,1\}^{\alpha_{0}-1}e^{-\lambda t}(\|u_{0}\|_{X^{\alpha_{0}}_{p}}+\|\omega_{0}\|_{Y^{\beta_{0}}_{q}}+\|\theta_{0}\|_{Z^{\gamma_{0}}_{r}}),
\end{equation}
\begin{equation}
\|\omega(t)\|_{Y^{1}_{q}}\leq C\min\{t,1\}^{\beta_{0}-1}e^{-\lambda t}(\|u_{0}\|_{X^{\alpha_{0}}_{p}}+\|\omega_{0}\|_{Y^{\beta_{0}}_{q}}+\|\theta_{0}\|_{Z^{\gamma_{0}}_{r}}),
\end{equation}
\begin{equation}
\|\theta(t)\|_{Z^{1}_{r}}\leq C\min\{t,1\}^{\gamma_{0}-1}e^{-\lambda t}(\|u_{0}\|_{X^{\alpha_{0}}_{p}}+\|\omega_{0}\|_{Y^{\beta_{0}}_{q}}+\|\theta_{0}\|_{Z^{\gamma_{0}}_{r}})
\end{equation}
for any $t>0$, where $C$ is a positive constant independent of $u$, $\omega$, $\theta$ and $t$.
\begin{equation*}
\mathrm{(ii)} \ u \in C^{0,\hat{\alpha}}((0,\infty);X^{1}_{p}), \ d_{t}u \in C^{0,\tilde{\alpha}}((0,\infty);X^{\alpha}_{p}),
\end{equation*}
\begin{equation*}
\omega \in C^{0,\hat{\beta}}((0,\infty);Y^{1}_{q}), \ d_{t}\omega \in C^{0,\tilde{\beta}}((0,\infty);Y^{\beta}_{q}),
\end{equation*}
\begin{equation*}
\theta \in C^{0,\hat{\gamma}}((0,\infty);Z^{1}_{r}), \ d_{t}\theta \in C^{0,\tilde{\gamma}}((0,\infty);Z^{\gamma}_{r})
\end{equation*}
for any $0<\hat{\alpha}<1$, $0<\hat{\beta}<1$, $0<\hat{\gamma}<1$, $0\leq\alpha<1$, $0\leq\beta<1$, $0\leq\gamma<1$, $0<\tilde{\alpha}<1-\alpha$, $0<\tilde{\beta}<1-\beta$, $0<\tilde{\gamma}<1-\gamma$ provided that $p$, $q$, $r$, $\alpha_{0}$, $\beta_{0}$ and $\gamma_{0}$ satisfy $(2.24)$.
\begin{equation}
\mathrm{(iii)} \ \|d_{t}u(t)\|_{X^{\alpha}_{p}}\leq C\min\{t,1\}^{\alpha_{0}-\alpha-1}e^{-\lambda t}(\|u_{0}\|_{X^{\alpha_{0}}_{p}}+\|\omega_{0}\|_{Y^{\beta_{0}}_{q}}+\|\theta_{0}\|_{Z^{\gamma_{0}}_{r}}),
\end{equation}
\begin{equation}
\|d_{t}\omega(t)\|_{Y^{\beta}_{q}}\leq C\min\{t,1\}^{\beta_{0}-\beta-1}e^{-\lambda t}(\|u_{0}\|_{X^{\alpha_{0}}_{p}}+\|\omega_{0}\|_{Y^{\beta_{0}}_{q}}+\|\theta_{0}\|_{Z^{\gamma_{0}}_{r}}),
\end{equation}
\begin{equation}
\|d_{t}\theta(t)\|_{Z^{\gamma}_{r}}\leq C\min\{t,1\}^{\gamma_{0}-\gamma-1}e^{-\lambda t}(\|u_{0}\|_{X^{\alpha_{0}}_{p}}+\|\omega_{0}\|_{Y^{\beta_{0}}_{q}}+\|\theta_{0}\|_{Z^{\gamma_{0}}_{r}})
\end{equation}
for any $0\leq\alpha<1$, $0\leq\beta<1$, $0\leq\gamma<1$, $t>0$, where $C$ is a positive constant independent of $u$, $\omega$, $\theta$ and $t$ provided that $p$, $q$, $r$, $\alpha_{0}$, $\beta_{0}$ and $\gamma_{0}$ satisfy $(2.24)$, $(2.25)$.
\end{theorem}
Some detailed considerations admit that a strong solution of (1.1), (1.2) with initial data $(u_{0},\omega_{0},\theta_{0}) \in L^{p}_{\sigma}(\Omega)\times (L^{q}(\Omega))^{3}\times L^{r}(\Omega)$ can be grasped in the classical sense.
Let $p$, $q=p$ and $r$ satisfy the following inequalities:
\begin{equation}
3<p<\infty, \ 3<r<\infty, \ \frac{1}{p}-\frac{1}{2r}\leq 0, \ \frac{1}{r}-\frac{1}{p}<\frac{2}{3}.
\end{equation}
Then we can take $\alpha_{0}$, $\beta_{0}$ and $\gamma_{0}$ in (2.13), (2.14), (2.24), (2.25) as zeros.
It is derived from Theorems 2.3 and 2.4 that we obtain the following corollaries:
\begin{corollary}
Let $f \in C^{0,1}(\mathbb{R};\mathbb{R}^{3})\cap C^{1}(\mathbb{R};\mathbb{R}^{3})$, $f(0)=0$, $g \in C^{0,1}(\mathbb{R};\mathbb{R}^{3})\cap C^{1}(\mathbb{R};\mathbb{R}^{3})$, $g(0)=0$, $p$, $q=p$ and $r$ satisfy $(2.38)$, $u_{0} \in L^{p}_{\sigma}(\Omega)$, $\omega_{0} \in (L^{p}(\Omega))^{3}$, $\theta_{0} \in L^{r}(\Omega)$.
Then a strong solution $(u,\omega,\theta)$ of $(1.1)$, $(1.2)$ in Theorem $2.3$ is a classical solution of $(1.1)$, $(1.2)$ in $(0,T]$ satisfying
\begin{equation*}
u \in C^{0,\hat{\alpha}}((0,T];(C^{2,\alpha}(\overline{\Omega}))^{3}), \ d_{t}u \in C^{0,\tilde{\alpha}}((0,T];(C^{1,\alpha}(\overline{\Omega}))^{3}),
\end{equation*}
\begin{equation*}
\omega \in C^{0,\hat{\beta}}((0,T];(C^{2,\beta}(\overline{\Omega}))^{3}), \ d_{t}\omega \in C^{0,\tilde{\beta}}((0,T];(C^{1,\beta}(\overline{\Omega}))^{3}),
\end{equation*}
\begin{equation*}
\theta \in C^{0,\hat{\gamma}}((0,T];C^{2,\gamma}(\overline{\Omega})), \ d_{t}\theta \in C^{0,\tilde{\gamma}}((0,T];C^{1,\gamma}(\overline{\Omega}))
\end{equation*}
for any $0<\hat{\alpha}<1/2$, $0<\hat{\beta}<1/2$, $0<\hat{\gamma}<1/2$, $0<\alpha<1-3/p$, $0<\beta<1-3/p$, $0<\gamma<1-3/r$ and for some $0<\tilde{\alpha}<1$, $0<\tilde{\beta}<1$, $0<\tilde{\gamma}<1$.
\end{corollary}
\begin{corollary}
Let $f \in C^{0,1}(\mathbb{R};\mathbb{R}^{3})\cap C^{1}(\mathbb{R};\mathbb{R}^{3})$, $f(0)=0$, $g \in C^{0,1}(\mathbb{R};\mathbb{R}^{3})\cap C^{1}(\mathbb{R};\mathbb{R}^{3})$, $g(0)=0$, $p$, $q=p$ and $r$ satisfy $(2.38)$, $u_{0} \in L^{p}_{\sigma}(\Omega)$, $\omega_{0} \in (L^{p}(\Omega))^{3}$, $\theta_{0} \in L^{r}(\Omega)$.
Then a strong solution $(u,\omega,\theta)$ of $(1.1)$, $(1.2)$ in Theorem $2.4$ is a classical solution of $(1.1)$, $(1.2)$ in $(0,\infty)$ satisfying
\begin{equation*}
u \in C^{0,\hat{\alpha}}((0,\infty);(C^{2,\alpha}(\overline{\Omega}))^{3}), \ d_{t}u \in C^{0,\tilde{\alpha}}((0,\infty);(C^{1,\alpha}(\overline{\Omega}))^{3}),
\end{equation*}
\begin{equation*}
\omega \in C^{0,\hat{\beta}}((0,\infty);(C^{2,\beta}(\overline{\Omega}))^{3}), \ d_{t}\omega \in C^{0,\tilde{\beta}}((0,\infty);(C^{1,\beta}(\overline{\Omega}))^{3}),
\end{equation*}
\begin{equation*}
\theta \in C^{0,\hat{\gamma}}((0,\infty);C^{2,\gamma}(\overline{\Omega})), \ d_{t}\theta \in C^{0,\tilde{\gamma}}((0,\infty);C^{1,\gamma}(\overline{\Omega}))
\end{equation*}
for any $0<\hat{\alpha}<1/2$, $0<\hat{\beta}<1/2$, $0<\hat{\gamma}<1/2$, $0<\alpha<1-3/p$, $0<\beta<1-3/p$, $0<\gamma<1-3/r$ and for some $0<\tilde{\alpha}<1$, $0<\tilde{\beta}<1$, $0<\tilde{\gamma}<1$.
\end{corollary}
\begin{remark}
It can be easily seen from \cite[Theorem 7.3.6]{Pazy}, \cite[Theorem 1.3]{Shibata} that our main results are still valid, instead of (1.2), for the following initial-boundary data:
\begin{equation*}
\begin{split}
&u|_{t=0}=u_{0} & \mathrm{in} \ \Omega, \\
&u_{\nu}|_{\partial\Omega}=0 & \mathrm{on} \ \partial\Omega\times(0,T), \\
&K(T(u,p)\nu)_{\tau}+(1-K)u_{\tau}|_{\partial\Omega}=0 & \mathrm{on} \ \partial\Omega\times(0,T), \\
&\omega|_{t=0}=\omega_{0} & \mathrm{in} \ \Omega, \\
&\omega|_{\partial\Omega}=0 & \mathrm{on} \ \partial\Omega\times(0,T), \\
&\theta|_{t=0}=\theta_{0} & \mathrm{in} \ \Omega, \\
&\theta|_{\partial\Omega}=\theta_{s} \ \mathrm{or} \ \kappa\partial_{\nu}\theta+\kappa_{s}\theta|_{\partial\Omega}=0 & \mathrm{on} \ \partial\Omega\times(0,T),
\end{split}
\end{equation*}
where $\nu$ is the outward unit normal vector on $\partial\Omega$, $u_{\nu}:=\nu\cdot u$, $u_{\tau}:=u-u_{\nu}\nu$, $T(u,p)$ is the Cauchy stress tensor defined as
\begin{equation*}
T(u,p)=-pI_{3}+2\mu D(u),
\end{equation*}
$I_{3}$ is the third identity matrix, $0\leq K<1$ is a constant, $\kappa_{s}$ is a positive constant.
Moreover, it is useful to remark that $(T(u,p)\nu)_{\tau}=T(u,p)\nu-(\nu\cdot T(u,p)\nu)\nu=2\mu(D(u)\nu)_{\tau}$.
\end{remark}

\subsection{$L^{p}_{\sigma}\times (L^{q})^{3}\times L^{r}$-estimates for linear and nonlinear terms}
We will state and prove some lemmas which play an important role throughout this paper.
It is assured by them that we obtain $L^{p}_{\sigma}$-estimates for $F(u,\omega,\theta)$, $(L^{q})^{3}$-estimates for $G(u,\omega,\theta)$ and $L^{r}$-estimates for $H(u,\omega,\theta)$.
\begin{lemma}
Let $1<p<\infty$,
\begin{equation*}
\alpha_{1}>0, \ 0\leq\delta_{1}<\frac{1}{2}+\frac{3}{2}\left(1-\frac{1}{p}\right), \ \alpha_{1}+\delta_{1}>\frac{1}{2}, \ 2\alpha_{1}+\delta_{1}\geq\frac{3}{2p}+\frac{1}{2}.
\end{equation*}
Then
\begin{equation}
\|A^{-\delta_{1}}_{p}P_{p}(u\cdot\nabla)v\|_{p}\leq C_{1}\|u\|_{X^{\alpha_{1}}_{p}}\|v\|_{X^{\alpha_{1}}_{p}}
\end{equation}
for any $u, v \in X^{\alpha_{1}}_{p}$, where $C_{1}=C_{1}(\alpha_{1},\delta_{1})$ is a positive constant.
\end{lemma}
\begin{proof}
It is \cite[Lemma 2.2]{Giga 3}.
\end{proof}
\begin{lemma}
Let $1<p<\infty$, $1<q<\infty$,
\begin{equation*}
\alpha_{2}, \beta_{2}\geq 0, \ \alpha_{2}>\frac{3}{2}\left(\frac{1}{p}-\frac{1}{q}\right), \ 0\leq\delta_{2}<\frac{1}{2}+\frac{3}{2}\left(1-\frac{1}{q}\right), \ \beta_{2}+\delta_{2}>\frac{1}{2}, \ \alpha_{2}+\beta_{2}+\delta_{2}\geq\frac{3}{2p}+\frac{1}{2}.
\end{equation*}
Then
\begin{equation}
\|\Gamma^{-\delta_{2}}_{q}(u\cdot\nabla)\omega\|_{q}\leq C_{2}\|u\|_{X^{\alpha_{2}}_{p}}\|\omega\|_{Y^{\beta_{2}}_{q}}
\end{equation}
for any $u \in X^{\alpha_{2}}_{p}$, $\omega \in Y^{\beta_{2}}_{q}$,where $C_{2}=C_{2}(\alpha_{2},\beta_{2},\delta_{2})$ is a positive constant.
\end{lemma}
\begin{proof}
It is \cite[Lemma 3.3]{Hishida}.
\end{proof}
\begin{lemma}
Let $1<p<\infty$, $1<r<\infty$,
\begin{equation*}
\alpha_{3}, \gamma_{3}\geq 0, \ \alpha_{3}>\frac{3}{2}\left(\frac{1}{p}-\frac{1}{r}\right), \ 0\leq\delta_{3}<\frac{1}{2}+\frac{3}{2}\left(1-\frac{1}{r}\right), \ \beta_{3}+\delta_{3}>\frac{1}{2}, \ \alpha_{3}+\gamma_{3}+\delta_{3}\geq\frac{3}{2p}+\frac{1}{2}.
\end{equation*}
Then
\begin{equation}
\|B^{-\delta_{3}}_{r}(u\cdot\nabla)\theta\|_{r}\leq C_{3}\|u\|_{X^{\alpha_{3}}_{p}}\|\theta\|_{Z^{\gamma_{3}}_{r}}
\end{equation}
for any $u \in X^{\alpha_{3}}_{p}$, $\theta \in Z^{\gamma_{3}}_{r}$,where $C_{3}=C_{3}(\alpha_{3},\gamma_{3},\delta_{3})$ is a positive constant.
\end{lemma}
\begin{proof}
It is \cite[Lemma 3.3]{Hishida}.
\end{proof}
\begin{lemma}
Let $1<p<\infty$, $1<q<\infty$, $1<r<\infty$,
\begin{equation*}
\max\left\{0, \frac{1}{2}+\frac{3}{2}\left(\frac{1}{p}-\frac{1}{2r}\right)\right\}\leq\alpha_{3}\leq 1, \ \max\left\{0, \frac{1}{2}+\frac{3}{2}\left(\frac{1}{q}-\frac{1}{2r}\right)\right\}\leq\beta_{3}\leq 1.
\end{equation*}
Then
\begin{equation}
\begin{split}
\|\Phi(u,v;\omega,\psi)\|_{q}\leq& C_{4}(1+\mu_{r})(\|u\|_{X^{\alpha_{3}}_{p}}\|v\|_{X^{\alpha_{3}}_{p}}+\|u\|_{X^{\alpha_{3}}_{p}}\|\psi\|_{Y^{\beta_{3}}_{q}} \\
&+\|v\|_{X^{\alpha_{3}}_{p}}\|\omega\|_{Y^{\beta_{3}}_{q}}+\|\omega\|_{Y^{\beta_{3}}_{q}}\|\psi\|_{Y^{\beta_{3}}_{q}})
\end{split}
\end{equation}
for any $u, v \in X^{\alpha_{3}}_{p}$, $\omega, \psi \in Y^{\beta_{3}}_{q}$, where $C_{4}=C_{4}(\alpha_{3},\beta_{3})$ is a positive constant.
\end{lemma}
\begin{proof}
After applying the Schwarz inequality to $\|\Phi(u,v;\omega,\psi)\|_{r}$, we can obtain (2.42) by (2.10) with $\alpha=\alpha_{3}$, $k=1$ and $s=2r$, (2.11) with $\beta=\beta_{3}$, $k=1$ and $s=2r$.
\end{proof}
\begin{lemma}
Let $1<p<\infty$, $1<q<\infty$,
\begin{equation*}
\beta_{1}\geq 0, \ \beta_{1}>\frac{3}{2}\left(\frac{1}{q}-\frac{1}{p}\right), \ 0\leq\delta_{1}<\frac{1}{2}+\frac{3}{2}\left(1-\frac{1}{p}\right), \ \beta_{1}+\delta_{1}\geq\frac{3}{2}\left(\frac{1}{q}-\frac{1}{p}\right)+\frac{1}{2}.
\end{equation*}
Then
\begin{equation}
\|A^{-\delta_{1}}_{p}P_{p}\mathrm{rot}\omega\|_{p}\leq C_{5}\|\omega\|_{Y^{\beta_{1}}_{q}}
\end{equation}
for any $\omega \in Y^{\beta_{1}}_{q}$, where $C_{5}=C_{5}(\beta_{1},\delta_{1})$ is a positive constant.
\end{lemma}
\begin{proof}
It follows immediately from from \cite[Lemma 2.2]{Giga 3}.
\end{proof}
\begin{lemma}
Let $1<q<\infty$, $0\leq\beta_{2}\leq 1$.
Then
\begin{equation}
\|\omega\|_{q}\leq C_{6}\|\omega\|_{Y^{\beta_{2}}_{q}}
\end{equation}
for any $\omega \in Y^{\beta_{2}}_{q}$, where $C_{6}=C_{6}(\beta_{2})$ is a positive constant.
\end{lemma}
\begin{proof}
It is clear from (2.11) with $\beta=\beta_{2}$, $k=0$ and $s=q$.
\end{proof}
\begin{lemma}
Let $1<p<\infty$, $1<q<\infty$,
\begin{equation*}
\alpha_{2}\geq 0, \ \alpha_{2}>\frac{3}{2}\left(\frac{1}{p}-\frac{1}{q}\right), \ 0\leq\delta_{2}<\frac{1}{2}+\frac{3}{2}\left(1-\frac{1}{q}\right), \ \alpha_{2}+\delta_{2}\geq\frac{3}{2}\left(\frac{1}{p}-\frac{1}{q}\right)+\frac{1}{2}.
\end{equation*}
Then
\begin{equation}
\|\Gamma^{-\delta_{2}}_{q}\mathrm{rot}u\|_{q}\leq C_{7}\|u\|_{X^{\alpha_{2}}_{p}}
\end{equation}
for any $u \in X^{\alpha_{2}}_{p}$, where $C_{7}=C_{7}(\alpha_{2},\delta_{2})$ is a positive constant.
\end{lemma}
\begin{proof}
It follows immediately from \cite[Lemma 2.2]{Giga 3}.
\end{proof}
\begin{lemma}
Let $f \in C^{0,1}(\mathbb{R})$ with the Lipschitz constant $L_{f}$, $f(0)=0$, $1<p<\infty$, $1<r<\infty$,
\begin{equation*}
\max\left\{0, \ \frac{3}{2}\left(\frac{1}{r}-\frac{1}{p}\right) \right\}\leq\gamma_{1}\leq 1.
\end{equation*}
Then
\begin{equation}
\|P_{p}f(\theta)\|_{p}\leq C_{8}L_{f}\|\theta\|_{Z^{\gamma_{1}}_{r}}
\end{equation}
for any $\theta \in Z^{\gamma_{1}}_{r}$, where $C_{8}=C_{8}(\gamma_{1})$ is a positive constant.
\end{lemma}
\begin{proof}
It is known in \cite[Theorem 1]{Fujiwara} that $P_{p}$ is a bounded operator in $(L^{p}(\Omega))^{3}$.
Since $\|f(\theta)\|_{p}\leq L_{f}\|\theta\|_{p}$, (2.46) follows from (2.12) with $\gamma=\gamma_{1}$, $k=0$ and $s=p$.
\end{proof}
\begin{lemma}
Let $g \in C^{0,1}(\mathbb{R})$ with the Lipschitz constant $L_{g}$, $g(0)=0$, $1<q<\infty$, $1<r<\infty$,
\begin{equation*}
\max\left\{0, \ \frac{3}{2}\left(\frac{1}{r}-\frac{1}{q}\right) \right\}\leq\gamma_{2}\leq 1.
\end{equation*}
Then
\begin{equation}
\|g(\theta)\|_{q}\leq C_{9}L_{g}\|\theta\|_{Z^{\gamma_{2}}_{r}}
\end{equation}
for any $\theta \in Z^{\gamma_{2}}_{r}$, where $C_{9}=C_{9}(\gamma_{2})$ is a positive constant.
\end{lemma}
\begin{proof}
Since $\|g(\theta)\|_{q}\leq L_{g}\|\theta\|_{q}$, (2.47) follows from (2.12) with $\gamma=\gamma_{2}$, $k=0$ and $s=q$.
\end{proof}

\subsection{$X^{\alpha}_{p}\times Y^{\beta}_{q}\times Z^{\gamma}_{r}$-estimates for linear and nonlinear terms}
First, we will fix nine exponents $\alpha_{1}$, $\alpha_{2}$, $\alpha_{3}$, $\beta_{1}$, $\beta_{2}$, $\beta_{3}$, $\gamma_{1}$, $\gamma_{2}$ and $\gamma_{3}$ in Lemmas 2.5--2.13 after the choice of three exponents $\delta_{1}$ in Lemmas 2.5 and 2.9, $\delta_{2}$ in Lemmas 2.6 and 2.11 and $\delta_{3}$ in Lemma 2.7.
We take $\delta_{1}$ as zero in the case where $\alpha_{0}>0$, and as an arbitrary positive constant in the case where $\alpha_{0}=0$.
$(\beta_{0},\delta_{2})$ and $(\gamma_{0},\delta_{3})$ are similarly taken.
It is essential for (2.13), (2.14) that we make an appropriate choice of $\alpha_{1}$ in Lemma 2.5, $\alpha_{2}$ in Lemmas 2.6 and 2.11, $\alpha_{3}$ in Lemmas 2.7 and 2.8, $\beta_{1}$ in Lemma 2.9, $\beta_{2}$ in Lemmas 2.6 and 2.10, $\beta_{3}$ in Lemma 2.8, $\gamma_{1}$ in Lemma 2.12, $\gamma_{2}$ in Lemma 2.13 and $\gamma_{3}$ in Lemma 2.7.
Some elementary demonstrations admit that we can chose $\alpha_{0}<\alpha_{1}<1-\delta_{1}$, $\alpha_{0}<\alpha_{2}<1-\delta_{1}$, $\alpha_{0}<\alpha_{3}<1-\delta_{1}$, $\beta_{0}<\beta_{1}<1-\delta_{2}$, $\beta_{0}<\beta_{2}<1-\delta_{2}$, $\beta_{0}<\beta_{3}<1-\delta_{2}$, $\gamma_{0}<\gamma_{1}<1-\delta_{3}$, $\gamma_{0}<\gamma_{2}<1-\delta_{3}$, $\gamma_{0}<\gamma_{3}<1-\delta_{3}$ which satisfy not only assumptions for Lemmas 2.5--2.13 but also
\begin{equation}
\begin{split}
&2\alpha_{1}+\delta_{1}\leq 1+\alpha_{0}, \ \alpha_{2}+\beta_{2}+\delta_{2}\leq 1+\alpha_{0} \ (\alpha_{2}+\delta_{2}<1+\alpha_{0}-\beta_{0}), \\
&\alpha_{3}+\gamma_{3}+\delta_{3}\leq 1+\alpha_{0}, \ \alpha_{3}\leq \alpha_{0}+\frac{1-\gamma_{0}}{2}, \ \beta_{3}\leq \beta_{0}+\frac{1-\gamma_{0}}{2},
\end{split}
\end{equation}
\begin{equation}
\begin{cases}
\beta_{1}+\delta_{1}<1-\alpha_{0}+\beta_{0} & \mathrm{if} \ \alpha_{0}-\beta_{0}-\dfrac{3}{2}\left(\dfrac{1}{p}-\dfrac{1}{q}\right)<\dfrac{1}{2}, \\
\beta_{1}+\delta_{1}=1-\alpha_{0}+\beta_{0} & \mathrm{if} \ \alpha_{0}-\beta_{0}-\dfrac{3}{2}\left(\dfrac{1}{p}-\dfrac{1}{q}\right)=\dfrac{1}{2},
\end{cases}
\end{equation}
\begin{equation}
\begin{cases}
\gamma_{1}<1-\alpha_{0}+\gamma_{0} & \mathrm{if} \ \left|\alpha_{0}-\gamma_{0}-\dfrac{3}{2}\left(\dfrac{1}{p}-\dfrac{1}{r}\right)\right|<1, \\
\gamma_{1}=1-\alpha_{0}+\gamma_{0} & \mathrm{if} \ \alpha_{0}-\gamma_{0}-\dfrac{3}{2}\left(\dfrac{1}{p}-\dfrac{1}{r}\right)=1,
\end{cases}
\end{equation}
\begin{equation}
\begin{cases}
\gamma_{2}<1-\beta_{0}+\gamma_{0} & \mathrm{if} \ \left|\beta_{0}-\gamma_{0}-\dfrac{3}{2}\left(\dfrac{1}{q}-\dfrac{1}{r}\right)\right|<1, \\
\gamma_{2}=1-\beta_{0}+\gamma_{0} & \mathrm{if} \ \beta_{0}-\gamma_{0}-\dfrac{3}{2}\left(\dfrac{1}{q}-\dfrac{1}{r}\right)=1.
\end{cases}
\end{equation}
These exponents are fixed throughout this paper.

Second, we obtain $X^{\alpha}_{p}\times Y^{\beta}_{q}\times Z^{\gamma}_{r}$-estimates for linear and nonlinear terms which appeared in (II).
Let $0\leq \alpha<1-\delta_{1}$, $0\leq \beta<1-\delta_{2}$, $0\leq\gamma<1-\delta_{3}$, $0<\lambda<\Lambda_{1}$, and set
\begin{equation*}
\mathcal{F}(u,\omega,\theta)(t)=\int^{t}_{0}e^{-(t-s)A_{p}}F(u,\omega,\theta)(s)ds,
\end{equation*}
\begin{equation*}
\mathcal{G}(u,\omega,\theta)(t)=\int^{t}_{0}e^{-(t-s)\Gamma_{q}}G(u,\omega,\theta)(s)ds,
\end{equation*}
\begin{equation*}
\mathcal{H}(u,\omega,\theta)(t)=\int^{t}_{0}e^{-(t-s)B_{r}}H(u,\omega,\theta)(s)ds.
\end{equation*}
Then $\|\mathcal{F}(u,\omega,\theta)(t)\|_{X^{\alpha}_{p}}$, $\|\mathcal{G}(u,\omega,\theta)(t)\|_{Y^{\beta}_{q}}$ and $\|\mathcal{H}(u,\omega,\theta)(t)\|_{Z^{\gamma}_{r}}$ are bounded as follows:
\begin{equation}
\begin{split}
\|\mathcal{F}(u,\omega,\theta)(t)\|_{X^{\alpha}_{p}}\leq& C_{A_{p},\alpha+\delta_{1},\lambda}C_{1}\int^{t}_{0}(t-s)^{-(\alpha+\delta_{1})}e^{-\lambda(t-s)}\|u(s)\|^{2}_{X^{\alpha_{1}}_{p}}ds \\
&+2C_{A_{p},\alpha+\delta_{1},\lambda}C_{5}\mu_{r}\int^{t}_{0}(t-s)^{-(\alpha+\delta_{1})}e^{-\lambda(t-s)}\|\omega(s)\|_{Y^{\beta_{1}}_{q}}ds \\
&+C_{A_{p},\alpha,\lambda}C_{8}L_{f}\int^{t}_{0}(t-s)^{-\alpha}e^{-\lambda(t-s)}\|\theta(s)\|_{Z^{\gamma_{1}}_{r}}ds,
\end{split}
\end{equation}
\begin{equation}
\begin{split}
\|\mathcal{G}(u,\omega,\theta)(t)\|_{Y^{\beta}_{q}}\leq& C_{\Gamma_{q},\beta+\delta_{2},\lambda}C_{2}\int^{t}_{0}(t-s)^{-(\beta+\delta_{2})}e^{-\lambda(t-s)}\|u(s)\|_{X^{\alpha_{2}}_{p}}\|\omega(s)\|_{Y^{\beta_{2}}_{q}}ds \\
&+4C_{\Gamma_{q},\beta,\lambda}C_{6}\mu_{r}\int^{t}_{0}(t-s)^{-\beta}e^{-\lambda(t-s)}\|\omega(s)\|_{Y^{\beta_{2}}_{q}}ds \\
&+2C_{\Gamma_{q},\beta+\delta_{2},\lambda}C_{7}\mu_{r}\int^{t}_{0}(t-s)^{-(\beta+\delta_{2})}e^{-\lambda(t-s)}\|u(s)\|_{X^{\alpha_{2}}_{p}}ds \\
&+C_{\Gamma_{q},\beta,\lambda}C_{9}L_{g}\int^{t}_{0}(t-s)^{-\beta}e^{-\lambda(t-s)}\|\theta(s)\|_{Z^{\gamma_{2}}_{r}}ds,
\end{split}
\end{equation}
\begin{equation}
\begin{split}
\|\mathcal{H}(u,\omega,\theta)(t)\|_{Z^{\gamma}_{r}}\leq& C_{B_{r},\gamma+\delta_{3},\lambda}C_{3}\int^{t}_{0}(t-s)^{-(\gamma+\delta_{3})}e^{-\lambda(t-s)}\|u(s)\|_{X^{\alpha_{3}}_{p}}\|\theta(s)\|_{Z^{\gamma_{3}}_{r}}ds \\
&+C_{B_{r},\gamma,\lambda}C_{4}(1+\mu_{r})\int^{t}_{0}(t-s)^{-\gamma}e^{-\lambda(t-s)}\|u(s)\|^{2}_{X^{\alpha_{3}}_{p}}ds \\
&+2C_{B_{r},\gamma,\lambda}C_{4}(1+\mu_{r})\int^{t}_{0}(t-s)^{-\gamma}e^{-\lambda(t-s)}\|u(s)\|_{X^{\alpha_{3}}_{p}}\|\omega(s)\|_{Y^{\beta_{3}}_{q}}ds \\
&+C_{B_{r},\gamma,\lambda}C_{4}(1+\mu_{r})\int^{t}_{0}(t-s)^{-\gamma}e^{-\lambda(t-s)}\|\omega(s)\|^{2}_{Y^{\beta_{3}}_{q}}ds
\end{split}
\end{equation}
for any $0\leq t\leq T$.
Let $(u_{1},\omega_{1},\theta_{1})$ and $(u_{2},\omega_{2},\theta_{2})$ be two mild solutions of (1.1), (1.2).
Then we have the following inequalities:
\begin{equation}
\begin{split}
\|\mathcal{F}(u_{2},\omega_{2}&,\theta_{2})(t)-\mathcal{F}(u_{1},\omega_{1},\theta_{1})(t)\|_{X^{\alpha}_{p}} \\
\leq& C_{A_{p},\alpha+\delta_{1},\lambda}C_{1}\int^{t}_{0}(t-s)^{-(\alpha+\delta_{1})}e^{-\lambda(t-s)}(\|u_{1}(s)\|_{X^{\alpha_{1}}_{p}}+\|u_{2}(s)\|_{X^{\alpha_{1}}_{p}}) \\
&\times\|(u_{2}-u_{1})(s)\|_{X^{\alpha_{1}}_{p}}ds \\
&+2C_{A_{p},\alpha+\delta_{1},\lambda}C_{5}\mu_{r}\int^{t}_{0}(t-s)^{-(\alpha+\delta_{1})}e^{-\lambda(t-s)}\|(\omega_{2}-\omega_{1})(s)\|_{Y^{\beta_{1}}_{q}}ds \\
&+C_{A_{p},\alpha,\lambda}C_{8}L_{f}\int^{t}_{0}(t-s)^{-\alpha}e^{-\lambda(t-s)}\|(\theta_{2}-\theta_{1})(s)\|_{Z^{\gamma_{1}}_{r}}ds,
\end{split}
\end{equation}
\begin{equation}
\begin{split}
\|\mathcal{G}(u_{2},\omega_{2}&,\theta_{2})(t)-\mathcal{G}(u_{1},\omega_{1},\theta_{1})(t)\|_{Y^{\beta}_{q}} \\
\leq& C_{\Gamma_{q},\beta+\delta_{2},\lambda}C_{2}\int^{t}_{0}(t-s)^{-(\beta+\delta_{2})}e^{-\lambda(t-s)}\|(u_{2}-u_{1})(s)\|_{X^{\alpha_{2}}_{p}}\|\omega_{2}(s)\|_{Y^{\beta_{2}}_{q}}ds \\
&+C_{\Gamma_{q},\beta+\delta_{2},\lambda}C_{2}\int^{t}_{0}(t-s)^{-(\beta+\delta_{2})}e^{-\lambda(t-s)}\|u_{1}(s)\|_{X^{\alpha_{2}}_{p}}\|(\omega_{2}-\omega_{1})(s)\|_{Y^{\beta_{2}}_{q}}ds \\
&+4C_{\Gamma_{q},\beta,\lambda}C_{6}\mu_{r}\int^{t}_{0}(t-s)^{-\beta}e^{-\lambda(t-s)}\|(\omega_{2}-\omega_{1})(s)\|_{Y^{\beta_{2}}_{q}}ds \\
&+2C_{\Gamma_{q},\beta+\delta_{2},\lambda}C_{7}\mu_{r}\int^{t}_{0}(t-s)^{-(\beta+\delta_{2})}e^{-\lambda(t-s)}\|(u_{2}-u_{1})(s)\|_{X^{\alpha_{2}}_{p}}ds \\
&+C_{\Gamma_{q},\beta,\lambda}C_{9}L_{g}\int^{t}_{0}(t-s)^{-\beta}e^{-\lambda(t-s)}\|(\theta_{2}-\theta_{1})(s)\|_{Z^{\gamma_{2}}_{r}}ds,
\end{split}
\end{equation}
\begin{equation}
\begin{split}
\|\mathcal{H}(u_{2},\omega_{2}&,\theta_{2})(t)-\mathcal{H}(u_{1},\omega_{1},\theta_{1})(t)\|_{Z^{\gamma}_{r}} \\
\leq& C_{B_{r},\gamma+\delta_{3},\lambda}C_{3}\int^{t}_{0}(t-s)^{-(\gamma+\delta_{3})}e^{-\lambda(t-s)}\|(u_{2}-u_{1})(s)\|_{X^{\alpha_{3}}_{p}}\|\theta_{2}(s)\|_{Z^{\gamma_{3}}_{r}}ds \\
&+C_{B_{r},\gamma+\delta_{3},\lambda}C_{3}\int^{t}_{0}(t-s)^{-(\gamma+\delta_{3})}e^{-\lambda(t-s)}\|u_{1}(s)\|_{X^{\alpha_{3}}_{p}}\|(\theta_{2}-\theta_{1})(s)\|_{Z^{\gamma_{3}}_{r}}ds \\
&+C_{B_{r},\gamma,\lambda}C_{4}(1+\mu_{r})\int^{t}_{0}(t-s)^{-\gamma}e^{-\lambda(t-s)}(\|u_{1}(s)\|_{X^{\alpha_{3}}_{p}}+\|u_{2}(s)\|_{X^{\alpha_{3}}_{p}}) \\
&\times\|(u_{2}-u_{1})(s)\|_{X^{\alpha_{3}}_{p}}ds \\
&+C_{B_{r},\gamma,\lambda}C_{4}(1+\mu_{r})\int^{t}_{0}(t-s)^{-\gamma}e^{-\lambda(t-s)}\|\omega_{2}(s)\|_{Y^{\beta_{3}}_{q}}\|(u_{2}-u_{1})(s)\|_{X^{\alpha_{3}}_{p}}ds \\
&+C_{B_{r},\gamma,\lambda}C_{4}(1+\mu_{r})\int^{t}_{0}(t-s)^{-\gamma}e^{-\lambda(t-s)}\|u_{1}(s)\|_{X^{\alpha_{3}}_{p}}\|(\omega_{2}-\omega_{1})(s)\|_{Y^{\beta_{3}}_{q}}ds \\
&+C_{B_{r},\gamma,\lambda}C_{4}(1+\mu_{r})\int^{t}_{0}(t-s)^{-\gamma}e^{-\lambda(t-s)}(\|\omega_{1}(s)\|_{Y^{\beta_{3}}_{q}}+\|\omega_{2}(s)\|_{Y^{\beta_{3}}_{q}}) \\
&\times\|(\omega_{2}-\omega_{1})(s)\|_{Y^{\beta_{3}}_{q}}ds
\end{split}
\end{equation}
for any $0\leq t\leq T$.

\section{Proof of Theorems 2.1 and 2.2}
We will prove Theorems 2.1 and 2.2 in this section.
In proving our main results, simplified notation is given as follows: We drop three subscripts $p$, $q$ and $r$ attached to $P$, $A$, $\Gamma$, $B$, $X^{\alpha}$, $Y^{\beta}$ and $Z^{\gamma}$ in the sequel.
It is useful to remark that a generic positive constant independent of $u$, $\omega$, $\theta$ and $t$ is simply denoted by $C$.

\subsection{Existence of local mild solutions}
We construct a mild solution $(u,\omega,\theta)$ of (1.1), (1.2) by the following successive approximation $(u^{m},\omega^{m},\theta^{m})$ $(m \in \mathbb{Z}, \ m\geq 0)$:
\begin{equation}
\begin{split}
&u^{0}(t)=e^{-tA}u_{0}, \\
&\omega^{0}(t)=e^{-t\Gamma}\omega_{0}, \\
&\theta^{0}(t)=e^{-tB}\theta_{0},
\end{split}
\end{equation}
\begin{equation}
\begin{split}
&u^{m+1}=u^{0}+\mathcal{F}(u^{m},\omega^{m},\theta^{m}), \\
&\omega^{m+1}=\omega^{0}+\mathcal{G}(u^{m},\omega^{m},\theta^{m}), \\
&\theta^{m+1}=\theta^{0}+\mathcal{H}(u^{m},\omega^{m},\theta^{m}).
\end{split}
\end{equation}
The following lemma admits that $\{t^{\alpha-\alpha_{0}}u^{m}\}_{m}$, $\{t^{\beta-\beta_{0}}\omega^{m}\}_{m}$ and $\{t^{\gamma-\gamma_{0}}\theta^{m}\}_{m}$ are well-defined as sequences in $C([0,T];X^{\alpha})$ for $\alpha=\alpha_{1}, \alpha_{2}, \alpha_{3}$, in $C([0,T];Y^{\beta})$ for $\beta=\beta_{1}, \beta_{2}, \beta_{3}$ and in $C([0,T];Z^{\gamma})$ for $\gamma=\gamma_{1}, \gamma_{2}, \gamma_{3}$ respectively.
\begin{lemma}
Let $\alpha=\alpha_{1}, \alpha_{2}, \alpha_{3}$, $\beta=\beta_{1}, \beta_{2}, \beta_{3}$, $\gamma=\gamma_{1}, \gamma_{2}, \gamma_{3}$.
Then there exist monotone increasing continuous functions $K^{m}_{1,\alpha}$, $K^{m}_{2,\beta}$ and $K^{m}_{3,\gamma}$ on $[0,T]$ for any $m \in \mathbb{Z}$, $m\geq 0$ such that $K^{m}_{1,\alpha}(0)=0$, $K^{m}_{2,\beta}(0)=0$, $K^{m}_{3,\gamma}(0)=0$,
\begin{equation}
\|u^{m}(t)\|_{X^{\alpha}}\leq K^{m}_{1,\alpha}(t)t^{\alpha_{0}-\alpha},
\end{equation}
\begin{equation}
\|\omega^{m}(t)\|_{Y^{\beta}}\leq K^{m}_{2,\beta}(t)t^{\beta_{0}-\beta},
\end{equation}
\begin{equation}
\|\theta^{m}(t)\|_{Z^{\gamma}}\leq K^{m}_{3,\gamma}(t)t^{\gamma_{0}-\gamma}
\end{equation}
for any $0<t\leq T$, $K^{m}_{1,\alpha}\leq K^{m+1}_{1,\alpha}$, $K^{m}_{2,\beta}\leq K^{m+1}_{2,\beta}$, $K^{m}_{3,\gamma}\leq K^{m+1}_{3,\gamma}$ on $[0,T]$.
\end{lemma}
\begin{proof}
We give the inductive definition of $K^{m}_{1,\alpha}$, $K^{m}_{2,\beta}$ and $K^{m}_{3,\gamma}$ with respect to $m$.
$K^{0}_{1,\alpha}$, $K^{0}_{2,\beta}$ and $K^{0}_{3,\gamma}$ are defined as
\begin{equation}
K^{0}_{1,\alpha}(t)=\sup_{0<s\leq t}s^{\alpha-\alpha_{0}}\|u^{0}(s)\|_{X^{\alpha}},
\end{equation}
\begin{equation}
K^{0}_{2,\beta}(t)=\sup_{0<s\leq t}s^{\beta-\beta_{0}}\|\omega^{0}(s)\|_{Y^{\beta}},
\end{equation}
\begin{equation}
K^{0}_{3,\gamma}(t)=\sup_{0<s\leq t}s^{\gamma-\gamma_{0}}\|\theta^{0}(s)\|_{Z^{\gamma}}.
\end{equation}
It is obvious from (3.6)--(3.8) that (3.3)--(3.5) with $m=0$ hold for any $0<t\leq T$.
Moreover, Lemma 2.3 yields that $K^{0}_{1,\alpha}(0)=0$, $K^{0}_{2,\beta}(0)=0$, $K^{0}_{3,\gamma}(0)=0$.
Assume that there exist $K^{m}_{1,\alpha}$, $K^{m}_{2,\beta}$ and $K^{m}_{3,\gamma}$ for some $m \in \mathbb{Z}$, $m\geq 0$.
After applying (2.52)--(2.54) to $(u^{m},\omega^{m},\theta^{m})$, it is derived from (3.2) that we have the following inequalities:
\begin{equation*}
\begin{split}
\|u^{m+1}(t)\|_{X^{\alpha}}\leq& K^{0}_{1,\alpha}(t)t^{\alpha_{0}-\alpha} \\
&+C_{A,\alpha+\delta_{1},\lambda}C_{1}B(1-(\alpha+\delta_{1}),1+2(\alpha_{0}-\alpha_{1}))K^{m}_{1,\alpha_{1}}(t)^{2}t^{1+2\alpha_{0}-\alpha-2\alpha_{1}-\delta_{1}} \\
&+2C_{A,\alpha+\delta_{1},\lambda}C_{5}\mu_{r}B(1-(\alpha+\delta_{1}),1+\beta_{0}-\beta_{1})K^{m}_{2,\beta_{1}}(t)t^{1+\beta_{0}-\alpha-\beta_{1}-\delta_{1}} \\
&+C_{A,\alpha,\lambda}C_{8}L_{f}B(1-\alpha,1+\gamma_{0}-\gamma_{1})K^{m}_{3,\gamma_{1}}(t)t^{1+\gamma_{0}-\alpha-\gamma_{1}},
\end{split}
\end{equation*}
\begin{equation*}
\begin{split}
\|\omega^{m+1}(t)\|_{Y^{\beta}}\leq& K^{0}_{2,\beta}(t)t^{\beta_{0}-\beta} \\
&+C_{\Gamma,\beta+\delta_{2},\lambda}C_{2}B(1-(\beta+\delta_{2}),1+\alpha_{0}+\beta_{0}-\alpha_{2}-\beta_{2}) \\
&\times K^{m}_{1,\alpha_{2}}(t)K^{m}_{2,\beta_{2}}(t)t^{1+\alpha_{0}+\beta_{0}-\beta-\alpha_{2}-\beta_{2}-\delta_{2}} \\
&+4C_{\Gamma,\beta,\lambda}C_{6}\mu_{r}B(1-\beta,1+\beta_{0}-\beta_{2})K^{m}_{2,\beta_{2}}(t)t^{1+\beta_{0}-\beta-\beta_{2}} \\
&+2C_{\Gamma,\beta+\delta_{2},\lambda}C_{7}\mu_{r}B(1-(\beta+\delta_{2}),1+\alpha_{0}-\alpha_{2})K^{m}_{1,\alpha_{2}}(t)t^{1+\alpha_{0}-\beta-\alpha_{2}-\delta_{2}} \\
&+C_{\Gamma,\beta,\lambda}C_{9}L_{g}B(1-\beta,1+\gamma_{0}-\gamma_{2})K^{m}_{3,\gamma_{2}}(t)t^{1+\gamma_{0}-\beta-\gamma_{2}},
\end{split}
\end{equation*}
\begin{equation*}
\begin{split}
\|\theta^{m+1}(t)\|_{Z^{\gamma}}\leq& K^{0}_{3,\gamma}(t)t^{\gamma_{0}-\gamma} \\
&+C_{B,\gamma+\delta_{3},\lambda}C_{3}B(1-(\gamma+\delta_{3}),1+\alpha_{0}+\gamma_{0}-\alpha_{3}-\gamma_{3}) \\
&\times K^{m}_{1,\alpha_{3}}(t)K^{m}_{3,\gamma_{3}}(t)t^{1+\alpha_{0}+\gamma_{0}-\gamma-\alpha_{3}-\gamma_{3}-\delta_{3}} \\
&+C_{B,\beta,\lambda}C_{4}(1+\mu_{r})B(1-\gamma,1+2(\alpha_{0}-\alpha_{3}))K^{m}_{1,\alpha_{3}}(t)^{2}t^{1+2\alpha_{0}-\gamma-2\alpha_{3}} \\
&+2C_{B,\beta,\lambda}C_{4}(1+\mu_{r})B(1-\gamma,1+\alpha_{0}+\beta_{0}-\alpha_{3}-\beta_{3}) \\
&\times K^{m}_{1,\alpha_{3}}(t)K^{m}_{2,\beta_{3}}(t)t^{1+\alpha_{0}+\beta_{0}-\gamma-\alpha_{3}-\beta_{3}} \\
&+C_{B,\beta,\lambda}C_{4}(1+\mu_{r})B(1-\gamma,1+2(\beta_{0}-\beta_{3}))K^{m}_{2,\beta_{3}}(t)^{2}t^{1+2\beta_{0}-\gamma-2\beta_{3}}
\end{split}
\end{equation*}
for any $0<t\leq T$, where $B(x,y)$ is the beta function.
Therefore, $K^{m+1}_{1,\alpha}$, $K^{m+1}_{2,\beta}$ and $K^{m+1}_{3,\gamma}$ can be defined as
\begin{equation}
\begin{split}
K^{m+1}_{1,\alpha}(t)=&K^{0}_{1,\alpha}(t) \\
&+C_{A,\alpha+\delta_{1},\lambda}C_{1}B(1-(\alpha+\delta_{1}),1+2(\alpha_{0}-\alpha_{1}))K^{m}_{1,\alpha_{1}}(t)^{2}t^{1+\alpha_{0}-2\alpha_{1}-\delta_{1}} \\
&+2C_{A,\alpha+\delta_{1},\lambda}C_{5}\mu_{r}B(1-(\alpha+\delta_{1}),1+\beta_{0}-\beta_{1})K^{m}_{2,\beta_{1}}(t)t^{1+\beta_{0}-\alpha_{0}-\beta_{1}-\delta_{1}} \\
&+C_{A,\alpha,\lambda}C_{8}L_{f}B(1-\alpha,1+\gamma_{0}-\gamma_{1})K^{m}_{3,\gamma_{1}}(t)t^{1+\gamma_{0}-\alpha_{0}-\gamma_{1}},
\end{split}
\end{equation}
\begin{equation}
\begin{split}
K^{m+1}_{2,\beta}(t)=&K^{0}_{2,\beta}(t) \\
&+C_{\Gamma,\beta+\delta_{2},\lambda}C_{2}B(1-(\beta+\delta_{2}),1+\alpha_{0}+\beta_{0}-\alpha_{2}-\beta_{2}) \\
&\times K^{m}_{1,\alpha_{2}}(t)K^{m}_{2,\beta_{2}}(t)t^{1+\alpha_{0}-\alpha_{2}-\beta_{2}-\delta_{2}} \\
&+4C_{\Gamma,\beta,\lambda}C_{6}\mu_{r}B(1-\beta,1+\beta_{0}-\beta_{2})K^{m}_{2,\beta_{2}}(t)t^{1-\beta_{2}} \\
&+2C_{\Gamma,\beta+\delta_{2},\lambda}C_{7}\mu_{r}B(1-(\beta+\delta_{2}),1+\alpha_{0}-\alpha_{2})K^{m}_{1,\alpha_{2}}(t)t^{1+\alpha_{0}-\beta_{0}-\alpha_{2}-\delta_{2}} \\
&+C_{\Gamma,\beta,\lambda}C_{9}L_{g}B(1-\beta,1+\gamma_{0}-\gamma_{2})K^{m}_{3,\gamma_{2}}(t)t^{1+\gamma_{0}-\beta_{0}-\gamma_{2}},
\end{split}
\end{equation}
\begin{equation}
\begin{split}
K^{m+1}_{3,\gamma}(t)=&K^{0}_{3,\gamma}(t) \\
&+C_{B,\gamma+\delta_{3},\lambda}C_{3}B(1-(\gamma+\delta_{3}),1+\alpha_{0}+\gamma_{0}-\alpha_{3}-\gamma_{3}) \\
&\times K^{m}_{1,\alpha_{3}}(t)K^{m}_{3,\gamma_{3}}(t)t^{1+\alpha_{0}-\alpha_{3}-\gamma_{3}-\delta_{3}} \\
&+C_{B,\beta,\lambda}C_{4}(1+\mu_{r})B(1-\gamma,1+2(\alpha_{0}-\alpha_{3}))K^{m}_{1,\alpha_{3}}(t)^{2}t^{1+2\alpha_{0}-\gamma_{0}-2\alpha_{3}} \\
&+2C_{B,\beta,\lambda}C_{4}(1+\mu_{r})B(1-\gamma,1+\alpha_{0}+\beta_{0}-\alpha_{3}-\beta_{3}) \\
&\times K^{m}_{1,\alpha_{3}}(t)K^{m}_{2,\beta_{3}}(t)t^{1+\alpha_{0}+\beta_{0}-\gamma_{0}-\alpha_{3}-\beta_{3}} \\
&+C_{B,\beta,\lambda}C_{4}(1+\mu_{r})B(1-\gamma,1+2(\beta_{0}-\beta_{3}))K^{m}_{2,\beta_{3}}(t)^{2}t^{1+2\beta_{0}-\gamma_{0}-2\beta_{3}}
\end{split}
\end{equation}
It follows from (3.9)--(3.11) that
\begin{equation*}
\|u^{m+1}(t)\|_{X^{\alpha}}\leq K^{m+1}_{1,\alpha}(t)t^{\alpha_{0}-\alpha},
\end{equation*}
\begin{equation*}
\|\omega^{m+1}(t)\|_{Y^{\beta}}\leq K^{m+1}_{2,\beta}(t)t^{\beta_{0}-\beta},
\end{equation*}
\begin{equation*}
\|\theta^{m+1}(t)\|_{Z^{\gamma}}\leq K^{m+1}_{3,\gamma}(t)t^{\gamma_{0}-\gamma}
\end{equation*}
for any $0<t\leq T$.
Furthermore, we utilize inductive assumptions for $K^{m}_{1,\alpha}$, $K^{m}_{2,\beta}$ and $K^{m}_{3,\gamma}$, and conclude that $K^{m+1}_{1,\alpha}(0)=0$, $K^{m+1}_{2,\beta}(0)=0$, $K^{m+1}_{3,\gamma}(0)=0$.
It can be easily seen from the induction with respect to $m$ that $K^{m}_{1,\alpha}\leq K^{m+1}_{1,\alpha}$, $K^{m}_{2,\beta}\leq K^{m+1}_{2,\beta}$, $K^{m}_{3,\gamma}\leq K^{m+1}_{3,\gamma}$ on $[0,T]$ for any $m \in \mathbb{Z}$, $m\geq 0$.
\end{proof}
We can see that a mild solution $(u,\omega,\theta)$ of (1.1), (1.2) is constructed by the following lemmas.
Set $K^{m}(t)=\max\{K^{m}_{1,\alpha}(t), K^{m}_{2,\beta}(t), K^{m}_{3,\gamma}(t) \ ; \ \alpha=\alpha_{1}, \alpha_{2}, \alpha_{3}, \beta=\beta_{1}, \beta_{2}, \beta_{3}, \gamma=\gamma_{1}, \gamma_{2}, \gamma_{3}\}$.
Then it follows from Lemma 3.1 that $K^{m}$ is a monotone increasing continuous function on $[0,T]$ satisfying $K^{m}(0)=0$, $K^{m}\leq K^{m+1}$ on $[0,T]$ for any $m \in \mathbb{Z}$, $m\geq 0$.
It is required that $C$ is independent of not only $u$, $\omega$, $\theta$ and $t$ but also $m$ throughout this subsection.
\begin{lemma}
Let $\alpha_{0}$, $\beta_{0}$ and $\gamma_{0}$ satisfy $(2.49)_{1}$, $(2.50)_{1}$, $(2.51)_{1}$.
Then there exists a positive constant $T_{1}\leq T$ depending only on $\Omega$, $p$, $q$, $r$, $\alpha_{0}$, $\beta_{0}$, $\gamma_{0}$, $u_{0}$, $\omega_{0}$, $\theta_{0}$, $\mu_{r}$, $L_{f}$, $L_{g}$ and $T$ such that $\{t^{\alpha-\alpha_{0}}u^{m}\}_{m}$, $\{t^{\beta-\beta_{0}}\omega^{m}\}_{m}$ and $\{t^{\gamma-\gamma_{0}}\theta^{m}\}_{m}$ are Cauchy sequences in $C([0,T_{1}];X^{\alpha})$ for $\alpha=\alpha_{1}, \alpha_{2}, \alpha_{3}$, in $C([0,T_{1}];Y^{\beta})$ for $\beta=\beta_{1}, \beta_{2}, \beta_{3}$ and in $C([0,T_{1}];Z^{\gamma})$ for $\gamma=\gamma_{1}, \gamma_{2}, \gamma_{3}$ respectively.
\end{lemma}
\begin{proof}
It follows from (2.48), $(2.49)_{1}$, $(2.50)_{1}$, $(2.51)_{1}$, (3.9)--(3.11) that $K^{m}$ satisfies the following inductive inequality with respect to $m$:
\begin{equation}
K^{m+1}(t)\leq K^{0}(t)+CK^{m}(t)^{2}+CK^{m}(t)t^{x}
\end{equation}
for any $0<t\leq T$, $m \in \mathbb{Z}$, $m\geq 0$, where $x=\min\{1+\beta_{0}-\alpha_{0}-\beta_{1}-\delta_{1}, 1+\alpha_{0}-\beta_{0}-\alpha_{2}-\delta_{2}, 1-\beta_{2}, 1+\gamma_{0}-\alpha_{0}-\gamma_{1}, 1+\gamma_{0}-\beta_{0}-\gamma_{2}\}$.
Since $K^{0}(0)=0$, $K^{m}\leq K^{m+1}$ on $[0,T]$, $x>0$, an elementary calculation shows that there exists a positive constant $\tau_{1}\leq T$ such that $K^{m}\leq CK^{0}$ on $[0,\tau_{1}]$.
Therefore, we can utilize (3.3)--(3.5) to obtain the following inequalities:
\begin{equation}
\max_{\alpha=\alpha_{1}, \alpha_{2}, \alpha_{3}}\{t^{\alpha-\alpha_{0}}\|u^{m}(t)\|_{X^{\alpha}}\}\leq CK^{0}(t),
\end{equation}
\begin{equation}
\max_{\beta=\beta_{1}, \beta_{2}, \beta_{3}}\{t^{\beta-\beta_{0}}\|\omega^{m}(t)\|_{Y^{\beta}}\}\leq CK^{0}(t),
\end{equation}
\begin{equation}
\max_{\gamma=\gamma_{1}, \gamma_{2}, \gamma_{3}}\{t^{\gamma-\gamma_{0}}\|\theta^{m}(t)\|_{Z^{\gamma}}\}\leq CK^{0}(t)
\end{equation}
for any $0<t\leq \tau_{1}$.
It is sufficient for the conclusion that we give $X^{\alpha}$-estimates for $u^{m+1}-u^{m}$, $Y^{\beta}$-estimates for $\omega^{m+1}-\omega^{m}$ and $Z^{\gamma}$-estimates for $\theta^{m+1}-\theta^{m}$.
It can be easily seen from (2.52)--(2.54), (3.2) with $m=0$ that
\begin{equation*}
\max_{\alpha=\alpha_{1}, \alpha_{2}, \alpha_{3}}\{t^{\alpha-\alpha_{0}}\|(u^{1}-u^{0})(t)\|_{X^{\alpha}}\}\leq CK^{0}(t)(K^{0}(t)+1),
\end{equation*}
\begin{equation*}
\max_{\beta=\beta_{1}, \beta_{2}, \beta_{3}}\{t^{\beta-\beta_{0}}\|(\omega^{1}-\omega^{0})(t)\|_{Y^{\beta}}\}\leq CK^{0}(t)(K^{0}(t)+1),
\end{equation*}
\begin{equation*}
\max_{\gamma=\gamma_{1}, \gamma_{2}, \gamma_{3}}\{t^{\gamma-\gamma_{0}}\|(\theta^{1}-\theta^{0})(t)\|_{Z^{\gamma}}\}\leq CK^{0}(t)^{2}
\end{equation*}
for any $0<t\leq \tau_{1}$.
We utilize (2.55)--(2.57), (3.13)--(3.15) and the induction with respect to $m$, and obtain the following inequalities:
\begin{equation}
\max_{\alpha=\alpha_{1}, \alpha_{2}, \alpha_{3}}\{t^{\alpha-\alpha_{0}}\|(u^{m+1}-u^{m})(t)\|_{X^{\alpha}}\}\leq CK^{0}(t)(K^{0}(t)+1)\{C(K^{0}(t)+t^{a})\}^{m},
\end{equation}
\begin{equation}
\max_{\beta=\beta_{1}, \beta_{2}, \beta_{3}}\{t^{\beta-\beta_{0}}\|(\omega^{m+1}-\omega^{m})(t)\|_{Y^{\beta}}\}\leq CK^{0}(t)(K^{0}(t)+1)\{C(K^{0}(t)+t^{b})\}^{m},
\end{equation}
\begin{equation}
\max_{\gamma=\gamma_{1}, \gamma_{2}, \gamma_{3}}\{t^{\gamma-\gamma_{0}}\|(\theta^{m+1}-\theta^{m})(t)\|_{Z^{\gamma}}\}\leq CK^{0}(t)(K^{0}(t)+1)(CK^{0}(t))^{m}
\end{equation}
for any $0<t\leq \tau_{1}$, where $a=\min\{1+\beta_{0}-\alpha_{0}-\beta_{1}-\delta_{1}, 1+\gamma_{0}-\alpha_{0}-\gamma_{1}\}$, $b=\min\{1+\alpha_{0}-\beta_{0}-\alpha_{2}-\delta_{2}, 1-\beta_{2}, 1+\gamma_{0}-\beta_{0}-\gamma_{2}\}$.
Since $K^{0}(0)=0$, $a>0$, $b>0$, we can take a positive constant $T_{1}\leq \tau_{1}$ satisfying $C(K^{0}(T_{1})+T^{a}_{1})<1$, $C(K^{0}(T_{1})+T^{b}_{1})<1$, $CK^{0}(T_{1})<1$.
Then $\{t^{\alpha-\alpha_{0}}u^{m}\}_{m}$, $\{t^{\beta-\beta_{0}}\omega^{m}\}_{m}$ and $\{t^{\gamma-\gamma_{0}}\theta^{m}\}_{m}$ are Cauchy sequences in $C([0,T_{1}];X^{\alpha})$, in $C([0,T_{1}];Y^{\beta})$ and in $C([0,T_{1}];Z^{\gamma})$ respectively.
\end{proof}
\begin{lemma}
Let $\alpha_{0}$, $\beta_{0}$ and $\gamma_{0}$ satisfy $(2.49)_{2}$, $(2.50)_{1}$, $(2.51)_{1}$ or $(2.49)_{1}$, $(2.50)_{2}$, $(2.51)_{1}$ or $(2.49)_{1}$, $(2.50)_{1}$, $(2.51)_{2}$ or $(2.49)_{2}$, $(2.50)_{2}$, $(2.51)_{1}$ or $(2.49)_{2}$, $(2.50)_{1}$, $(2.51)_{2}$ or $(2.49)_{1}$, $(2.50)_{2}$, $(2.51)_{2}$ or $(2.49)_{2}$, $(2.50)_{2}$, $(2.51)_{2}$.
Then there exists a positive constant $T_{2}\leq T$ depending only on $\Omega$, $p$, $q$, $r$, $\alpha_{0}$, $\beta_{0}$, $\gamma_{0}$, $u_{0}$, $\omega_{0}$, $\theta_{0}$, $\mu_{r}$, $L_{f}$, $L_{g}$ and $T$ such that $\{t^{\alpha-\alpha_{0}}u^{m}\}_{m}$, $\{t^{\beta-\beta_{0}}\omega^{m}\}_{m}$ and $\{t^{\gamma-\gamma_{0}}\theta^{m}\}_{m}$ are Cauchy sequences in $C([0,T_{2}];X^{\alpha})$ for $\alpha=\alpha_{1}, \alpha_{2}, \alpha_{3}$, in $C([0,T_{2}];Y^{\beta})$ for $\beta=\beta_{1}, \beta_{2}, \beta_{3}$ and in $C([0,T_{2}];Z^{\gamma})$ for $\gamma=\gamma_{1}, \gamma_{2}, \gamma_{3}$ respectively.
\end{lemma}
\begin{proof}
It is clear from $(2.49)_{2}$, $(2.50)_{2}$, $(2.51)_{2}$ that $1+\beta_{0}-\alpha_{0}-\beta_{1}-\delta_{1}=0$, $1+\gamma_{0}-\alpha_{0}-\gamma_{1}=0$, $1+\gamma_{0}-\alpha_{0}-\gamma_{2}=0$ respectively.
We must consider, instead of (3.12), the following inductive inequality:
\begin{equation}
\begin{split}
&K^{m+1}_{1,\alpha}(t)\leq K^{0}_{1,\alpha}(t)+CK^{m}_{1,\alpha_{1}}(t)^{2}+CK^{m}_{2,\beta_{1}}(t)+CK^{m}_{3,\gamma_{1}}(t), \\
&K^{m+1}_{2,\beta}(t)\leq K^{0}_{2,\beta}(t)+CK^{m}_{1,\alpha_{2}}(t)K^{m}_{2,\beta_{2}}(t)+C(K^{m}_{2,\beta_{2}}(t)t^{1-\beta_{2}}+K^{m}_{1,\alpha_{2}}(t)t^{1+\alpha_{0}-\beta_{0}-\alpha_{2}-\delta_{2}}+K^{m}_{3,\gamma_{2}}(t)), \\
&K^{m+1}_{3,\gamma}(t)\leq K^{0}_{3,\gamma}(t)+C(K^{m}_{1,\alpha_{3}}(t)K^{m}_{3,\gamma_{3}}(t)+K^{m}_{1,\alpha_{3}}(t)^{2}+K^{m}_{1,\alpha_{3}}(t)K^{m}_{2,\beta_{3}}(t)+K^{m}_{2,\beta_{3}}(t)^{2})
\end{split}
\end{equation}
for any $0<t\leq T$, $m \in \mathbb{Z}$, $m\geq 0$.
It can be easily seen from (3.19) that
\begin{equation}
K^{m+2}(t)\leq C(K^{0}(t)+K^{m+1}(t)^{2}+K^{m}(t)^{2})+CK^{m+1}(t)t^{y}
\end{equation}
for any $0<t\leq T$, where $y=\min\{1+\alpha_{0}-\beta_{0}-\alpha_{2}-\delta_{2}, 1-\beta_{2}\}$.
Since $K^{0}(0)=0$, $K^{m}\leq K^{m+1}\leq K^{m+2}$ on $[0,T]$, $y>0$, an elementary calculation shows that there exists a positive constant $\tau_{2}\leq T$ such that $K^{m}\leq CK^{0}$ on $[0,\tau_{2}]$.
It remains to give $X^{\alpha}$-estimates for $u^{m+1}-u^{m}$, $Y^{\beta}$-estimates for $\omega^{m+1}-\omega^{m}$ and $Z^{\gamma}$-estimates for $\theta^{m+1}-\theta^{m}$.
It follows from (2.55)--(2.57) that
\begin{equation}
\max_{\alpha=\alpha_{1}, \alpha_{2}, \alpha_{3}}\{t^{\alpha-\alpha_{0}}\|(u^{m+1}-u^{m})(t)\|_{X^{\alpha}}\}\leq CK^{0}(t)(K^{0}(t)+1)L^{m}_{1}(t),
\end{equation}
\begin{equation}
\max_{\beta=\beta_{1}, \beta_{2}, \beta_{3}}\{t^{\beta-\beta_{0}}\|(\omega^{m+1}-\omega^{m})(t)\|_{Y^{\beta}}\}\leq CK^{0}(t)(K^{0}(t)+1)L^{m}_{2}(t),
\end{equation}
\begin{equation}
\max_{\gamma=\gamma_{1}, \gamma_{2}, \gamma_{3}}\{t^{\gamma-\gamma_{0}}\|(\theta^{m+1}-\theta^{m})(t)\|_{Z^{\gamma}}\}\leq CK^{0}(t)(K^{0}(t)+1)L^{m}_{3}(t)
\end{equation}
for any $0<t\leq \tau_{2}$, where $L^{m}_{1}(t)$, $L^{m}_{2}(t)$ and $L^{m}_{3}(t)$ are defined as
\begin{equation*}
L^{0}(t)=\begin{pmatrix} 1 \\ 1 \\ 1 \end{pmatrix}, \
L^{m}(t)=\begin{pmatrix} L^{m}_{1}(t) \\ L^{m}_{2}(t) \\ L^{m}_{3}(t) \end{pmatrix}, \ L^{m+1}(t)=K(t)L^{m}(t),
\end{equation*}
\begin{equation*}
K(t)=\begin{pmatrix} K^{0}(t) & 1 & 1 \\ K^{0}(t)+t^{1+\alpha_{0}-\beta_{0}-\alpha_{2}-\delta_{2}} & K^{0}(t)+t^{1-\beta_{2}} & 1 \\ K^{0}(t) & K^{0}(t) & K^{0}(t) \end{pmatrix}.
\end{equation*}
Let $\lambda=\lambda(t)$ be a eigenvalue of $K=K(t)$.
Since $\lambda$ is a continuous function on $[0,\tau_{2}]$ satisfying $\lambda(0)=0$, we can take a positive constant $T_{2}\leq \tau_{2}$ satisfying $|\lambda(t)|<1$ for any $0<t\leq T_{2}$.
Then $\{t^{\alpha-\alpha_{0}}u^{m}\}_{m}$, $\{t^{\beta-\beta_{0}}\omega^{m}\}_{m}$ and $\{t^{\gamma-\gamma_{0}}\theta^{m}\}_{m}$ are Cauchy sequences in $C([0,T_{2}];X^{\alpha})$, in $C([0,T_{2}];Y^{\beta})$ and in $C([0,T_{2}];Z^{\gamma})$ respectively.
\end{proof}
Set $T_{*}=\min\{T_{1}, T_{2}\}$.
Then it follows from Lemmas 3.2 and 3.3 that there exists a pair of three functions $(u,\omega,\theta)$ satisfying
\begin{equation*}
u \in C((0,T_{*}];X^{\alpha_{0}}),
\end{equation*}
\begin{equation*}
\omega \in C((0,T_{*}];Y^{\beta_{0}}),
\end{equation*}
\begin{equation*}
\theta \in C((0,T_{*}];Z^{\gamma_{0}})
\end{equation*}
such that
\begin{equation}
\begin{split}
&t^{\alpha-\alpha_{0}}u^{m}\rightarrow t^{\alpha-\alpha_{0}}u & \mathrm{in} \ C([0,T_{*}];X^{\alpha}) \ \mathrm{as} \ m\rightarrow \infty, \\
&t^{\beta-\beta_{0}}\omega^{m}\rightarrow t^{\beta-\beta_{0}}\omega & \mathrm{in} \ C([0,T_{*}];Y^{\beta}) \ \mathrm{as} \ m\rightarrow \infty, \\
&t^{\gamma-\gamma_{0}}\theta^{m}\rightarrow t^{\gamma-\gamma_{0}}\theta & \mathrm{in} \ C([0,T_{*}];Z^{\gamma}) \ \mathrm{as} \ m\rightarrow \infty
\end{split}
\end{equation}
for $\alpha=\alpha_{1}, \alpha_{2}, \alpha_{3}$, $\beta=\beta_{1}, \beta_{2}, \beta_{3}$, $\gamma=\gamma_{1}, \gamma_{2}, \gamma_{3}$.
By applying the dominated convergence theorem to (3.2), we can conclude that $(u,\omega,\theta)$ satisfies (II) in $(0,T_{*}]$.

\subsection{$X^{\alpha}\times Y^{\beta}\times Z^{\gamma}$-estimates for local mild solutions}
We will be concerned with basic properties of local mild solutions of (1.1), (1.2).
It is sufficient for (2.15)--(2.20) that we prove the following lemma:
\begin{lemma}
Let $(u,\omega,\theta)$ be a mild solution of $(1.1)$, $(1.2)$ in $(0,T_{*}]$ given by $(3.24)$.
Then
\begin{equation}
t^{\alpha-\alpha_{0}}\|u(t)-e^{-tA}u_{0}\|_{X^{\alpha}}\leq CK^{0}(t),
\end{equation}
\begin{equation}
t^{\beta-\beta_{0}}\|\omega(t)-e^{-t\Gamma}\omega_{0}\|_{Y^{\beta}}\leq CK^{0}(t),
\end{equation}
\begin{equation}
t^{\gamma-\gamma_{0}}\|\theta(t)-e^{-tB}\theta_{0}\|_{Z^{\gamma}}\leq CK^{0}(t)
\end{equation}
for any $\alpha_{0}\leq\alpha<1$, $\beta_{0}\leq\beta<1$, $\gamma_{0}\leq\gamma<1$, $0<t\leq T_{*}$, where $C$ is a positive constant independent of $u$, $\omega$, $\theta$ and $t$.
\end{lemma}
\begin{proof}
It can be easily seen that (3.13)--(3.15) with $(u,\omega,\theta)$ instead of $(u^{m},\omega^{m},\theta^{m})$ hold for any $\alpha_{0}\leq\alpha<1-\delta_{1}$, $\beta_{0}\leq\beta<1-\delta_{2}$, $\gamma_{0}\leq\gamma<1-\delta_{3}$.
By applying (3.13)--(3.15) to (2.52)--(2.54), we have the following inequalities:
\begin{equation*}
t^{\alpha-\alpha_{0}}\|u(t)-e^{-tA}u_{0}\|_{X^{\alpha}}\leq CK^{0}(t)(K^{0}(t)+1),
\end{equation*}
\begin{equation*}
t^{\beta-\beta_{0}}\|\omega(t)-e^{-t\Gamma}\omega_{0}\|_{Y^{\beta}}\leq CK^{0}(t)(K^{0}(t)+1),
\end{equation*}
\begin{equation*}
t^{\gamma-\gamma_{0}}\|\theta(t)-e^{-tB}\theta_{0}\|_{Z^{\gamma}}\leq CK^{0}(t)^{2}
\end{equation*}
for any $\alpha_{0}\leq\alpha<1-\delta_{1}$, $\beta_{0}\leq\beta<1-\delta_{2}$, $\gamma_{0}\leq\gamma<1-\delta_{3}$, $0<t\leq T_{*}$.
Furthermore, the choice of $\delta_{1}$, $\delta_{2}$ and $\delta_{3}$ allows us to assume that $\alpha_{0}\leq \alpha<1$, $\beta_{0}\leq \beta<1$, $\gamma_{0}\leq\gamma<1$.
These inequalities clearly yield (3.25)--(3.27).
\end{proof}
It follows from (3.25) with $\alpha=\alpha_{0}$, (3.26) with $\beta=\beta_{0}$, (3.27) with $\gamma=\gamma_{0}$ that
\begin{equation}
\|u(t)-u_{0}\|_{X^{\alpha_{0}}}\leq \|(e^{-tA}-I)u_{0}\|_{X^{\alpha_{0}}}+CK^{0}(t),
\end{equation}
\begin{equation}
\|\omega(t)-\omega_{0}\|_{Y^{\beta_{0}}}\leq \|(e^{-t\Gamma}-I)\omega_{0}\|_{Y^{\beta_{0}}}+CK^{0}(t),
\end{equation}
\begin{equation}
\|\theta(t)-\theta_{0}\|_{Z^{\gamma_{0}}}\leq \|(e^{-tB}-I)\theta_{0}\|_{Z^{\gamma_{0}}}+CK^{0}(t)
\end{equation}
for any $0<t\leq T_{*}$.
By taking $t$ as zero, (3.28)--(3.30) imply that $u(0)=u_{0}$, $\omega(0)=\omega_{0}$, $\theta(0)=\theta_{0}$, consequently, $(u,\omega,\theta)$ is a mild solution of (1.1), (1.2) on $[0,T_{*}]$.
It is obvious from (3.25)--(3.27) that
\begin{equation}
t^{\alpha-\alpha_{0}}\|u(t)\|_{X^{\alpha}}\leq t^{\alpha-\alpha_{0}}\|e^{-tA}u_{0}\|_{X^{\alpha}}+CK^{0}(t),
\end{equation}
\begin{equation}
t^{\beta-\beta_{0}}\|\omega(t)\|_{Y^{\beta}}\leq t^{\beta-\beta_{0}}\|e^{-t\Gamma}\omega_{0}\|_{Y^{\beta}}+CK^{0}(t),
\end{equation}
\begin{equation}
t^{\gamma-\gamma_{0}}\|\theta(t)\|_{Z^{\gamma}}\leq t^{\gamma-\gamma_{0}}\|e^{-tB}\theta_{0}\|_{Z^{\gamma}}+CK^{0}(t)
\end{equation}
for any $\alpha_{0}\leq\alpha<1$, $\beta_{0}\leq\beta<1$, $\gamma_{0}\leq\gamma<1$, $0<t\leq T_{*}$.
(2.1)--(2.3), (3.31)--(3.33) clearly lead to (2.15)--(2.17).
Moreover, it can be easily seen from (2.7)--(2.9), (3.31)--(3.33) that (2.18)--(2.20) hold for any $\alpha_{0}<\alpha<1$, $\beta_{0}<\beta<1$, $\gamma_{0}<\gamma<1$.

\subsection{Uniqueness of mild solutions}
We proceed to the uniqueness of mild solutions of (1.1), (1.2) on $[0,T]$.
Throughout this subsection, it is required that $C$ is a positive constant independent of $t$, consequently, $C$ may depend on $u$, $\omega$ and $\theta$.
For any $\alpha_{0}\leq \alpha<1$, $\beta_{0}\leq \beta<1$, $\gamma_{0}\leq \gamma<1$, $0<\tau\leq t\leq T$, let us introduce the following notation:
\begin{equation*}
\|u(t)\|_{(\alpha)}=t^{\alpha-\alpha_{0}}\|u(t)\|_{X^{\alpha}},
\end{equation*}
\begin{equation*}
\|\omega(t)\|_{(\beta)}=t^{\beta-\beta_{0}}\|\omega(t)\|_{Y^{\beta}},
\end{equation*}
\begin{equation*}
\|\theta(t)\|_{(\gamma)}=t^{\gamma-\gamma_{0}}\|\theta(t)\|_{Z^{\gamma}},
\end{equation*}
\begin{equation*}
\|(u,\omega,\theta)(t)\|_{(\alpha,\beta,\gamma)}=\|u(t)\|_{(\alpha)}+\|\omega(t)\|_{(\beta)}+\|\theta(t)\|_{(\gamma)},
\end{equation*}
\begin{equation*}
\|u\|_{(\alpha;t)}=\sup_{0<s\leq t}s^{\alpha-\alpha_{0}}\|u(s)\|_{X^{\alpha}},
\end{equation*}
\begin{equation*}
\|\omega\|_{(\beta;t)}=\sup_{0<s\leq t}s^{\beta-\beta_{0}}\|\omega(s)\|_{Y^{\beta}},
\end{equation*}
\begin{equation*}
\|\theta\|_{(\gamma;t)}=\sup_{0<s\leq t}s^{\gamma-\gamma_{0}}\|\theta(s)\|_{Z^{\gamma}},
\end{equation*}
\begin{equation*}
\|(u,\omega,\theta)\|_{(\alpha,\beta,\gamma;t)}=\|u\|_{(\alpha;t)}+\|\omega\|_{(\beta;t)}+\|\theta\|_{(\gamma;t)},
\end{equation*}
\begin{equation*}
\|u\|_{(\alpha;\tau,t)}=\max_{\tau\leq s\leq t}\|u(s)\|_{X^{\alpha}},
\end{equation*}
\begin{equation*}
\|\omega\|_{(\beta;\tau,t)}=\max_{\tau\leq s\leq t}\|\omega(s)\|_{Y^{\beta}},
\end{equation*}
\begin{equation*}
\|\theta\|_{(\gamma;\tau,t)}=\max_{\tau\leq s\leq t}\|\theta(s)\|_{Z^{\gamma}},
\end{equation*}
\begin{equation*}
\|(u,\omega,\theta)\|_{(\alpha,\beta,\gamma;\tau,t)}=\|u\|_{(\alpha;\tau,t)}+\|\omega\|_{(\beta;\tau,t)}+\|\theta\|_{(\gamma;\tau,t)}.
\end{equation*}
It is clear that the uniqueness is derived from the continuous dependence with respect to initial data.
We prove the following lemma:
\begin{lemma}
Let $(u,\omega,\theta)$ and $(\bar{u},\bar{\omega},\bar{\theta})$ be two mild solutions of $(1.1)$, $(1.2)$ on $[0,T]$ with initial data $(u_{0},\omega_{0},\theta_{0})$ and $(\bar{u}_{0},\bar{\omega}_{0},\bar{\theta}_{0})$ respectively which satisfy
\begin{equation*}
\mathrm{(i)} \ t^{\alpha-\alpha_{0}}u, \ t^{\alpha-\alpha_{0}}\bar{u} \in C([0,T];X^{\alpha}),
\end{equation*}
\begin{equation*}
t^{\beta-\beta_{0}}\omega, \ t^{\beta-\beta_{0}}\bar{\omega} \in C([0,T];Y^{\beta}),
\end{equation*}
\begin{equation*}
t^{\gamma-\gamma_{0}}\theta, \ t^{\gamma-\gamma_{0}}\bar{\theta} \in C([0,T];Z^{\gamma})
\end{equation*}
for any $\alpha_{0}\leq\alpha<1$, $\beta_{0}\leq\beta<1$, $\gamma_{0}\leq\gamma<1$.
\begin{equation*}
\mathrm{(ii)} \ \|u(t)\|_{X^{\alpha}}=o(t^{\alpha_{0}-\alpha}), \ \|\bar{u}(t)\|_{X^{\alpha}}=o(t^{\alpha_{0}-\alpha}) \ \mathrm{as} \ t\rightarrow +0,
\end{equation*}
\begin{equation*}
\|\omega(t)\|_{Y^{\beta}}=o(t^{\beta_{0}-\beta}), \ \|\bar{\omega}(t)\|_{Y^{\beta}}=o(t^{\beta_{0}-\beta}) \ \mathrm{as} \ t\rightarrow +0,
\end{equation*}
\begin{equation*}
\|\theta(t)\|_{Z^{\gamma}}=o(t^{\gamma_{0}-\gamma}), \ \|\bar{\theta}(t)\|_{Z^{\gamma}}=o(t^{\gamma_{0}-\gamma}) \ \mathrm{as} \ t\rightarrow +0
\end{equation*}
for any $\alpha_{0}<\alpha<1$, $\beta_{0}<\beta<1$, $\gamma_{0}<\gamma<1$.
Then
\begin{equation}
\|(u-\bar{u},\omega-\bar{\omega},\theta-\bar{\theta})(t)\|_{(\alpha,\beta,\gamma)}\leq C\|(u_{0}-\bar{u}_{0},\omega_{0}-\bar{\omega}_{0},\theta_{0}-\bar{\theta}_{0})\|_{(\alpha_{0},\beta_{0},\gamma_{0})}
\end{equation}
for any $\alpha_{0}\leq\alpha<1$, $\beta_{0}\leq\beta<1$, $\gamma_{0}\leq\gamma<1$, $0<t\leq T$, where $C$ is a positive constant independent of $t$.
\end{lemma}
\begin{proof}
$D_{0}$, $D$ and $M$ are defined as
\begin{equation*}
D_{0}=\|(u_{0}-\bar{u}_{0},\omega_{0}-\bar{\omega}_{0},\theta_{0}-\bar{\theta}_{0})\|_{(\alpha_{0},\beta_{0},\gamma_{0})},
\end{equation*}
\begin{equation*}
\begin{split}
D(t)=\max\{&\|(u-\bar{u},\omega-\bar{\omega},\theta-\bar{\theta})\|_{(\alpha_{1},\beta_{1},\gamma_{1};t)}, \|(u-\bar{u},\omega-\bar{\omega},\theta-\bar{\theta})\|_{(\alpha_{2},\beta_{2},\gamma_{2};t)}, \\
&\|(u-\bar{u},\omega-\bar{\omega},\theta-\bar{\theta})\|_{(\alpha_{3},\beta_{3},\gamma_{3};t)}\},
\end{split}
\end{equation*}
\begin{equation*}
\begin{split}
M(t)=\max\{&\|u\|_{(\alpha_{1};t)}, \|u\|_{(\alpha_{3};t)}, \|\bar{u}\|_{(\alpha_{1};t)}, \|\bar{u}\|_{(\alpha_{2};t)}, \|\bar{u}\|_{(\alpha_{3};t)}, \\
&\|\omega\|_{(\beta_{2};t)}, \|\omega\|_{(\beta_{3};t)}, \|\bar{\omega}\|_{(\beta_{3};t)}, \|\theta\|_{(\gamma_{3};t)}\}.
\end{split}
\end{equation*}
By applying (2.55)--(2.57) to $(u,\omega,\theta)$ and $(\bar{u},\bar{\omega},\bar{\theta})$, we have the following inequalities:
\begin{equation}
\begin{split}
\|(u-\bar{u})(t)\|_{(\alpha)}\leq& C\|u_{0}-\bar{u}_{0}\|_{(\alpha_{0})}+CM(t)\|u-\bar{u}\|_{(\alpha_{1};t)}+Ct^{1+\beta_{0}-\alpha_{0}-\beta_{1}-\delta_{1}}\|\omega-\bar{\omega}\|_{(\beta_{1};t)} \\
&+Ct^{1+\gamma_{0}-\alpha_{0}-\gamma_{1}}\|\theta-\bar{\theta}\|_{(\gamma_{1};t)},
\end{split}
\end{equation}
\begin{equation}
\begin{split}
\|(\omega-\bar{\omega})(t)\|_{(\beta)}\leq& C\|\omega_{0}-\bar{\omega}_{0}\|_{(\beta_{0})}+CM(t)\|(u-\bar{u},\omega-\bar{\omega},0)\|_{(\alpha_{2},\beta_{2},\gamma_{0};t)}+Ct^{1-\beta_{2}}\|\omega-\bar{\omega}\|_{(\beta_{2};t)} \\
&+Ct^{1+\alpha_{0}-\beta_{0}-\alpha_{2}-\delta_{2}}\|u-\bar{u}\|_{(\alpha_{2};t)}+Ct^{1+\gamma_{0}-\beta_{0}-\gamma_{2}}\|\theta-\bar{\theta}\|_{(\gamma_{2};t)},
\end{split}
\end{equation}
\begin{equation}
\|(\theta-\bar{\theta})(t)\|_{(\gamma)}\leq C\|\theta_{0}-\bar{\theta}_{0}\|_{(\gamma_{0})}+CM(t)\|(u-\bar{u},\omega-\bar{\omega},\theta-\bar{\theta})\|_{(\alpha_{3},\beta_{3},\gamma_{3};t)}
\end{equation}
for any $\alpha_{0}\leq\alpha<1-\delta_{1}$, $\beta_{0}\leq\beta<1-\delta_{2}$, $\gamma_{0}\leq\gamma<1-\delta_{3}$, $0<t\leq T$.
Moreover, it can be easily seen from (3.35) with $\alpha=\alpha_{1}, \alpha_{2}, \alpha_{3}$, (3.36) with $\beta=\beta_{1}, \beta_{2}, \beta_{3}$, (3.37) with $\gamma=\gamma_{1}, \gamma_{2}, \gamma_{3}$ that
\begin{equation}
D(t)\leq CD_{0}+CN(t)D(t)
\end{equation}
for any $0<t\leq T$, where
\begin{equation*}
N(t):=\begin{cases}
M(t)+t^{1+\beta_{0}-\alpha_{0}-\beta_{1}-\delta_{1}}+t^{1+\gamma_{0}-\alpha_{0}-\gamma_{1}}+t^{1-\beta_{2}} \\
+t^{1+\alpha_{0}-\beta_{0}-\alpha_{2}-\delta_{2}}+t^{1+\gamma_{0}-\beta_{0}-\gamma_{2}} & \mathrm{if} \ (2.49)_{1}, (2.50)_{1}, (2.51)_{1} \\
M(t)+t^{1+\gamma_{0}-\alpha_{0}-\gamma_{1}}+t^{1-\beta_{2}}+t^{1+\alpha_{0}-\beta_{0}-\alpha_{2}-\delta_{2}} \\
+t^{1+\gamma_{0}-\beta_{0}-\gamma_{2}} & \mathrm{if} \ (2.49)_{2}, (2.50)_{1}, (2.51)_{1}, \\
M(t)+t^{1+\beta_{0}-\alpha_{0}-\beta_{1}-\delta_{1}}+t^{1-\beta_{2}}+t^{1+\alpha_{0}-\beta_{0}-\alpha_{2}-\delta_{2}} \\
+t^{1+\gamma_{0}-\beta_{0}-\gamma_{2}} & \mathrm{if} \ (2.49)_{1}, (2.50)_{2}, (2.51)_{1}, \\
M(t)+t^{1+\beta_{0}-\alpha_{0}-\beta_{1}-\delta_{1}}+t^{1+\gamma_{0}-\alpha_{0}-\gamma_{1}}+t^{1-\beta_{2}} \\
+t^{1+\alpha_{0}-\beta_{0}-\alpha_{2}-\delta_{2}} & \mathrm{if} \ (2.49)_{1}, (2.50)_{1}, (2.51)_{2} \\
M(t)+t^{1-\beta_{2}}+t^{1+\alpha_{0}-\beta_{0}-\alpha_{2}-\delta_{2}}+t^{1+\gamma_{0}-\beta_{0}-\gamma_{2}} & \mathrm{if} \ (2.49)_{2}, (2.50)_{2}, (2.51)_{1}, \\
M(t)+t^{1+\gamma_{0}-\alpha_{0}-\gamma_{1}}+t^{1-\beta_{2}}+t^{1+\alpha_{0}-\beta_{0}-\alpha_{2}-\delta_{2}} & \mathrm{if} \ (2.49)_{2}, (2.50)_{1}, (2.51)_{2}, \\
M(t)+t^{1+\beta_{0}-\alpha_{0}-\beta_{1}-\delta_{1}}+t^{1-\beta_{2}}+t^{1+\alpha_{0}-\beta_{0}-\alpha_{2}-\delta_{2}} & \mathrm{if} \ (2.49)_{1}, (2.50)_{2}, (2.51)_{2}, \\
M(t)+t^{1-\beta_{2}}+t^{1+\alpha_{0}-\beta_{0}-\alpha_{2}-\delta_{2}} & \mathrm{if} \ (2.49)_{2}, (2.50)_{2}, (2.51)_{2}. \\
\end{cases}
\end{equation*}
Since $(u,\omega,\theta)$ and $(\bar{u},\bar{\omega},\bar{\theta})$ satisfy (i), (ii), $N$ is a monotone increasing continuous function on $[0,T]$ satisfying $N(0)=0$.
Then we can take a positive constant $\tau_{0}\leq T$ satisfying $CN(\tau_{0})<1$, consequently, $D(\tau_{0})\leq CD_{0}$.
It remains to prove (3.34) for any $\tau_{0}\leq t\leq T$.
For any $\tau_{0}\leq \tau\leq T$, $D(\tau,\cdot)$ and $M(\tau,\cdot)$ are defined as
\begin{equation*}
\begin{split}
D(\tau,t)=\max\{&\|(u-\bar{u},\omega-\bar{\omega},\theta-\bar{\theta})\|_{(\alpha_{1},\beta_{1},\gamma_{1};\tau,t)}, \|(u-\bar{u},\omega-\bar{\omega},\theta-\bar{\theta})\|_{(\alpha_{2},\beta_{2},\gamma_{2};\tau,t)}, \\
&\|(u-\bar{u},\omega-\bar{\omega},\theta-\bar{\theta})\|_{(\alpha_{3},\beta_{3},\gamma_{3};\tau,t)}\},
\end{split}
\end{equation*}
\begin{equation*}
\begin{split}
M(\tau,t)=\max\{&\|u\|_{(\alpha_{1};\tau,t)}, \|u\|_{(\alpha_{3};\tau,t)}, \|\bar{u}\|_{(\alpha_{1};\tau,t)}, \|\bar{u}\|_{(\alpha_{2};\tau,t)}, \|\bar{u}\|_{(\alpha_{3};\tau,t)}, \|\omega\|_{(\beta_{2};\tau,t)}, \\
&\|\omega\|_{(\beta_{3};\tau,t)}, \|\bar{\omega}\|_{(\beta_{3};\tau,t)}, \|\theta\|_{(\gamma_{3};\tau,t)}\}.
\end{split}
\end{equation*}
It is necessary to remark that $(u,\omega,\theta)$ and $(\bar{u},\bar{\omega},\bar{\theta})$ satisfy
\begin{equation}
\begin{split}
&u(t)=e^{-(t-\tau)A}u(\tau)+\displaystyle\int^{t}_{\tau}e^{-(t-s)A}F(u,\omega,\theta)(s)ds, \\
&\omega(t)=e^{-(t-\tau)\Gamma}\omega(\tau)+\displaystyle\int^{t}_{\tau}e^{-(t-s)\Gamma}G(u,\omega,\theta)(s)ds, \\
&\theta(t)=e^{-(t-\tau)B}\theta(\tau)+\displaystyle\int^{t}_{\tau}e^{-(t-s)B}H(u,\omega,\theta)(s)ds
\end{split}
\end{equation}
for any $\tau\leq t\leq T$.
We subtract (3.39) with $(\bar{u},\bar{\omega},\bar{\theta})$ from (3.39) with $(u,\omega,\theta)$, and obtain
\begin{equation*}
\begin{split}
\|(u-\bar{u})(t)\|_{X^{\alpha}}\leq& \|e^{-(t-\tau)A}(u-\bar{u})(\tau)\|_{X^{\alpha}} \\
&+\int^{t}_{\tau}\|e^{-(t-s)A}(F(u,\omega,\theta)-F(\bar{u},\bar{\omega},\bar{\theta}))(s)\|_{X^{\alpha}}ds,
\end{split}
\end{equation*}
\begin{equation*}
\begin{split}
\|(\omega-\bar{\omega})(t)\|_{Y^{\beta}}\leq& \|e^{-(t-\tau)\Gamma}(\omega-\bar{\omega})(\tau)\|_{Y^{\beta}} \\
&+\int^{t}_{\tau}\|e^{-(t-s)\Gamma}(G(u,\omega,\theta)-G(\bar{u},\bar{\omega},\bar{\theta}))(s)\|_{Y^{\beta}}ds,
\end{split}
\end{equation*}
\begin{equation*}
\begin{split}
\|(\theta-\bar{\theta})(t)\|_{Z^{\gamma}}\leq& \|e^{-(t-\tau)B}(\theta-\bar{\theta})(\tau)\|_{Z^{\gamma}} \\
&+\int^{t}_{\tau}\|e^{-(t-s)B}(H(u,\omega,\theta)-H(\bar{u},\bar{\omega},\bar{\theta}))(s)\|_{Z^{\gamma}}ds
\end{split}
\end{equation*}
for any $\alpha_{0}\leq\alpha<1-\delta_{1}$, $\beta_{0}\leq\beta<1-\delta_{2}$, $\gamma_{0}\leq\gamma<1-\delta_{3}$, $\tau\leq t\leq T$.
It is clear that
\begin{equation*}
\int^{t}_{\tau}\|e^{-(t-s)A}(F(u,\omega,\theta)-F(\bar{u},\bar{\omega},\bar{\theta}))(s)\|_{X^{\alpha}}ds,
\end{equation*}
\begin{equation*}
\int^{t}_{\tau}\|e^{-(t-s)\Gamma}(G(u,\omega,\theta)-G(\bar{u},\bar{\omega},\bar{\theta}))(s)\|_{Y^{\beta}}ds,
\end{equation*}
\begin{equation*}
\int^{t}_{\tau}\|e^{-(t-s)B}(H(u,\omega,\theta)-H(\bar{u},\bar{\omega},\bar{\theta}))(s)\|_{Z^{\gamma}}ds
\end{equation*}
are estimated like (2.55)--(2.57), consequently, we have the following inequalities:
\begin{equation}
\begin{split}
\|(u-\bar{u})(t)\|_{X^{\alpha}}\leq& C\tau^{\alpha_{0}-\alpha}_{0}D_{0}+C(t-\tau)^{1-(\alpha+\delta_{1})}M(\tau_{0},T)\|u-\bar{u}\|_{(\alpha_{1};\tau,t)} \\
&+C(t-\tau)^{1-(\alpha+\delta_{1})}\|\omega-\bar{\omega}\|_{(\beta_{1};\tau,t)}+C(t-\tau)^{1-\alpha}\|\theta-\bar{\theta}\|_{(\gamma_{1};\tau,t)},
\end{split}
\end{equation}
\begin{equation}
\begin{split}
\|(\omega-\bar{\omega})(t)\|_{Y^{\beta}}\leq& C\tau^{\beta_{0}-\beta}_{0}D_{0}+C(t-\tau)^{1-(\beta+\delta_{2})}M(\tau_{0},T)\|(u-\bar{u},\omega-\bar{\omega},0)\|_{(\alpha_{2},\beta_{2},\gamma_{0};\tau,t)} \\
&+C(t-\tau)^{1-\beta}\|\omega-\bar{\omega}\|_{(\beta_{2};\tau,t)}+C(t-\tau)^{1-(\beta+\delta_{2})}\|u-\bar{u}\|_{(\alpha_{2};\tau,t)} \\
&+C(t-\tau)^{1-\beta}\|\theta-\bar{\theta}\|_{(\gamma_{2};\tau,t)},
\end{split}
\end{equation}
\begin{equation}
\|(\theta-\bar{\theta})(t)\|_{Z^{\gamma}}\leq C\tau^{\gamma_{0}-\gamma}_{0}D_{0}+C\{(t-\tau)^{1-(\gamma+\delta_{3})}+(t-\tau)^{1-\gamma}\}M(\tau_{0},T)\|(u-\bar{u},\omega-\bar{\omega},\theta-\bar{\theta})\|_{(\alpha_{2},\beta_{2},\gamma_{2};\tau,t)}
\end{equation}
for any $\tau\leq t\leq T$.
Similarly to (3.38), it follows from (3.40) with $\alpha=\alpha_{1}, \alpha_{2}, \alpha_{3}$, (3.41) with $\beta=\beta_{1}, \beta_{2}, \beta_{3}$, (3.42) with $\gamma=\gamma_{1}, \gamma_{2}, \gamma_{3}$ that
\begin{equation}
D(\tau,t)\leq C(\tau^{\alpha_{0}-\alpha}_{0}+\tau^{\beta_{0}-\beta}_{0}+\tau^{\gamma_{0}-\gamma}_{0})D_{0}+CN(\tau,t)D(\tau,t)
\end{equation}
for any $\tau\leq t\leq T$, where
\begin{equation*}
\begin{split}
N(\tau,t):=&\{(t-\tau)^{1-(\alpha+\delta_{1})}+(t-\tau)^{1-(\beta+\delta_{2})}+(t-\tau)^{1-(\gamma+\delta_{3})}+(t-\tau)^{1-\gamma}\}M(\tau_{0},T) \\
&+(t-\tau)^{1-(\alpha+\delta_{1})}+(t-\tau)^{1-\alpha}+(t-\tau)^{1-(\beta+\delta_{2})}+(t-\tau)^{1-\beta}.
\end{split}
\end{equation*}
It is clear that there exists a positive constant $\tau_{1}\leq T-\tau$ independent of $\tau$ such that $CN(\tau,\tau+\tau_{1})<1$, consequently, $D(\tau,\tau+\tau_{1})\leq CD_{0}$.
We repeat to carry out the same proof as above, and obtain $D(\tau_{0},T)\leq CD_{0}$.
\end{proof}

\subsection{Existence of global mild solutions}
The main purpose of this subsection is to extend a mild solution of (1.1), (1.2) locally in time to the one globally in time.
By virtue of Theorem 2.1, it is essential for Theorem 2.2 that we obtain global $X^{\alpha}$-estimates (2.21) for $u$, global $Y^{\beta}$-estimates (2.22) for $\omega$ and global $Z^{\gamma}$-estimates (2.23) for $\theta$.
For any $0<\lambda<\Lambda_{1}$, $\lambda<\lambda_{1}<\Lambda_{1}$, $\lambda<\lambda_{2}<\min\{2\lambda, \lambda_{1}\}$, let us introduce monotone increasing continuous functions on $[0,\infty)$ defined as
\begin{equation*}
E_{1,\alpha}(t)=\sup_{0<s\leq t}m(s)^{\alpha-\alpha_{0}}e^{\lambda s}\|u(s)\|_{X^{\alpha}},
\end{equation*}
\begin{equation*}
E_{2,\beta}(t)=\sup_{0<s\leq t}m(s)^{\beta-\beta_{0}}e^{\lambda s}\|\omega(s)\|_{Y^{\beta}},
\end{equation*}
\begin{equation*}
E_{3,\gamma}(t)=\sup_{0<s\leq t}m(s)^{\gamma-\gamma_{0}}e^{\lambda_{2} s}\|\theta(s)\|_{Z^{\gamma}},
\end{equation*}
where $m(t)=\min\{t,1\}$.
It is obvious that (2.21)--(2.23) are established by proving the following lemma:
\begin{lemma}
There exist positive constants $\varepsilon_{1}$ and $\varepsilon_{2}$ depending only on $\Omega$, $p$, $q$, $r$, $\alpha_{0}$, $\beta_{0}$, $\gamma_{0}$, $L_{f}$, $L_{g}$ and $\lambda$ such that
\begin{equation}
E_{1,\alpha}(t)\leq C(\|u_{0}\|_{X^{\alpha_{0}}}+\|\omega_{0}\|_{Y^{\beta_{0}}}+\|\theta_{0}\|_{Z^{\gamma_{0}}}),
\end{equation}
\begin{equation}
E_{2,\beta}(t)\leq C(\|u_{0}\|_{X^{\alpha_{0}}}+\|\omega_{0}\|_{Y^{\beta_{0}}}+\|\theta_{0}\|_{Z^{\gamma_{0}}}),
\end{equation}
\begin{equation}
E_{3,\gamma}(t)\leq C(\|u_{0}\|_{X^{\alpha_{0}}}+\|\omega_{0}\|_{Y^{\beta_{0}}}+\|\theta_{0}\|_{Z^{\gamma_{0}}})
\end{equation}
for any $\alpha_{0}\leq\alpha<1$, $\beta_{0}\leq\beta<1$, $\gamma_{0}\leq\gamma<1$, $t>0$, where $C$ is a positive constant independent of $u$, $\omega$, $\theta$ and $t$ provided that
\begin{equation*}
\mu_{r}\leq\varepsilon_{1},
\end{equation*}
\begin{equation*}
\|u_{0}\|_{X^{\alpha_{0}}}+\|\omega_{0}\|_{Y^{\beta_{0}}}+\|\theta_{0}\|_{Z^{\gamma_{0}}}\leq\varepsilon_{2}.
\end{equation*}
\end{lemma}
\begin{proof}
It follows from $\mathrm{(II)}_{1}$, (2.52) that 
\begin{equation*}
\begin{split}
&m(t)^{\alpha-\alpha_{0}}e^{\lambda t}\|u(t)\|_{X^{\alpha}}\leq C_{A,\alpha-\alpha_{0},\lambda_{1}}e^{-(\lambda_{1}-\lambda)t}\|u_{0}\|_{X^{\alpha_{0}}} \\
&+C_{A,\alpha+\delta_{1},\lambda_{1}}C_{1}m(t)^{\alpha-\alpha_{0}}e^{-(\lambda_{1}-\lambda)t}E_{1,\alpha_{1}}(t)^{2}\int^{t}_{0}(t-s)^{-(\alpha+\delta_{1})}m(s)^{-2(\alpha_{1}-\alpha_{0})}e^{-(2\lambda-\lambda_{1})s}ds \\
&+2C_{A,\alpha+\delta_{1},\lambda_{1}}C_{5}\mu_{r}m(t)^{\alpha-\alpha_{0}}e^{-(\lambda_{1}-\lambda)t}E_{2,\beta_{1}}(t)\int^{t}_{0}(t-s)^{-(\alpha+\delta_{1})}m(s)^{-(\beta_{1}-\beta_{0})}e^{-(\lambda-\lambda_{1})s}ds \\
&+C_{A,\alpha,\lambda_{1}}C_{8}L_{f}m(t)^{\alpha-\alpha_{0}}e^{-(\lambda_{1}-\lambda)t}E_{3,\gamma_{1}}(t)\int^{t}_{0}(t-s)^{-\alpha}m(s)^{-(\gamma_{1}-\gamma_{0})}e^{-(\lambda_{2}-\lambda_{1})s}ds \\
&\leq C_{A,\alpha-\alpha_{0},\lambda_{1}}\|u_{0}\|_{X^{\alpha_{0}}}+Cm(t)^{1+\alpha_{0}-2\alpha_{1}-\delta_{1}}e^{-\lambda t}E_{1,\alpha_{1}}(t)^{2}+C\mu_{r}m(t)^{1+\beta_{0}-\alpha_{0}-\beta_{1}-\delta_{1}}E_{2,\beta_{1}}(t) \\
&+CL_{f}m(t)^{1+\gamma_{0}-\alpha_{0}-\gamma_{1}}e^{-(\lambda_{2}-\lambda)t}E_{3,\gamma_{1}}(t),
\end{split}
\end{equation*}
\begin{equation}
E_{1,\alpha}(t)\leq C(\|u_{0}\|_{X^{\alpha_{0}}}+E_{1,\alpha_{1}}(t)^{2}+\mu_{r}E_{2,\beta_{1}}(t)+L_{f}E_{3,\gamma_{1}}(t))
\end{equation}
for any $\alpha_{0}\leq \alpha<1-\delta_{1}$, $t>0$.
Similarly to (3.47), we can utilize $\mathrm{(II)}_{2}$, $\mathrm{(II)}_{3}$, (2.53), (2.54) to obtain the following inequality:
\begin{equation*}
\begin{split}
&m(t)^{\beta-\beta_{0}}e^{\lambda t}\|\omega(t)\|_{Y^{\beta}}\leq C_{\Gamma,\beta-\beta_{0},\lambda_{1}}e^{-(\lambda_{1}-\lambda)t}\|\omega_{0}\|_{Y^{\beta_{0}}} \\
&+C_{\Gamma,\beta+\delta_{2},\lambda_{1}}C_{2}m(t)^{\beta-\beta_{0}}e^{-(\lambda_{1}-\lambda)t}E_{1,\alpha_{2}}(t)E_{2,\beta_{2}}(t)\int^{t}_{0}(t-s)^{-(\beta+\delta_{2})}m(s)^{-(\alpha_{2}-\alpha_{0})-(\beta_{2}-\beta_{0})}e^{-(2\lambda-\lambda_{1})s}ds \\
&+4C_{\Gamma,\beta,\lambda_{1}}C_{6}\mu_{r}m(t)^{\beta-\beta_{0}}e^{-(\lambda_{1}-\lambda)t}E_{2,\beta_{2}}(t)\int^{t}_{0}(t-s)^{-\beta}m(s)^{-(\beta_{2}-\beta_{0})}e^{-(\lambda-\lambda_{1})s}ds \\
&+2C_{\Gamma,\beta+\delta_{2},\lambda_{1}}C_{7}\mu_{r}m(t)^{\beta-\beta_{0}}e^{-(\lambda_{1}-\lambda)t}E_{1,\alpha_{2}}(t)\int^{t}_{0}(t-s)^{-(\beta+\delta_{2})}m(s)^{-(\alpha_{2}-\alpha_{0})}e^{-(\lambda-\lambda_{1})s}ds \\
&+C_{\Gamma,\beta,\lambda_{1}}C_{9}L_{g}m(t)^{\beta-\beta_{0}}e^{-(\lambda_{1}-\lambda)t}E_{3,\gamma_{2}}(t)\int^{t}_{0}(t-s)^{-\beta}m(s)^{-(\gamma_{2}-\gamma_{0})}e^{-(\lambda_{2}-\lambda_{1})s}ds \\
&\leq C_{\Gamma,\beta-\beta_{0},\lambda_{1}}\|\omega_{0}\|_{Y^{\beta_{0}}}+Cm(t)^{1+\alpha_{0}-\alpha_{2}-\beta_{2}-\delta_{2}}e^{-\lambda t}E_{1,\alpha_{2}}(t)E_{2,\beta_{2}}(t)+C\mu_{r}m(t)^{1-\beta_{2}}E_{2,\beta_{2}}(t) \\
&+C\mu_{r}m(t)^{1+\alpha_{0}-\beta_{0}-\alpha_{2}-\delta_{2}}E_{1,\alpha_{2}}(t)+CL_{g}m(t)^{1+\gamma_{0}-\beta_{0}-\gamma_{2}}e^{-(\lambda_{2}-\lambda)t}E_{3,\gamma_{2}}(t),
\end{split}
\end{equation*}
\begin{equation}
E_{2,\beta}(t)\leq C(\|\omega_{0}\|_{Y^{\beta_{0}}}+E_{1,\alpha_{2}}(t)E_{2,\beta_{2}}(t)+\mu_{r}E_{2,\beta_{2}}(t)+\mu_{r}E_{1,\alpha_{2}}(t)+L_{g}E_{3,\gamma_{2}}(t))
\end{equation}
for any $\beta_{0}\leq \beta<1-\delta_{2}$, $t>0$,
\begin{equation*}
\begin{split}
&m(t)^{\gamma-\gamma_{0}}e^{\lambda_{2}t}\|\theta(t)\|_{Z^{\gamma}}\leq C_{B,\gamma-\gamma_{0},\lambda_{1}}e^{-(\lambda_{1}-\lambda_{2})t}\|\theta_{0}\|_{Z^{\gamma_{0}}} \\
&+C_{B,\gamma+\delta_{3},\lambda_{1}}C_{3}m(t)^{\gamma-\gamma_{0}}e^{-(\lambda_{1}-\lambda_{2})t}E_{1,\alpha_{3}}(t)E_{3,\gamma_{3}}(t)\int^{t}_{0}(t-s)^{-(\gamma+\delta_{3})}m(s)^{-(\alpha_{3}-\alpha_{0})-(\gamma_{3}-\gamma_{0})}e^{-(\lambda+\lambda_{2}-\lambda_{1})s}ds \\
&+C_{B,\gamma,\lambda_{1}}C_{4}(1+\mu_{r})m(t)^{\gamma-\gamma_{0}}e^{-(\lambda_{1}-\lambda_{2})t}E_{1,\alpha_{3}}(t)^{2}\int^{t}_{0}(t-s)^{-\gamma}m(s)^{-2(\alpha_{3}-\alpha_{0})}e^{-(2\lambda-\lambda_{1})s}ds \\
&+2C_{B,\gamma,\lambda_{1}}C_{4}(1+\mu_{r})m(t)^{\gamma-\gamma_{0}}e^{-(\lambda_{1}-\lambda_{2})t}E_{1,\alpha_{3}}(t)E_{2,\beta_{3}}(t)\int^{t}_{0}(t-s)^{-\gamma}m(s)^{-(\alpha_{3}-\alpha_{0})-(\beta_{3}-\beta_{0})}e^{-(2\lambda-\lambda_{1})s}ds \\
&+C_{B,\gamma,\lambda_{1}}C_{4}(1+\mu_{r})m(t)^{\gamma-\gamma_{0}}e^{-(\lambda_{1}-\lambda_{2})t}E_{2,\beta_{3}}(t)^{2}\int^{t}_{0}(t-s)^{-\gamma}m(s)^{-2(\beta_{3}-\beta_{0})}e^{-(2\lambda-\lambda_{1})s}ds \\
&\leq C_{B,\gamma-\gamma_{0},\lambda_{1}}\|\theta_{0}\|_{Z^{\gamma_{0}}}+Cm(t)^{1+\alpha_{0}-\alpha_{3}-\gamma_{3}-\delta_{3}}e^{-\lambda t}E_{1,\alpha_{3}}(t)E_{3,\gamma_{3}}(t) \\
&+C(1+\mu_{r})m(t)^{1+2\alpha_{0}-\gamma_{0}-2\alpha_{3}}e^{-(2\lambda-\lambda_{2})t}E_{1,\alpha_{3}}(t)^{2} \\
&+C(1+\mu_{r})m(t)^{1+\alpha_{0}+\beta_{0}-\gamma_{0}-\alpha_{3}-\beta_{3}}e^{-(2\lambda-\lambda_{2})t}E_{1,\alpha_{3}}(t)E_{2,\beta_{3}}(t) \\
&+C(1+\mu_{r})m(t)^{1+2\beta_{0}-\gamma_{0}-2\beta_{3}}e^{-(2\lambda-\lambda_{2})t}E_{2,\beta_{3}}(t)^{2},
\end{split}
\end{equation*}
\begin{equation}
E_{3,\gamma}(t)\leq C\{\|\theta_{0}\|_{Z^{\gamma_{0}}}+E_{1,\alpha_{3}}(t)E_{3,\gamma_{3}}(t)+(1+\mu_{r})E_{1,\alpha_{3}}(t)^{2}+(1+\mu_{r})E_{1,\alpha_{3}}(t)E_{2,\beta_{3}}(t)+(1+\mu_{r})E_{2,\beta_{3}}(t)^{2}\}
\end{equation}
for any $\gamma_{0}\leq \gamma<1-\delta_{3}$, $t>0$.
Set $E(t)=\max\{E_{1,\alpha}(t), E_{2,\beta}(t), E_{3,\gamma}(t) \ ; \ \alpha=\alpha_{1}, \alpha_{2}, \alpha_{3}, \beta=\beta_{1}, \beta_{2}, \beta_{3}, \gamma=\gamma_{1}, \gamma_{2}, \gamma_{3}\}$.
Then (3.47)--(3.49) lead to the following inequality:
\begin{equation*}
E(t)\leq C\{(\|u_{0}\|_{X^{\alpha_{0}}}+\|\omega_{0}\|_{Y^{\beta_{0}}}+\|\theta_{0}\|_{Z^{\gamma_{0}}})+\mu_{r}E(t)+(1+\mu_{r})E(t)^{2}\},
\end{equation*}
\begin{equation}
E(t)\leq C\{(\|u_{0}\|_{X^{\alpha_{0}}}+\|\omega_{0}\|_{Y^{\beta_{0}}}+\|\theta_{0}\|_{Z^{\gamma_{0}}})+E(t)^{2}\}
\end{equation}
for any $t>0$ provided that $\mu_{r}$ is sufficiently small.
An elementary calculation shows that
\begin{equation}
E(t)\leq C(\|u_{0}\|_{X^{\alpha_{0}}}+\|\omega_{0}\|_{Y^{\beta_{0}}}+\|\theta_{0}\|_{Z^{\gamma_{0}}})
\end{equation}
for any $t>0$ provided that $\|u_{0}\|_{X^{\alpha_{0}}}$, $\|\omega_{0}\|_{Y^{\beta_{0}}}$ and $\|\theta_{0}\|_{Z^{\gamma_{0}}}$ are sufficiently small.
Therefore, it is clear from (3.51) that (3.44)--(3.46) are established by (3.47)--(3.49).
\end{proof}

\section{Proof of Theorems 2.3 and 2.4}
We will prove Theorems 2.3 and 2.4 in this section.
Since the proof of Theorem 2.4 is essentially the same as the proof of Theorem 2.3, we have only to prove Theorem 2.3.
Furthermore, in proving Theorem 2.3, we restrict ourselves to the case where $\delta_{1}=0$, $\delta_{2}=0$, $\delta_{3}=0$.
Even if $\delta_{1}>0$ or $\delta_{2}>0$ or $\delta_{3}>0$, it is sufficient for Theorem 2.3 that we slightly modify the argument in this section.

\subsection{$X^{\alpha}\times Y^{\beta}\times Z^{\gamma}$-estimates for integrals}
Theorem 2.3 is proved by the following lemmas:
\begin{lemma}
Let
\begin{equation}
\mathcal{F}(t)=\int^{t}_{0}e^{-(t-s)A}F(s)ds
\end{equation}
with $F \in C((0,T];L^{p}_{\sigma}(\Omega))$ satisfying
\begin{equation}
\|F(t)\|_{p}\leq C_{F}t^{-a}
\end{equation}
for any $0<t<t+h\leq T$, where $C_{F}$ is a positive constant, $0\leq a<1$.
Then
\begin{equation*}
\mathrm{(i)} \ \mathcal{F} \in C^{0,\tilde{\alpha}}((0,T];X^{\alpha})
\end{equation*}
for any $0\leq \alpha<1$, $0<\tilde{\alpha}<1-\alpha$.
\begin{equation}
\mathrm{(ii)} \ \|\mathcal{F}(t+h)-\mathcal{F}(t)\|_{X^{\alpha}}\leq L_{\mathcal{F}}C_{F}(h^{1-\alpha}t^{-a}+h^{\tilde{\alpha}}t^{1-\alpha-\tilde{\alpha}-a})
\end{equation}
for any $0<t<t+h\leq T$, where $L_{\mathcal{F}}=L_{\mathcal{F}}(\alpha,\tilde{\alpha})$ is a positive constant.
\end{lemma}
\begin{proof}
It is \cite[Lemma 3.4]{Hishida}.
\end{proof}
\begin{lemma}
Let $\mathcal{F}$ be a integral given by $(4.1)$ with $F \in C((0,T];L^{p}_{\sigma}(\Omega))$ satisfying $(4.2)$ and
\begin{equation}
\|F(t+h)-F(t)\|_{p}\leq L_{F}h^{b}t^{-c}
\end{equation}
for any $0<t<t+h\leq T$, where $L_{F}$ is positive constant, $0<b\leq 1$, $c>0$.
Then
\begin{equation*}
\mathrm{(i)} \ \mathcal{F} \in C^{0,\hat{\alpha}}((0,T];X^{1}), \ d_{t}\mathcal{F} \in C^{0,\tilde{\alpha}}((0,T];X^{\alpha})
\end{equation*}
for any $0<\hat{\alpha}<b$, $0\leq \alpha<b$, $0<\tilde{\alpha}<b-\alpha$.
\begin{equation*}
\mathrm{(ii)} \ d_{t}\mathcal{F}(t)+A\mathcal{F}(t)=F(t)
\end{equation*}
for any $0<t\leq T$.
\begin{equation}
\mathrm{(iii)} \ \|\mathcal{F}(t)\|_{X^{1}}\leq C_{1,\mathcal{F}}(C_{F}t^{-a}+L_{F}t^{b-c})
\end{equation}
for any $0<t\leq T$, where $C_{1,\mathcal{F}}=C_{1,\mathcal{F}}(b,c)$ is a positive constant.
\begin{equation}
\mathrm{(iv)} \ \|d_{t}\mathcal{F}(t)\|_{X^{\alpha}}\leq C_{2,\mathcal{F}}(C_{F}t^{-(\alpha+a)}+L_{F}t^{b-(\alpha+c)})
\end{equation}
for any $0\leq \alpha< b$, $0<t\leq T$, where $C_{2,\mathcal{F}}=C_{2,\mathcal{F}}(\alpha,b,c)$ is a positive constant.
\end{lemma}
\begin{proof}
It is \cite[Lemmas 2.13 and 2.14]{Fujita}, \cite[Lemma 3.5]{Hishida}.
\end{proof}
We remark that the regularity lemmas similar to Lemmas 4.1 and 4.2 are still valid for $\Gamma$ $(B)$, $G$ $(H)$ and $\mathcal{G}$ $(\mathcal{H})$ instead of $A$, $F$ and $\mathcal{F}$ respectively.

It is useful for the time derivative of strong solutions of (1.1), (1.2) to be stated the following generalized Gronwall lemma:
\begin{lemma}
Let $y$ be a nonnegative continuous and integrable function in $(0,T]$ satisfying
\begin{equation}
y(t)\leq \sum^{l}_{i=1}a_{i}t^{-\alpha_{i}}+\sum^{m}_{j=1}b_{j}\int^{t}_{0}(t-s)^{-\beta_{j}}y(s)ds
\end{equation}
for any $0<t\leq T$, where $a_{i}>0$, $b_{j}>0$, $0\leq \alpha_{i}<1$, $0\leq \beta_{j}<1$.
Then
\begin{equation}
y(t)\leq C\sum^{l}_{i=1}a_{i}t^{-\alpha_{i}}(1+B_{n_{\beta}+1}(t)e^{CB_{n_{\beta}+1}(t)})\sum^{n_{\beta}}_{k=0}B_{k}(t)
\end{equation}
for any $0<t\leq T$, where $C=C(\alpha_{1},\cdots,\alpha_{l},\beta_{1},\cdots,\beta_{m})$ is a positive constant, $n_{\beta}=[\beta/(1-\beta)]+1$, $\beta=\max\{\beta_{j} \ ; \ j=1,\cdots,m\}$,
\begin{equation*}
B_{k}(t)=\left(\sum^{m}_{j=1}b_{j}(t)\right)^{k}, \ b_{j}(t)=b_{j}t^{1-\beta_{j}}.
\end{equation*}
\end{lemma}
\begin{proof}
It is \cite[Remark to Lemma 3.6]{Hishida}.
\end{proof}

\subsection{Regularity of mild solutions}
We will show not only that a mild solution of (1.1), (1.2) can be a strong solution but also that (2.26)--(2.28) are established.
It can be easily seen from Lemmas 2.5--2.13, (2.15)--(2.17) that
\begin{equation}
\|F(u,\omega,\theta)(t)\|_{p}\leq C(t^{2(\alpha_{0}-\alpha_{1})}+t^{\beta_{0}-\beta_{1}}+t^{\gamma_{0}-\gamma_{1}})\times(\|u_{0}\|_{X^{\alpha_{0}}}+\|\omega_{0}\|_{Y^{\beta_{0}}}+\|\theta_{0}\|_{Z^{\gamma_{0}}}),
\end{equation}
\begin{equation}
\|G(u,\omega,\theta)(t)\|_{q}\leq C(t^{\alpha_{0}+\beta_{0}-\alpha_{2}-\beta_{2}}+t^{\beta_{0}-\beta_{2}}+t^{\alpha_{0}-\alpha_{2}}+t^{\gamma_{0}-\gamma_{2}})(\|u_{0}\|_{X^{\alpha_{0}}}+\|\omega_{0}\|_{Y^{\beta_{0}}}+\|\theta_{0}\|_{Z^{\gamma_{0}}}),
\end{equation}
\begin{equation}
\begin{split}
\|H(u,\omega,\theta)(t)\|_{r}\leq& C(t^{\alpha_{0}+\gamma_{0}-\alpha_{3}-\gamma_{3}}+t^{2(\alpha_{0}-\alpha_{3})}+t^{\alpha_{0}+\beta_{0}-\alpha_{3}-\beta_{3}}+t^{2(\beta_{0}-\beta_{3})}) \\
&\times(\|u_{0}\|_{X^{\alpha_{0}}}+\|\omega_{0}\|_{Y^{\beta_{0}}}+\|\theta_{0}\|_{Z^{\gamma_{0}}})
\end{split}
\end{equation}
for any $0<t\leq T$.
Since
\begin{equation*}
u(t+h)-u(t)=(e^{-hA}-I)e^{-tA}u_{0}+\mathcal{F}(u,\omega,\theta)(t+h)-\mathcal{F}(u,\omega,\theta)(t),
\end{equation*}
\begin{equation*}
\omega(t+h)-\omega(t)=(e^{-h\Gamma}-I)e^{-t\Gamma}\omega_{0}+\mathcal{G}(u,\omega,\theta)(t+h)-\mathcal{G}(u,\omega,\theta)(t),
\end{equation*}
\begin{equation*}
\theta(t+h)-\theta(t)=(e^{-hB}-I)e^{-tB}\theta_{0}+\mathcal{H}(u,\omega,\theta)(t+h)-\mathcal{H}(u,\omega,\theta)(t)
\end{equation*}
for any $0<t<t+h\leq T$, it follows from (4.3), (4.9)--(4.11) that
\begin{equation}
\begin{split}
\|u(t+h)-u(t)\|_{X^{\alpha}}\leq& C(h^{b_{1}}t^{\alpha_{0}-\alpha-b_{1}}+h^{1-\alpha}t^{2(\alpha_{0}-\alpha_{1})}+h^{1-\alpha}t^{\beta_{0}-\beta_{1}}+h^{1-\alpha}t^{\gamma_{0}-\gamma_{1}}) \\
&\times(\|u_{0}\|_{X^{\alpha_{0}}}+\|\omega_{0}\|_{Y^{\beta_{0}}}+\|\theta_{0}\|_{Z^{\gamma_{0}}}),
\end{split}
\end{equation}
\begin{equation}
\begin{split}
\|\omega(t+h)-\omega(t)\|_{Y^{\beta}}\leq& C(h^{b_{2}}t^{\beta_{0}-\beta-b_{2}}+h^{1-\beta}t^{\alpha_{0}+\beta_{0}-\alpha_{2}-\beta_{2}}+h^{1-\beta}t^{\beta_{0}-\beta_{2}}+h^{1-\beta}t^{\alpha_{0}-\alpha_{2}} \\
&+h^{1-\beta}t^{\gamma_{0}-\gamma_{2}})(\|u_{0}\|_{X^{\alpha_{0}}}+\|\omega_{0}\|_{Y^{\beta_{0}}}+\|\theta_{0}\|_{Z^{\gamma_{0}}}),
\end{split}
\end{equation}
\begin{equation}
\begin{split}
\|\theta(t+h)-\theta(t)\|_{Z^{\gamma}}\leq& C(h^{b_{3}}t^{\gamma_{0}-\gamma-b_{3}}+h^{1-\gamma}t^{\alpha_{0}+\gamma_{0}-\alpha_{3}-\gamma_{3}}+h^{1-\gamma}t^{2(\alpha_{0}-\alpha_{3})}+h^{1-\gamma}t^{\alpha_{0}+\beta_{0}-\alpha_{3}-\beta_{3}} \\
&+h^{1-\gamma}t^{2(\beta_{0}-\beta_{3})})(\|u_{0}\|_{X^{\alpha_{0}}}+\|\omega_{0}\|_{Y^{\beta_{0}}}+\|\theta_{0}\|_{Z^{\gamma_{0}}})
\end{split}
\end{equation}
for any $0<b_{1}<1-\alpha$, $0<b_{2}<1-\beta$, $0<b_{3}<1-\gamma$, $0<t<t+h\leq T$.
It is derived from (4.11) with $\alpha=\alpha_{1}, \alpha_{2}, \alpha_{3}$, (4.12) with $\beta=\beta_{1}, \beta_{2}, \beta_{3}$, (4.13) with $\gamma=\gamma_{1}, \gamma_{2}, \gamma_{3}$ that
\begin{equation}
F(u,\omega,\theta) \in C^{0,\hat{\alpha}}((0,T];L^{p}_{\sigma}(\Omega)),
\end{equation}
\begin{equation}
G(u,\omega,\theta) \in C^{0,\hat{\beta}}((0,T];(L^{q}(\Omega))^{3}),
\end{equation}
\begin{equation}
H(u,\omega,\theta) \in C^{0,\hat{\gamma}}((0,T];L^{r}(\Omega))
\end{equation}
for any $0<\hat{\alpha}<\min\{1-\alpha_{1}, 1-\beta_{1}, 1-\gamma_{1}\}$, $0<\hat{\beta}<\min\{1-\alpha_{2}, 1-\beta_{2}, 1-\gamma_{2}\}$, $0<\hat{\gamma}<\min\{1-\alpha_{3}, 1-\beta_{3}, 1-\gamma_{3}\}$.
Therefore, Lemma 4.2 (i), (ii) admit that $(u,\omega,\theta)$ is a strong solution of (1.1), (1.2).
By applying (4.5) to (II), we have the following inequalities:
\begin{equation}
t^{1-\alpha_{0}}\|u(t)\|_{X^{1}}\leq CM_{1}(t)(\|u_{0}\|_{X^{\alpha_{0}}}+\|\omega_{0}\|_{Y^{\beta_{0}}}+\|\theta_{0}\|_{Z^{\gamma_{0}}}),
\end{equation}
\begin{equation}
t^{1-\beta_{0}}\|\omega(t)\|_{Y^{1}}\leq CM_{2}(t)(\|u_{0}\|_{X^{\alpha_{0}}}+\|\omega_{0}\|_{Y^{\beta_{0}}}+\|\theta_{0}\|_{Z^{\gamma_{0}}}),
\end{equation}
\begin{equation}
t^{1-\gamma_{0}}\|\theta(t)\|_{Z^{1}}\leq CM_{3}(t)(\|u_{0}\|_{X^{\alpha_{0}}}+\|\omega_{0}\|_{Y^{\beta_{0}}}+\|\theta_{0}\|_{Z^{\gamma_{0}}})
\end{equation}
for any $0<t\leq T$, where
\begin{equation*}
\begin{split}
M_{1}(t):=&1+t^{1+\alpha_{0}-2\alpha_{1}}+t^{1+\beta_{0}-\alpha_{0}-\beta_{1}}+t^{1+\gamma_{0}-\alpha_{0}-\gamma_{1}}+t^{2(1+\alpha_{0}-2\alpha_{1})} \\
&+t^{2+\beta_{0}-2\alpha_{1}-\beta_{1}}+t^{2+\gamma_{0}-2\alpha_{1}-\gamma_{1}}+t^{2+\beta_{0}-\beta_{1}-\alpha_{2}-\beta_{2}}+t^{2+\beta_{0}-\alpha_{0}-\beta_{1}-\beta_{2}} \\
&+t^{2-\beta_{1}-\alpha_{2}}+t^{2+\gamma_{0}-\alpha_{0}-\beta_{1}-\gamma_{2}}+t^{2+\gamma_{0}-\gamma_{1}-\alpha_{3}-\gamma_{3}}+t^{2+\alpha_{0}-\gamma_{1}-2\alpha_{3}} \\
&+t^{2+\beta_{0}-\gamma_{1}-\alpha_{3}-\beta_{3}}+t^{2+2\beta_{0}-\alpha_{0}-\gamma_{1}-2\beta_{3}},
\end{split}
\end{equation*}
\begin{equation*}
\begin{split}
M_{2}(t):=&1+t^{1+\alpha_{0}-\alpha_{2}-\beta_{2}}+t^{1-\beta_{2}}+t^{1+\alpha_{0}-\beta_{0}-\alpha_{2}}+t^{1+\gamma_{0}-\beta_{0}-\gamma_{2}} \\
&+t^{2+2\alpha_{0}-2\alpha_{1}-\alpha_{2}-\beta_{2}}+t^{2+\beta_{0}-\beta_{1}-\alpha_{2}-\beta_{2}}+t^{2+\gamma_{0}-\gamma_{1}-\alpha_{2}-\beta_{2}}+t^{2(1+\alpha_{0}-\alpha_{2}-\beta_{2})} \\
&+t^{2+\alpha_{0}-\alpha_{2}-2\beta_{2}}+t^{2+2\alpha_{0}-\beta_{0}-2\alpha_{2}-\beta_{2}}+t^{2+\alpha_{0}+\gamma_{0}-\beta_{0}-\alpha_{2}-\beta_{2}-\gamma_{2}}+t^{2(1-\beta_{2})} \\
&+t^{2+\alpha_{0}-\beta_{0}-\alpha_{2}-\beta_{2}}+t^{2+\gamma_{0}-\beta_{0}-\beta_{2}-\gamma_{2}}+t^{2+2\alpha_{0}-\beta_{0}-2\alpha_{1}-\alpha_{2}}+t^{2-\beta_{1}-\alpha_{2}} \\
&+t^{2+\gamma_{0}-\beta_{0}-\gamma_{1}-\alpha_{2}}+t^{2+\alpha_{0}+\gamma_{0}-\beta_{0}-\gamma_{2}-\alpha_{3}-\gamma_{3}}+t^{2+2\alpha_{0}-\beta_{0}-\gamma_{2}-2\alpha_{3}} \\
&+t^{2+\alpha_{0}-\gamma_{2}-\alpha_{3}-\beta_{3}}+t^{2+\beta_{0}-\gamma_{2}-2\beta_{3}},
\end{split}
\end{equation*}
\begin{equation*}
\begin{split}
M_{3}(t):=&1+t^{1+\alpha_{0}-\alpha_{3}-\gamma_{3}}+t^{1+2\alpha_{0}-\gamma_{0}-2\alpha_{3}}+t^{1+\alpha_{0}+\beta_{0}-\gamma_{0}-\alpha_{3}-\beta_{3}}+t^{1+2\beta_{0}-\gamma_{0}-2\beta_{3}} \\
&+t^{2+2\alpha_{0}-2\alpha_{1}-\alpha_{3}-\gamma_{3}}+t^{2+\beta_{0}-\beta_{1}-\alpha_{3}-\gamma_{3}}+t^{2+\gamma_{0}-\gamma_{1}-\alpha_{3}-\gamma_{3}}+t^{2(1+\alpha_{0}-\alpha_{3}-\gamma_{3})} \\
&+t^{2+3\alpha_{0}-\gamma_{0}-3\alpha_{3}-\gamma_{3}}+t^{2+2\alpha_{0}+\beta_{0}-\gamma_{0}-2\alpha_{3}-\beta_{3}-\gamma_{3}}+t^{2+\alpha_{0}+2\beta_{0}-\gamma_{0}-\alpha_{3}-2\beta_{3}-\gamma_{3}} \\
&+t^{2+3\alpha_{0}-\gamma_{0}-2\alpha_{1}-2\alpha_{3}}+t^{2+\alpha_{0}+\beta_{0}-\gamma_{0}-\beta_{1}-2\alpha_{3}}+t^{2+\alpha_{0}-\gamma_{1}-2\alpha_{3}} \\
&+t^{2+2\alpha_{0}+\beta_{0}-\gamma_{0}-2\alpha_{1}-\alpha_{3}-\beta_{3}}+t^{2+2\beta_{0}-\gamma_{0}-\beta_{1}-\alpha_{3}-\beta_{3}}+t^{2+\beta_{0}-\gamma_{1}-\alpha_{3}-\beta_{3}} \\
&+t^{2+2\alpha_{0}+\beta_{0}-\gamma_{0}-\alpha_{2}-\beta_{2}-\alpha_{3}-\beta_{3}}+t^{2+\alpha_{0}+\beta_{0}-\gamma_{0}-\beta_{2}-\alpha_{3}-\beta_{3}}+t^{2+2\alpha_{0}-\gamma_{0}-\alpha_{2}-\alpha_{3}-\beta_{3}} \\
&+t^{2+\alpha_{0}-\gamma_{2}-\alpha_{3}-\beta_{3}}+t^{2+\alpha_{0}+2\beta_{0}-\gamma_{0}-\alpha_{2}-\beta_{2}-2\beta_{3}}+t^{2+2\beta_{0}-\gamma_{0}-\beta_{2}-2\beta_{3}} \\
&+t^{2+\alpha_{0}+\beta_{0}-\gamma_{0}-\alpha_{2}-2\beta_{3}}+t^{2+\beta_{0}-\gamma_{2}-2\beta_{3}}.
\end{split}
\end{equation*}
It is clear from (2.48)--(2.51) that (2.26)--(2.28) are established by (4.18)--(4.20).
\begin{remark}
Theorem 2.3 (i) can be, more precisely, stated as follows:
\begin{equation*}
u \in C^{0,\hat{\alpha}}((0,T];X^{1}), \ d_{t}u \in C^{0,\tilde{\alpha}}((0,T];X^{\alpha}),
\end{equation*}
\begin{equation*}
\omega \in C^{0,\hat{\beta}}((0,T];Y^{1}), \ d_{t}\omega \in C^{0,\tilde{\beta}}((0,T];Y^{\beta}),
\end{equation*}
\begin{equation*}
\theta \in C^{0,\hat{\gamma}}((0,T];Z^{1}), \ d_{t}\theta \in C^{0,\tilde{\gamma}}((0,T];Z^{\gamma})
\end{equation*}
for any $0<\hat{\alpha}<\min\{1-\alpha_{1}, 1-\beta_{1}, 1-\gamma_{1}\}$, $0<\hat{\beta}<\min\{1-\alpha_{2}, 1-\beta_{2}, 1-\gamma_{2}\}$, $0<\hat{\gamma}<\min\{1-\alpha_{3}, 1-\beta_{3}, 1-\gamma_{3}\}$, $0\leq\alpha<\min\{1-\alpha_{1}, 1-\beta_{1}, 1-\gamma_{1}\}$, $0\leq\beta<\min\{1-\alpha_{2}, 1-\beta_{2}, 1-\gamma_{2}\}$, $0\leq\gamma<\min\{1-\alpha_{3}, 1-\beta_{3}, 1-\gamma_{3}\}$, $0<\tilde{\alpha}<\min\{1-\alpha_{1}, 1-\beta_{1}, 1-\gamma_{1}\}-\alpha$, $0<\tilde{\beta}<\min\{1-\alpha_{2}, 1-\beta_{2}, 1-\gamma_{2}\}-\beta$, $0<\tilde{\gamma}<\min\{1-\alpha_{3}, 1-\beta_{3}, 1-\gamma_{3}\}-\gamma$.
\end{remark}

\subsection{Regularity of the time derivative of strong solutions}
We will obtain the stronger regularity of strong solutions of (1.1), (1.2) under appropriate assumptions for $p$, $q$, $r$, $\alpha_{0}$, $\beta_{0}$ and $\gamma_{0}$.
Let us remark that $(u,\omega,\theta)$ satisfies (3.39) for any $0<\tau<t<T$.
Then it can be easily seen from (3.39) that
\begin{equation}
\begin{split}
u(t+h)-u(t)=&(e^{-hA}-I)e^{-tA}u(\tau) \\
&+\int^{\tau+h}_{\tau}e^{-(t+h-s)A}F(u,\omega,\theta)(s)ds \\
&+\int^{t}_{\tau}e^{-(t-s)A}(F(u,\omega,\theta)(s+h)-F(u,\omega,\theta)(s))ds,
\end{split}
\end{equation}
\begin{equation}
\begin{split}
\omega(t+h)-\omega(t)=&(e^{-h\Gamma}-I)e^{-t\Gamma}\omega(\tau) \\
&+\int^{\tau+h}_{\tau}e^{-(t+h-s)\Gamma}G(u,\omega,\theta)(s)ds \\
&+\int^{t}_{\tau}e^{-(t-s)\Gamma}(G(u,\omega,\theta)(s+h)-G(u,\omega,\theta)(s))ds,
\end{split}
\end{equation}
\begin{equation}
\begin{split}
\theta(t+h)-\theta(t)=&(e^{-hB}-I)e^{-tB}\theta(\tau) \\
&+\int^{\tau+h}_{\tau}e^{-(t+h-s)B}H(u,\omega,\theta)(s)ds \\
&+\int^{t}_{\tau}e^{-(t-s)B}(H(u,\omega,\theta)(s+h)-H(u,\omega,\theta)(s))ds
\end{split}
\end{equation}
for any $\tau<t<t+h\leq T$.
It follows from (2.1), (2.4), (2.26), (4.9) that
\begin{equation*}
\|(e^{-hA}-I)e^{-tA}u(\tau)\|_{X^{\alpha}}\leq Ch(t-\tau)^{-\alpha}\tau^{\alpha_{0}-1}(\|u_{0}\|_{X^{\alpha_{0}}}+\|\omega_{0}\|_{Y^{\beta_{0}}}+\|\theta_{0}\|_{Z^{\gamma_{0}}}),
\end{equation*}
\begin{equation*}
\begin{split}
&\int^{t+\tau}_{\tau}\|e^{-(t+h-s)A}F(u,\omega,\theta)(s)\|_{X^{\alpha}}ds \\
&\leq C\int^{\tau+h}_{\tau}(t+h-s)^{-\alpha}(s^{2(\alpha_{0}-\alpha_{1})}+s^{\beta_{0}-\beta_{1}}+s^{\gamma_{0}-\gamma_{1}})ds(\|u_{0}\|_{X^{\alpha_{0}}}+\|\omega_{0}\|_{Y^{\beta_{0}}}+\|\theta_{0}\|_{Z^{\gamma_{0}}}) \\
&\leq Ch(t-\tau)^{-\alpha}(\tau^{2(\alpha_{0}-\alpha_{1})}+\tau^{\beta_{0}-\beta_{1}}+\tau^{\gamma_{0}-\gamma_{1}})(\|u_{0}\|_{X^{\alpha_{0}}}+\|\omega_{0}\|_{Y^{\beta_{0}}}+\|\theta_{0}\|_{Z^{\gamma_{0}}})
\end{split}
\end{equation*}
for any $\alpha_{0}\leq \alpha<1$, $\tau<t<t+h\leq T$.
Similarly to $u$ and $F(u,\omega,\theta)$, we can utilize (2.2), (2.3), (2.5), (2.6), (2.27), (2.28), (4.10), (4.11) to obtain that
\begin{equation*}
\|(e^{-h\Gamma}-I)e^{-t\Gamma}\omega(\tau)\|_{Y^{\beta}}\leq Ch(t-\tau)^{-\beta}\tau^{\beta_{0}-1}(\|u_{0}\|_{X^{\alpha_{0}}}+\|\omega_{0}\|_{Y^{\beta_{0}}}+\|\theta_{0}\|_{Z^{\gamma_{0}}}),
\end{equation*}
\begin{equation*}
\begin{split}
&\int^{t+\tau}_{\tau}\|e^{-(t+h-s)\Gamma}G(u,\omega,\theta)(s)\|_{Y^{\beta}}ds \\
&\leq C\int^{\tau+h}_{\tau}(t+h-s)^{-\beta}(s^{\alpha_{0}+\beta_{0}-\alpha_{2}-\beta_{2}}+s^{\beta_{0}-\beta_{2}}+s^{\alpha_{0}-\alpha_{2}}+s^{\gamma_{0}-\gamma_{2}})ds(\|u_{0}\|_{X^{\alpha_{0}}}+\|\omega_{0}\|_{Y^{\beta_{0}}}+\|\theta_{0}\|_{Z^{\gamma_{0}}}) \\
&\leq Ch(t-\tau)^{-\beta}(\tau^{\alpha_{0}+\beta_{0}-\alpha_{2}-\beta_{2}}+\tau^{\beta_{0}-\beta_{2}}+\tau^{\alpha_{0}-\alpha_{2}}+\tau^{\gamma_{0}-\gamma_{2}})(\|u_{0}\|_{X^{\alpha_{0}}}+\|\omega_{0}\|_{Y^{\beta_{0}}}+\|\theta_{0}\|_{Z^{\gamma_{0}}})
\end{split}
\end{equation*}
for any $\beta_{0}\leq \beta<1$, $\tau<t<t+h\leq T$,
\begin{equation*}
\|(e^{-hB}-I)e^{-tB}\theta(\tau)\|_{Z^{\gamma}}\leq Ch(t-\tau)^{-\gamma}\tau^{\gamma_{0}-1}(\|u_{0}\|_{X^{\alpha_{0}}}+\|\omega_{0}\|_{Y^{\beta_{0}}}+\|\theta_{0}\|_{Z^{\gamma_{0}}}),
\end{equation*}
\begin{equation*}
\begin{split}
&\int^{t+\tau}_{\tau}\|e^{-(t+h-s)B}H(u,\omega,\theta)(s)\|_{Z^{\gamma}}ds \\
&\leq C\int^{\tau+h}_{\tau}(t+h-s)^{-\gamma}(s^{\alpha_{0}+\gamma_{0}-\alpha_{3}-\gamma_{3}}+s^{2(\alpha_{0}-\alpha_{3})}+s^{\alpha_{0}+\beta_{0}-\alpha_{3}-\beta_{3}}+s^{2(\beta_{0}-\beta_{3})})ds \\
&\times(\|u_{0}\|_{X^{\alpha_{0}}}+\|\omega_{0}\|_{Y^{\beta_{0}}}+\|\theta_{0}\|_{Z^{\gamma_{0}}}) \\
&\leq Ch(t-\tau)^{-\gamma}(\tau^{\alpha_{0}+\gamma_{0}-\alpha_{3}-\gamma_{3}}+\tau^{2(\alpha_{0}-\alpha_{3})}+\tau^{\alpha_{0}+\beta_{0}-\alpha_{3}-\beta_{3}}+\tau^{2(\beta_{0}-\beta_{3})})(\|u_{0}\|_{X^{\alpha_{0}}}+\|\omega_{0}\|_{Y^{\beta_{0}}}+\|\theta_{0}\|_{Z^{\gamma_{0}}})
\end{split}
\end{equation*}
for any $\gamma_{0}\leq \gamma<1$, $\tau<t<t+h\leq T$.
Therefore, it follows from (4.21)--(4.23) that
\begin{equation}
\begin{split}
\|u(t+h)-u(t)\|_{X^{\alpha}}\leq& C_{1}(\tau)h(t-\tau)^{-\alpha}\tau^{\alpha_{0}-1}(\|u_{0}\|_{X^{\alpha_{0}}}+\|\omega_{0}\|_{Y^{\beta_{0}}}+\|\theta_{0}\|_{Z^{\gamma_{0}}}) \\
&+C_{A,\alpha,\lambda}\int^{t}_{\tau}(t-s)^{-\alpha}\|F(u,\omega,\theta)(s+h)-F(u,\omega,\theta)(s)\|_{p}ds,
\end{split}
\end{equation}
\begin{equation}
\begin{split}
\|\omega(t+h)-\omega(t)\|_{Y^{\beta}}\leq& C_{2}(\tau)h(t-\tau)^{-\beta}\tau^{\beta_{0}-1}(\|u_{0}\|_{X^{\alpha_{0}}}+\|\omega_{0}\|_{Y^{\beta_{0}}}+\|\theta_{0}\|_{Z^{\gamma_{0}}}) \\
&+C_{\Gamma,\beta,\lambda}\int^{t}_{\tau}(t-s)^{-\beta}\|G(u,\omega,\theta)(s+h)-G(u,\omega,\theta)(s)\|_{q}ds,
\end{split}
\end{equation}
\begin{equation}
\begin{split}
\|\theta(t+h)-\theta(t)\|_{Z^{\gamma}}\leq& C_{3}(\tau)h(t-\tau)^{-\gamma}\tau^{\gamma_{0}-1}(\|u_{0}\|_{X^{\alpha_{0}}}+\|\omega_{0}\|_{Y^{\beta_{0}}}+\|\theta_{0}\|_{Z^{\gamma_{0}}}) \\
&+C_{B,\gamma,\lambda}\int^{t}_{\tau}(t-s)^{-\gamma}\|H(u,\omega,\theta)(s+h)-H(u,\omega,\theta)(s)\|_{r}ds
\end{split}
\end{equation}
for any $\tau<t<t+h\leq T$, where
\begin{equation*}
C_{1}(\tau):=C(1+\tau^{1+\alpha_{0}-2\alpha_{1}}+\tau^{1+\beta_{0}-\alpha_{0}-\beta_{1}}+\tau^{1+\gamma_{0}-\alpha_{0}-\gamma_{1}}),
\end{equation*}
\begin{equation*}
C_{2}(\tau):=C(1+\tau^{1+\alpha_{0}-\alpha_{2}-\beta_{2}}+\tau^{1-\beta_{2}}+\tau^{1+\alpha_{0}-\beta_{0}-\alpha_{2}}+\tau^{1+\gamma_{0}-\beta_{0}-\gamma_{2}}),
\end{equation*}
\begin{equation*}
C_{3}(\tau):=C(1+\tau^{1+\alpha_{0}-\alpha_{3}-\gamma_{3}}+\tau^{1+2\alpha_{0}-\gamma_{0}-2\alpha_{3}}+\tau^{1+\alpha_{0}+\beta_{0}-\gamma_{0}-\alpha_{3}-\beta_{3}}+\tau^{1+2\beta_{0}-\gamma_{0}-2\beta_{3}}).
\end{equation*}
Moreover, we can obtain $L^{p}_{\sigma}$-estimates for $F(u,\omega,\theta)(t+h)-F(u,\omega,\theta)(t)$, $(L^{q})^{3}$-estimates for $G(u,\omega,\theta)(t+h)-G(u,\omega,\theta)(t)$ and $L^{r}$-estimates for $H(u,\omega,\theta)(t+h)-H(u,\omega,\theta)(t)$ with the aid of (4.24)--(4.26), consequently,
\begin{equation}
\begin{split}
\|F(u,\omega&,\theta)(t+h)-F(u,\omega,\theta)(t)\|_{p} \\
\leq& Ch\{(t-\tau)^{-\alpha_{1}}\tau^{2\alpha_{0}-\alpha_{1}-1}+(t-\tau)^{-\beta_{1}}\tau^{\beta_{0}-1}+(t-\tau)^{-\gamma_{1}}\tau^{\gamma_{0}-1}\} \\
&\times(\|u_{0}\|_{X^{\alpha_{0}}}+\|\omega_{0}\|_{Y^{\beta_{0}}}+\|\theta_{0}\|_{Z^{\gamma_{0}}}) \\
&+C\tau^{\alpha_{0}-\alpha_{1}}\int^{t}_{\tau}(t-s)^{-\alpha_{1}}\|F(u,\omega,\theta)(s+h)-F(u,\omega,\theta)(s)\|_{p}ds \\
&+C\int^{t}_{\tau}(t-s)^{-\beta_{1}}\|G(u,\omega,\theta)(s+h)-G(u,\omega,\theta)(s)\|_{q}ds \\
&+C\int^{t}_{\tau}(t-s)^{-\gamma_{1}}\|H(u,\omega,\theta)(s+h)-H(u,\omega,\theta)(s)\|_{r}ds,
\end{split}
\end{equation}
\begin{equation}
\begin{split}
\|G(u,\omega&,\theta)(t+h)-G(u,\omega,\theta)(t)\|_{q} \\
\leq& Ch\{(t-\tau)^{-\alpha_{2}}(\tau^{\alpha_{0}+\beta_{0}-\beta_{2}-1}+\tau^{\alpha_{0}-1})+(t-\tau)^{-\beta_{2}}(\tau^{\alpha_{0}+\beta_{0}-\alpha_{2}-1} \\
&+\tau^{\beta_{0}-1})+(t-\tau)^{-\gamma_{2}}\tau^{\gamma_{0}-1}\}(\|u_{0}\|_{X^{\alpha_{0}}}+\|\omega_{0}\|_{Y^{\beta_{0}}}+\|\theta_{0}\|_{Z^{\gamma_{0}}}) \\
&+C(\tau^{\beta_{0}-\beta_{2}}+1)\int^{t}_{\tau}(t-s)^{-\alpha_{2}}\|F(u,\omega,\theta)(s+h)-F(u,\omega,\theta)(s)\|_{p}ds \\
&+C(\tau^{\alpha_{0}-\alpha_{2}}+1)\int^{t}_{\tau}(t-s)^{-\beta_{2}}\|G(u,\omega,\theta)(s+h)-G(u,\omega,\theta)(s)\|_{q}ds \\
&+C\int^{t}_{\tau}(t-s)^{-\gamma_{2}}\|H(u,\omega,\theta)(s+h)-H(u,\omega,\theta)(s)\|_{r}ds,
\end{split}
\end{equation}
\begin{equation}
\begin{split}
\|H(u,\omega&,\theta)(t+h)-H(u,\omega,\theta)(t)\|_{r} \\
\leq& Ch\{(t-\tau)^{-\alpha_{3}}(\tau^{\alpha_{0}+\gamma_{0}-\gamma_{3}-1}+\tau^{2\alpha_{0}-\alpha_{3}-1}+\tau^{\alpha_{0}+\beta_{0}-\beta_{3}-1})+(t-\tau)^{-\beta_{3}}(\tau^{\alpha_{0}+\beta_{0}-\alpha_{3}-1} \\
&+\tau^{2\beta_{0}-\beta_{3}-1})+(t-\tau)^{-\gamma_{3}}\tau^{\alpha_{0}+\gamma_{0}-\alpha_{3}-1}\}(\|u_{0}\|_{X^{\alpha_{0}}}+\|\omega_{0}\|_{Y^{\beta_{0}}}+\|\theta_{0}\|_{Z^{\gamma_{0}}}) \\
&+C(\tau^{\gamma_{0}-\gamma_{3}}+\tau^{\alpha_{0}-\alpha_{3}}+\tau^{\beta_{0}-\beta_{3}})\int^{t}_{\tau}(t-s)^{-\alpha_{3}}\|F(u,\omega,\theta)(s+h)-F(u,\omega,\theta)(s)\|_{p}ds \\
&+C(\tau^{\alpha_{0}-\alpha_{3}}+\tau^{\beta_{0}-\beta_{3}})\int^{t}_{\tau}(t-s)^{-\beta_{3}}\|G(u,\omega,\theta)(s+h)-G(u,\omega,\theta)(s)\|_{q}ds \\
&+C\tau^{\alpha_{0}-\alpha_{3}}\int^{t}_{\tau}(t-s)^{-\gamma_{3}}\|H(u,\omega,\theta)(s+h)-H(u,\omega,\theta)(s)\|_{r}ds
\end{split}
\end{equation}
for any $\tau<t<t+h\leq T$.
Let $p$, $q$, $r$, $\alpha_{0}$, $\beta_{0}$ and $\gamma_{0}$ satisfy (2.24), and set
\begin{equation*}
\begin{split}
y(t)=&\|F(u,\omega,\theta)(t+h)-F(u,\omega,\theta)(t)\|_{p}+\|G(u,\omega,\theta)(t+h)-G(u,\omega,\theta)(t)\|_{q} \\
&+\|H(u,\omega,\theta)(t+h)-H(u,\omega,\theta)(t)\|_{r}.
\end{split}
\end{equation*}
By applying Lemma 4.3 for $(\tau,T-h]$ instead of $(0,T]$ to (4.27)--(4.29) and letting $\tau=t/2$, we have the following inequality:
\begin{equation}
y(t)\leq ChM(t)(\|u_{0}\|_{X^{\alpha_{0}}}+\|\omega_{0}\|_{Y^{\beta_{0}}}+\|\theta_{0}\|_{Z^{\gamma_{0}}})
\end{equation}
for any $0<t\leq T-h$, where
\begin{equation*}
\begin{split}
M(t):=&t^{2(\alpha_{0}-\alpha_{1})-1}+t^{\beta_{0}-\beta_{1}-1}+t^{\gamma_{0}-\gamma_{1}-1}+t^{\alpha_{0}+\beta_{0}-\alpha_{2}-\beta_{2}-1}+t^{\beta_{0}-\beta_{2}-1}+t^{\alpha_{0}-\alpha_{2}-1} \\
&+t^{\gamma_{0}-\gamma_{2}-1}+t^{\alpha_{0}+\gamma_{0}-\alpha_{3}-\gamma_{3}-1}+t^{2(\alpha_{0}-\alpha_{3})-1}+t^{\alpha_{0}+\beta_{0}-\alpha_{3}-\beta_{3}-1}+t^{2(\beta_{0}-\beta_{3})-1}.
\end{split}
\end{equation*}
It is obvious from (4.30) that
\begin{equation*}
F(u,\omega,\theta) \in C^{0,1}((0,T];L^{p}_{\sigma}(\Omega)),
\end{equation*}
\begin{equation*}
G(u,\omega,\theta) \in C^{0,1}((0,T];(L^{q}(\Omega))^{3}),
\end{equation*}
\begin{equation*}
H(u,\omega,\theta) \in C^{0,1}((0,T];L^{r}(\Omega)).
\end{equation*}
Therefore, Lemma 4.2 (i) yields Theorem 2.3 (ii).
Let $p$, $q$, $r$, $\alpha_{0}$, $\beta_{0}$ and $\gamma_{0}$ satisfy (2.25) in addition to (2.24).
By applying (4.6) to (II), it follows from (4.9)--(4.11), (4.30) that
\begin{equation}
t^{1+\alpha-\alpha_{0}}\|d_{t}u(t)\|_{X^{\alpha}}\leq CM_{1}(t)(\|u_{0}\|_{X^{\alpha_{0}}}+\|\omega_{0}\|_{Y^{\beta_{0}}}+\|\theta_{0}\|_{Z^{\gamma_{0}}}),
\end{equation}
\begin{equation}
t^{1+\beta-\beta_{0}}\|d_{t}\omega(t)\|_{Y^{\beta}}\leq CM_{2}(t)(\|u_{0}\|_{X^{\alpha_{0}}}+\|\omega_{0}\|_{Y^{\beta_{0}}}+\|\theta_{0}\|_{Z^{\gamma_{0}}}),
\end{equation}
\begin{equation}
t^{1+\beta-\beta_{0}}\|d_{t}\theta(t)\|_{Z^{\gamma}}\leq CM_{3}(t)(\|u_{0}\|_{X^{\alpha_{0}}}+\|\omega_{0}\|_{Y^{\beta_{0}}}+\|\theta_{0}\|_{Z^{\gamma_{0}}})
\end{equation}
for $0<t\leq T$, where
\begin{equation*}
\begin{split}
M_{1}(t):=&1+t^{1+\alpha_{0}-2\alpha_{1}}+t^{1+\beta_{0}-\alpha_{0}-\beta_{1}}+t^{1+\gamma_{0}-\alpha_{0}-\gamma_{1}}+t^{1+\beta_{0}-\alpha_{2}-\beta_{2}} \\
&+t^{1-\alpha_{2}}+t^{1+\beta_{0}-\alpha_{0}-\beta_{2}}+t^{1+\gamma_{0}-\alpha_{0}-\gamma_{2}}+t^{1+\gamma_{0}-\alpha_{3}-\gamma_{3}} \\
&+t^{1+\alpha_{0}-2\alpha_{3}}+t^{1+\beta_{0}-\alpha_{3}-\beta_{3}}+t^{1+2\beta_{0}-\alpha_{0}-2\beta_{3}},
\end{split}
\end{equation*}
\begin{equation*}
\begin{split}
M_{2}(t):=&1+t^{1+2\alpha_{0}-\beta_{0}-2\alpha_{1}}+t^{1-\beta_{1}}+t^{1+\gamma_{0}-\beta_{0}-\gamma_{1}}+t^{1+\alpha_{0}-\alpha_{2}-\beta_{2}} \\
&+t^{1+\alpha_{0}-\beta_{0}-\alpha_{2}}+t^{1-\beta_{2}}+t^{1+\gamma_{0}-\beta_{0}-\gamma_{2}}+t^{1+\alpha_{0}+\gamma_{0}-\beta_{0}-\alpha_{3}-\gamma_{3}} \\
&+t^{1+2\alpha_{0}-\beta_{0}-2\alpha_{3}}+t^{1+\alpha_{0}-\alpha_{3}-\beta_{3}}+t^{1+\beta_{0}-2\beta_{3}},
\end{split}
\end{equation*}
\begin{equation*}
\begin{split}
M_{3}(t):=&1+t^{1+2\alpha_{0}-\gamma_{0}-2\alpha_{1}}+t^{1+\beta_{0}-\gamma_{0}-\beta_{1}}+t^{1-\gamma_{1}}+t^{1+\alpha_{0}+\beta_{0}-\gamma_{0}-\alpha_{2}-\beta_{2}} \\
&+t^{1+\alpha_{0}-\gamma_{0}-\alpha_{2}}+t^{1+\beta_{0}-\gamma_{0}-\beta_{2}}+t^{1-\gamma_{2}}+t^{1+\alpha_{0}-\alpha_{3}-\gamma_{3}} \\
&+t^{1+2\alpha_{0}-\gamma_{0}-2\alpha_{3}}+t^{1+\alpha_{0}+\beta_{0}-\gamma_{0}-\alpha_{3}-\beta_{3}}+t^{1+2\beta_{0}-\gamma_{0}-2\beta_{3}}.
\end{split}
\end{equation*}
It is clear from (2.48)--(2.51) that (2.29)--(2.31) are established by (4.31)--(4.33).
\begin{remark}
Even if $p$, $q$, $r$, $\alpha_{0}$, $\beta_{0}$ and $\gamma_{0}$ satisfy only (2.13), (2.14), it can be easily seen from (4.15)--(4.17) that (2.29)--(2.31) hold for any $0\leq\alpha<\min\{1-\alpha_{1}, 1-\beta_{1}, 1-\gamma_{1}\}$, $0\leq\beta<\min\{1-\alpha_{2}, 1-\beta_{2}, 1-\gamma_{2}\}$, $0\leq\gamma<\min\{1-\alpha_{3}, 1-\beta_{3}, 1-\gamma_{3}\}$.
\end{remark}

\section{Proof of Corollaries 2.1 and 2.2}
We will prove Corollaries 2.1 and 2.2 in this section.
Since Corollary 2.2 is proved the same as in Corollary 2.1, it is essential for Corollaries 2.1 and 2.2 that we prove Corollary 2.1.

\subsection{$(W^{k,p})^{3}\times (W^{k,q})^{3}\times W^{k,r}$-estimates for linear and nonlinear terms}
We will state and prove some lemmas for $(W^{k,p})^{3}\times (W^{k,q})^{3}\times W^{k,r}$-estimates.
They admit that we obtain $(W^{k,p})^{3}$-estimates for $F(u,\omega,\theta)$, $(W^{k,q})^{3}$-estimates for $G(u,\omega,\theta)$ and $W^{k,r}$-estimates for $H(u,\omega,\theta)$.
\begin{lemma}
\rm{(i)} Let $3<p<\infty$.
Then
\begin{equation}
\|P(u\cdot\nabla)v\|_{p}\leq C\|u\|_{1,p}\|v\|_{1,p}
\end{equation}
for any $u, v \in (W^{1,p}(\Omega))^{3}$, where $C=C(p)$ is a positive constant.

\rm{(ii)} Let $k \in \mathbb{Z}$, $k>3/p$.
Then
\begin{equation}
\|P(u\cdot\nabla)v\|_{k,p}\leq C\|u\|_{k,p}\|v\|_{k+1,p}
\end{equation}
for any $u \in (W^{k,p}(\Omega))^{3}$, $v \in (W^{k+1,p}(\Omega))^{3}$, where $C=C(k,p)$ is a positive constant.
\end{lemma}
\begin{proof}
It is \cite[Lemma 3.3]{Giga 3}.
\end{proof}
\begin{lemma}
\rm{(i)} Let $3<p<\infty$, $1<q<\infty$.
Then
\begin{equation}
\|(u\cdot\nabla)\omega\|_{q}\leq C\|u\|_{1,p}\|\omega\|_{1,q}
\end{equation}
for any $u \in (W^{1,p}(\Omega))^{3}$, $\omega \in (W^{1,q}(\Omega))^{3}$, where $C=C(p,q)$ is a positive constant.

\rm{(ii)} Let $3<p<\infty$, $3<q<\infty$, $q\leq p$, $k \in \mathbb{Z}$, $k>3/q$.
Then
\begin{equation}
\|(u\cdot\nabla)\omega\|_{k,q}\leq C\|u\|_{k,p}\|\omega\|_{k+1,q}
\end{equation}
for any  $u \in (W^{k,p}(\Omega))^{3}$, $\omega \in (W^{k+1,q}(\Omega))^{3}$, where $C=C(k,p,q)$ is a positive constant.
\end{lemma}
\begin{proof}
(i) Let us notice that $W^{1,p}(\Omega)\hookrightarrow C(\overline{\Omega})$ from the Sobolev embedding theorem.
Then we obtain that
\begin{equation*}
\begin{split}
\|(u\cdot\nabla)\omega\|_{q}&\leq C\|u\|_{\infty}\|\omega\|_{1,q} \\
&\leq C\|u\|_{1,p}\|\omega\|_{1,q}.
\end{split}
\end{equation*}

(ii) It is known in \cite[Theorem 4.39]{Adams} that $W^{k,q}(\Omega)$ is a Banach algebra for any $k \in \mathbb{Z}$, $k>3/q$.
Therefore, the conclusion follows immediately from the above fact and $q\leq p$.
\end{proof}
\begin{lemma}
\rm{(i)} Let $3<p<\infty$, $1<r<\infty$.
Then
\begin{equation}
\|(u\cdot\nabla)\theta\|_{r}\leq C\|u\|_{1,p}\|\theta\|_{1,r}
\end{equation}
for any $u \in (W^{1,p}(\Omega))^{3}$, $\theta \in W^{1,r}(\Omega)$, where $C=C(p,r)$ is a positive constant.

\rm{(ii)} Let $3<p<\infty$, $3<r<\infty$, $r\leq p$, $k \in \mathbb{Z}$, $k>3/r$.
Then
\begin{equation}
\|(u\cdot\nabla)\theta\|_{k,r}\leq C\|u\|_{k,p}\|\theta\|_{k+1,r}
\end{equation}
for any  $u \in (W^{k,p}(\Omega))^{3}$, $\theta \in W^{k+1,r}(\Omega)$, where $C=C(k,p,r)$ is a positive constant.
\end{lemma}
\begin{proof}
It is \cite[Lemma 5.2]{Kakizawa}.
\end{proof}
\begin{lemma}
\rm{(i)} Let $3<p<\infty$, $3<q<\infty$, $1<r<\infty$, $2r\leq p$, $2r\leq q$.
Then
\begin{equation}
\|\Phi(u,v;\omega,\psi)\|_{r}\leq C(1+\mu_{r})(\|u\|_{1,p}\|v\|_{1,p}+\|u\|_{1,p}\|\psi\|_{1,q}+\|v\|_{1,p}\|\omega\|_{1,q}+\|\omega\|_{1,q}\|\psi\|_{1,q})
\end{equation}
for any $u, v \in (W^{1,p}(\Omega))^{3}$, $\omega, \psi \in (W^{1,q}(\Omega))^{3}$ where $C=C(p,q,r)$ is a positive constant.

\rm{(ii)} Let $3<p<\infty$, $3<q<\infty$, $1<r<\infty$, $2r\leq p$, $2r\leq q$, $k \in \mathbb{Z}$, $k>3/p$, $k>3/q$.
Then
\begin{equation}
\begin{split}
\|\Phi(u,v;\omega,\psi)\|_{k,r}\leq& C(1+\mu_{r})(\|u\|_{k+1,p}\|v\|_{k+1,p}+\|u\|_{k+1,p}\|\psi\|_{k+1,q} \\
&+\|v\|_{k+1,p}\|\omega\|_{k+1,q}+\|\omega\|_{k+1,q}\|\psi\|_{k+1,q})
\end{split}
\end{equation}
for any $u, v \in (W^{k+1,p}(\Omega))^{3}$, $\omega, \psi \in (W^{k+1,q}(\Omega))^{3}$, where $C=C(k,p,q,r)$ is a positive constant.
\end{lemma}
\begin{proof}
(i) After applying the Schwarz inequality to $\|\Phi(u,v;\omega,\psi)\|_{r}$, we can obtain (5.7) by $W^{1,p}(\Omega)\hookrightarrow W^{1,2r}(\Omega)$, $W^{1,q}(\Omega)\hookrightarrow W^{1,2r}(\Omega)$.

(ii) It can be easily seen from the Leibniz rule and the Schwarz inequality that
\begin{equation*}
\begin{split}
\|\Phi(u,v;\omega,\psi)\|_{k,q}\leq& C(1+\mu_{r})(\|u\|_{k+1,2r}\|v\|_{k+1,2r}+\|u\|_{k+1,2r}\|\psi\|_{k+1,2r} \\
&+\|v\|_{k+1,2r}\|\omega\|_{k+1,2r}+\|\omega\|_{k+1,2r}\|\psi\|_{k+1,2r}).
\end{split}
\end{equation*}
Therefore, $W^{k+1,p}(\Omega)\hookrightarrow W^{k+1,2r}(\Omega)$, $W^{k+1,q}(\Omega)\hookrightarrow W^{k+1,2r}(\Omega)$ imply (5.8).
\end{proof}
\begin{lemma}
\rm{(i)} Let $1<p<\infty$, $1<q<\infty$, $p\leq q$.
Then
\begin{equation}
\|P\mathrm{rot}\omega\|_{p}\leq C\|\omega\|_{1,q}
\end{equation}
for any $\omega \in (W^{1,q}(\Omega))^{3}$, where $C=C(p,q)$ is a positive constant.

\rm{(ii)} Let $1<p<\infty$, $1<q<\infty$, $p\leq q$, $k \in \mathbb{Z}$, $k\geq 1$.
Then
\begin{equation}
\|P\mathrm{rot}\omega\|_{k,p}\leq C\|\omega\|_{k+1,q}
\end{equation}
for any $\omega \in (W^{k+1,q}(\Omega))^{3}$, where $C=C(k,p,q)$ is a positive constant.
\end{lemma}
\begin{proof}
(i) Since $P$ is a bounded operator in $(L^{p}(\Omega))^{3}$ and $\|\mathrm{rot}\omega\|_{p}\leq C\|\omega\|_{1,p}$, it follows from $W^{1,q}(\Omega)\hookrightarrow W^{1,p}(\Omega)$ that we obtain (5.9).

(ii) It is known in \cite[Lemma 3.3]{Giga 3} that $P$ is a bounded operator not only in $(L^{p}(\Omega))^{3}$ but also in $(W^{k,p}(\Omega))^{3}$.
Since $\|\mathrm{rot}\omega\|_{k,p}\leq C\|\omega\|_{k+1,p}$, (5.10) follows immediately from $W^{k+1,q}(\Omega)\hookrightarrow W^{k+1,p}(\Omega)$.
\end{proof}
\begin{lemma}
\rm{(i)} Let $1<p<\infty$, $1<q<\infty$, $q\leq p$.
Then
\begin{equation}
\|\mathrm{rot}u\|_{q}\leq C\|u\|_{1,p}
\end{equation}
for any $u \in (W^{1,p}(\Omega))^{3}$, where $C=C(p,q)$ is a positive constant.

\rm{(ii)} Let $1<p<\infty$, $1<q<\infty$, $q\leq p$, $k \in \mathbb{Z}$, $k\geq 1$.
Then
\begin{equation}
\|\mathrm{rot}u\|_{k,q}\leq C\|u\|_{k+1,p}
\end{equation}
for any $\omega \in (W^{k+1,q}(\Omega))^{3}$, where $C=C(k,p,q)$ is a positive constant.
\end{lemma}
\begin{proof}
(i) Since $\|\mathrm{rot}u\|_{q}\leq C\|u\|_{1,q}$, it follows from $W^{1,p}(\Omega)\hookrightarrow W^{1,q}(\Omega)$ that we obtain (5.11).

(ii) Since $\|\mathrm{rot}u\|_{k,q}\leq C\|u\|_{k+1,q}$, (5.12) follows immediately from $W^{k+1,p}(\Omega)\hookrightarrow W^{k+1,q}(\Omega)$.
\end{proof}
\begin{lemma}
\rm{(i)} Let $f \in C^{0,1}(\mathbb{R};\mathbb{R}^{3})$ with the Lipschitz constant $L_{f}$, $f(0)=0$, $3<p<\infty$, $3<r<\infty$, $r\leq p$.
Then
\begin{equation}
\|Pf(\theta)\|_{p}\leq CL_{f}\|\theta\|_{1,r}
\end{equation}
for any $\theta \in W^{1,r}(\Omega)$, where $C=C(p,r)$ is a positive constant.

\rm{(ii)} Let $f \in C^{0,1}(\mathbb{R};\mathbb{R}^{3})\cap C^{1}(\mathbb{R};\mathbb{R}^{3})$ with the Lipschitz constant $L_{f}$, $f(0)=0$, $3<p<\infty$, $3<r<\infty$, $r\leq p$.
Then
\begin{equation}
\|Pf(\theta)\|_{1,p}\leq C\|\theta\|_{2,r}
\end{equation}
for any $\theta \in W^{2,r}(\Omega)$, where $C=C(p,r)$ is a positive constant.
\end{lemma}
\begin{proof}
(i) Since $P$ is a bounded operator in $(L^{p}(\Omega))^{3}$ and $\|f(\theta)\|_{p}\leq L_{f}\|\theta\|_{p}$, it follows from $W^{1,r}(\Omega)\hookrightarrow L^{p}(\Omega)$ that we obtain (5.13).

(ii) It is known in \cite[Lemma 3.3]{Giga 3} that $P$ is a bounded operator not only in $(L^{p}(\Omega))^{3}$ but also in $(W^{1,p}(\Omega))^{3}$.
Since $f \in C^{0,1}(\mathbb{R};\mathbb{R}^{3})\cap C^{1}(\mathbb{R};\mathbb{R}^{3})$ implies $f' \in C_{b}(\mathbb{R};\mathbb{R}^{3})$, $\|f(\theta)\|_{1,p}\leq C\|\theta\|_{1,p}$.
Therefore, (5.14) follows immediately from $W^{2,r}(\Omega)\hookrightarrow W^{1,p}(\Omega)$.
\end{proof}
\begin{lemma}
\rm{(i)} Let $g \in C^{0,1}(\mathbb{R};\mathbb{R}^{3})$ with the Lipschitz constant $L_{g}$, $g(0)=0$, $3<q<\infty$, $3<r<\infty$, $r\leq q$.
Then
\begin{equation}
\|g(\theta)\|_{q}\leq CL_{g}\|\theta\|_{1,r}
\end{equation}
for any $\theta \in W^{1,r}(\Omega)$, where $C=C(q,r)$ is a positive constant.

\rm{(ii)} Let $g \in C^{0,1}(\mathbb{R};\mathbb{R}^{3})\cap C^{1}(\mathbb{R};\mathbb{R}^{3})$ with the Lipschitz constant $L_{g}$, $g(0)=0$, $3<q<\infty$, $3<r<\infty$, $r\leq q$.
Then
\begin{equation}
\|g(\theta)\|_{1,q}\leq C\|\theta\|_{2,r}
\end{equation}
for any $\theta \in W^{2,r}(\Omega)$, where $C=C(q,r)$ is a positive constant.
\end{lemma}
\begin{proof}
(i) Since $\|g(\theta)\|_{q}\leq L_{g}\|\theta\|_{q}$, it follows from $W^{1,r}(\Omega)\hookrightarrow L^{q}(\Omega)$ that we obtain (5.15).

(ii) Since $g \in C^{0,1}(\mathbb{R};\mathbb{R}^{3})\cap C^{1}(\mathbb{R};\mathbb{R}^{3})$ implies $g' \in C_{b}(\mathbb{R};\mathbb{R}^{3})$, $\|g(\theta)\|_{1,q}\leq C\|\theta\|_{1,q}$.
Therefore, (5.16) follows immediately from $W^{2,r}(\Omega)\hookrightarrow W^{1,q}(\Omega)$.
\end{proof}

\subsection{Regularity of strong solutions}
It is essential for Corollary 2.1 that we prove the following lemmas:
\begin{lemma}
Let $f \in C^{0,1}(\mathbb{R};\mathbb{R}^{3})$, $f(0)=0$, $g \in C^{0,1}(\mathbb{R};\mathbb{R}^{3})$, $g(0)=0$, $p$, $q=p$ and $r$ satisfy $(2.38)$.
Then
\begin{equation*}
F(u,\omega,\theta) \in C^{0,\hat{\alpha}}((0,T];(W^{1,p}(\Omega))^{3}),
\end{equation*}
\begin{equation*}
G(u,\omega,\theta) \in C^{0,\hat{\beta}}((0,T];(W^{1,p}(\Omega))^{3}),
\end{equation*}
\begin{equation*}
H(u,\omega,\theta) \in C^{0,\hat{\gamma}}((0,T];W^{1,r}(\Omega))
\end{equation*}
for any $0<\hat{\alpha}<1$, $0<\hat{\beta}<1$, $0<\hat{\gamma}<1$.
\end{lemma}
\begin{proof}
It follows immediately from Theorem 2.3 (ii), Lemmas 5.1--5.8 (ii) with $k=1$.
\end{proof}
\begin{lemma}
Let $f \in C^{0,1}(\mathbb{R};\mathbb{R}^{3})$, $f(0)=0$, $g \in C^{0,1}(\mathbb{R};\mathbb{R}^{3})$, $g(0)=0$, $p$, $q=p$ and $r$ satisfy $(2.38)$.
Then
\begin{equation*}
u \in C^{0,\hat{\alpha}}((0,T];(W^{3,p}(\Omega))^{3}),
\end{equation*}
\begin{equation*}
\omega \in C^{0,\hat{\beta}}((0,T];(W^{3,p}(\Omega))^{3}),
\end{equation*}
\begin{equation*}
\theta \in C^{0,\hat{\gamma}}((0,T];W^{3,r}(\Omega))
\end{equation*}
for any $0<\hat{\alpha}<1/2$, $0<\hat{\beta}<1/2$, $0<\hat{\gamma}<1/2$.
\end{lemma}
\begin{proof}
It is clear from (2.10)--(2.12) that $X^{1/2}\hookrightarrow (W^{1,p}(\Omega))^{3}$, $Y^{1/2}\hookrightarrow (W^{1,p}(\Omega))^{3}$, $Z^{1/2}\hookrightarrow W^{1,r}(\Omega)$.
By applying Theorem 2.3 (ii) with $\alpha=1/2$, $\beta=1/2$, $\gamma=1/2$ to $(d_{t}u,d_{t}\omega,d_{t}\theta)$, the conclusion follows immediately from Lemma 5.9, $u=A^{-1}(F(u,\omega,\theta)-d_{t}u)$, $\omega=\Gamma^{-1}(G(u,\omega,\theta)-d_{t}\omega)$, $\theta=B^{-1}(H(u,\omega,\theta)-d_{t}\theta)$.
\end{proof}
\begin{lemma}
Let $f \in C^{0,1}(\mathbb{R};\mathbb{R}^{3})\cap C^{1}(\mathbb{R};\mathbb{R}^{3})$, $f(0)=0$, $g \in C^{0,1}(\mathbb{R};\mathbb{R}^{3})\cap C^{1}(\mathbb{R};\mathbb{R}^{3})$, $g(0)=0$, $p$, $q=p$ and $r$ satisfy $(2.38)$.
Then
\begin{equation*}
d_{t}F(u,\omega,\theta) \in C^{0,\hat{\alpha}}((0,T];L^{p}_{\sigma}(\Omega)),
\end{equation*}
\begin{equation*}
d_{t}G(u,\omega,\theta) \in C^{0,\hat{\beta}}((0,T];(L^{p}(\Omega))^{3}),
\end{equation*}
\begin{equation*}
d_{t}H(u,\omega,\theta) \in C^{0,\hat{\gamma}}((0,T];L^{r}(\Omega))
\end{equation*}
for any $0<\hat{\alpha}<\min\{1-\alpha_{1}, 1-\beta_{1}, 1-\gamma_{1}\}$, $0<\hat{\beta}<\min\{1-\alpha_{2}, 1-\beta_{2}, 1-\gamma_{2}\}$, $0<\hat{\gamma}<\min\{1-\alpha_{3}, 1-\beta_{3}, 1-\gamma_{3}\}$.
\end{lemma}
\begin{proof}
It follows immediately from Theorem 2.3 (ii) with $\alpha=\alpha_{1},\alpha_{2}, \alpha_{3}$, $\beta=\beta_{1}, \beta_{2},\beta_{3}$, $\gamma=\gamma_{1},\gamma_{2}, \gamma_{3}$, Lemmas 5.1--5.8 (i).
\end{proof}
\begin{lemma}
Let $f \in C^{0,1}(\mathbb{R};\mathbb{R}^{3})\cap C^{1}(\mathbb{R};\mathbb{R}^{3})$, $f(0)=0$, $g \in C^{0,1}(\mathbb{R};\mathbb{R}^{3})\cap C^{1}(\mathbb{R};\mathbb{R}^{3})$, $g(0)=0$, $p$, $q=p$ and $r$ satisfy $(2.38)$.
Then
\begin{equation*}
d_{t}u \in C^{0,\hat{\alpha}}((0,T];X^{1}), \ d^{2}_{t}u \in C^{0,\tilde{\alpha}}((0,T];X^{\alpha}),
\end{equation*}
\begin{equation*}
d_{t}\omega \in C^{0,\hat{\beta}}((0,T];Y^{1}), \ d^{2}_{t}\omega \in C^{0,\tilde{\beta}}((0,T];Y^{\beta}),
\end{equation*}
\begin{equation*}
d_{t}\theta \in C^{0,\hat{\gamma}}((0,T];Z^{1}), \ d^{2}_{t}\theta \in C^{0,\tilde{\gamma}}((0,T];Z^{\gamma})
\end{equation*}
for any $0<\hat{\alpha}<\min\{1-\alpha_{1}, 1-\beta_{1}, 1-\gamma_{1}\}$, $0<\hat{\beta}<\min\{1-\alpha_{2}, 1-\beta_{2}, 1-\gamma_{2}\}$,  $0<\hat{\gamma}<\min\{1-\alpha_{3}, 1-\beta_{3}, 1-\gamma_{3}\}$, $0\leq\alpha<\min\{1-\alpha_{1}, 1-\beta_{1}, 1-\gamma_{1}\}$, $0\leq\beta<\min\{1-\alpha_{2}, 1-\beta_{2}, 1-\gamma_{2}\}$, $0\leq\gamma<\min\{1-\alpha_{3}, 1-\beta_{3}, 1-\gamma_{3}\}$, $0<\tilde{\alpha}<\min\{1-\alpha_{1}, 1-\beta_{1}, 1-\gamma_{1}\}-\alpha$, $0<\tilde{\beta}<\min\{1-\alpha_{2}, 1-\beta_{2}, 1-\gamma_{2}\}-\beta$, $0<\tilde{\gamma}<\min\{1-\alpha_{3}, 1-\beta_{3}, 1-\gamma_{3}\}-\gamma$.
\end{lemma}
\begin{proof}
Lemma 5.11 admits that we differentiate (I) with respect to $t$ and obtain the following abstract integral equations:
\begin{equation*}
\begin{split}
&d_{t}u(t)=e^{-(t-\tau)A}d_{t}u(\tau)+\displaystyle\int^{t}_{\tau}e^{-(t-s)A}d_{t}F(u,\omega,\theta)(s)ds, \\
&d_{t}\omega(t)=e^{-(t-\tau)\Gamma}d_{t}\omega(\tau)+\displaystyle\int^{t}_{\tau}e^{-(t-s)\Gamma}d_{t}G(u,\omega,\theta)(s)ds, \\
&d_{t}\theta(t)=e^{-(t-\tau)B}d_{t}\theta(\tau)+\displaystyle\int^{t}_{\tau}e^{-(t-s)B}d_{t}H(u,\omega,\theta)(s)ds
\end{split}
\end{equation*}
for any $0<\tau\leq t\leq T$.
Therefore, the conclusion follows immediately from Lemmas 4.2 (\cite[Lemma 2.14]{Fujita}) and 5.11.
\end{proof}
It follows from the Sobolev embedding theorem that $W^{k+1,p}(\Omega)\hookrightarrow C^{k,\alpha}(\overline{\Omega})$, $W^{k+1,p}(\Omega)$ $\hookrightarrow C^{k,\beta}(\overline{\Omega})$, $W^{k+1,r}(\Omega)\hookrightarrow C^{k,\gamma}(\overline{\Omega})$ for any $k \in \mathbb{Z}$, $k\geq 0$, $0<\alpha<1-3/p$, $0<\beta<1-3/p$, $0<\gamma<1-3/r$.
Therefore, Lemmas 5.10 and 5.12 yield Corollary 2.1.


\end{document}